\documentclass[12pt]{article}

\usepackage{times}

\usepackage{enumitem}
\usepackage[utf8]{inputenc}
\usepackage{commath}
\usepackage{amsthm, amscd}
\usepackage{amsmath}
\usepackage{amssymb}
\usepackage{cite}
\usepackage{setspace}
\usepackage{ stmaryrd }
\usepackage{tikz-cd}

\newcommand*\tasklabelformat[1]{#1)}

\usepackage{tasks}[newest]
\settasks{
  label = \alph* ,
  label-format = \tasklabelformat ,
  label-width  = 12pt
}

\usepackage{bigints}
\usepackage{comment}
\usepackage{mathtools}
\usepackage{mathrsfs}
\usepackage{fancyhdr}

\allowdisplaybreaks[4]

\numberwithin{equation}{section}
\usepackage{etoolbox}
\patchcmd{\thebibliography}
  {\settowidth}
  {\setlength{\itemsep}{0pt plus -10pt}\settowidth}
  {}{}
\apptocmd{\thebibliography}
  {
  }
  {}{}
\makeatletter
\newtheorem*{rep@theorem}{\rep@title}
\newcommand{\newreptheorem}[2]{%
\newenvironment{rep#1}[1]{%
 \def\rep@title{#2 \ref{##1}}%
 \begin{rep@theorem}}%
 {\end{rep@theorem}}}
\makeatother

\theoremstyle{theorem}

\newreptheorem{theorem}{Theorem}
\newtheorem{thm}{Theorem}[section]
\newtheorem*{thm*}{Theorem}
\theoremstyle{definition}
\newtheorem{prop}[thm]{Proposition}
\newtheorem*{prop*}{Proposition}
\newtheorem{defn}[thm]{Definition}
\newtheorem{lem}[thm]{Lemma}
\newtheorem{cor}[thm]{Corollary}
\newtheorem*{cor*}{Corollary}
\theoremstyle{remark}

\newtheorem{rem}[thm]{Remark}

\title{\vspace*{-1cm} On the metric structure of section ring}

\author
{Siarhei Finski
}

\date{}

\usepackage[%
    left=1in,%
    right=1in,%
    top=1.1in,%
    bottom=0.8in,%
    paperheight=11in,%
    paperwidth=8.5in%
]{geometry}

\newcommand{\imun} {\sqrt{-1}}

\newcommand{\res}{{\rm{Res}}}

\newcommand{\comp}{\mathbb{C}}
\newcommand{\real}{\mathbb{R}}

\newcommand{\nat}{\mathbb{N}}

\newcommand{\dist}{{\rm{dist}}}

\newcommand{\enmr}[1]{\text{End}{(#1)}}

\newcommand{\ccal}{\mathscr{C}}

\newcommand{\dbar}{ \overline{\partial} }

\renewcommand{\Re}{\operatorname{Re}}
\renewcommand{\Im}{\operatorname{Im}}
\newcommand{\scal}[2]{\langle #1, #2 \rangle}

\makeatletter
\DeclareFontFamily{OMX}{MnSymbolE}{}
\DeclareSymbolFont{MnLargeSymbols}{OMX}{MnSymbolE}{m}{n}
\SetSymbolFont{MnLargeSymbols}{bold}{OMX}{MnSymbolE}{b}{n}
\DeclareFontShape{OMX}{MnSymbolE}{m}{n}{
    <-6>  MnSymbolE5
   <6-7>  MnSymbolE6
   <7-8>  MnSymbolE7
   <8-9>  MnSymbolE8
   <9-10> MnSymbolE9
  <10-12> MnSymbolE10
  <12->   MnSymbolE12
}{}
\DeclareFontShape{OMX}{MnSymbolE}{b}{n}{
    <-6>  MnSymbolE-Bold5
   <6-7>  MnSymbolE-Bold6
   <7-8>  MnSymbolE-Bold7
   <8-9>  MnSymbolE-Bold8
   <9-10> MnSymbolE-Bold9
  <10-12> MnSymbolE-Bold10
  <12->   MnSymbolE-Bold12
}{}

\let\llangle\@undefined
\let\rrangle\@undefined
\DeclareMathDelimiter{\llangle}{\mathopen}%
                     {MnLargeSymbols}{'164}{MnLargeSymbols}{'164}
\DeclareMathDelimiter{\rrangle}{\mathclose}%
                     {MnLargeSymbols}{'171}{MnLargeSymbols}{'171}
\makeatother

\newenvironment{sciabstract}{}

\begin{document} 

\maketitle

\vspace*{-0.6cm}

{\centering \small \textit{In memory of Jean-Pierre Demailly}\par}

\vspace*{0.3cm}

\begin{sciabstract}
  \textbf{Abstract.} The main goal of this article is to study for a projective manifold and an ample line bundle over it the relation between metric and algebraic structures on the associated section ring.
  \par 
  More precisely, we prove that once the kernel is factored out, the multiplication operator of the section ring becomes an approximate isometry (up to a normalization) with respect to the $L^2$-norms and the induced Hermitian tensor product norm. 
  We also show that the analogous result holds for the $L^1$ and $L^{\infty}$-norms if instead of the Hermitian tensor product norm, we consider the projective and injective tensor norms induced by $L^1$ and $L^{\infty}$-norms respectively.
  \par 
 Then we show that $L^2$-norms associated with continuous plurisubharmonic metrics are actually characterized by the multiplicativity properties of this type.
 Using this, we refine the theorem of Phong-Sturm about quantization of Mabuchi geodesics from the weaker level of Fubini-Study convergence to the stronger level of norm equivalences.
  
\end{sciabstract}

\pagestyle{fancy}
\lhead{}
\chead{On the metric structure of section ring}
\rhead{\thepage}
\cfoot{}


\newcommand{\Addresses}{{
  \bigskip
  \footnotesize
  \noindent \textsc{Siarhei Finski, CNRS-CMLS, École Polytechnique F-91128 Palaiseau Cedex, France.}\par\nopagebreak
  \noindent  \textit{E-mail }: \texttt{finski.siarhei@gmail.com}.
}} 

\vspace*{-0.6cm}

\tableofcontents

\section{Introduction}\label{sect_intro}
	Many questions from algebraic and complex geometry can be restated in terms of some problems related to the study of the section ring.
	For example, finite generation of the canonical ring is of great importance to the minimal model program, \cite{BirkarCasciniHaconMcKernan}.
	The K-stability condition, famously known for its connection with the existence of the constant scalar curvature metrics, is related to the asymptotic study of certain filtrations on the section ring \cite{NystromTest}, \cite{SzekeTestConf}, \cite{BouckJohn21}.
	\par 
	In this article, we propose to study further properties of the section ring; now in connection with the metric structure of the manifold.
	\par 
	More precisely, we fix a complex projective manifold $X$ of dimension $n$ and an ample line bundle $L$ over $X$.
	We consider the \textit{section ring} $R(X, L)$, defined as follows
	\begin{equation}
		R(X, L) := \oplus_{k = 1}^{\infty} H^0(X, L^k).
	\end{equation}
	For any $r \in \nat^*$, $k; k_1, \ldots, k_r \in \nat^*$, $k_1 + \cdots + k_r = k$, we define the multiplication map
	\begin{equation}\label{eq_mult_map}
		{\rm{Mult}}_{k_1, \cdots, k_r} : H^0(X, L^{k_1}) \otimes \cdots \otimes H^0(X, L^{k_r}) 
		\to
		H^0(X, L^k),
	\end{equation}	
	as $f_1 \otimes \cdots \otimes f_r \mapsto f_1 \cdots f_r$.
	Those multiplication maps endow $R(X, L)$ with the structure of a graded ring.
	It follows from the ${\rm{Proj}}$-construction that the (graded) ring structure of $R(X, L)$ carries the same amount of information as the pair $(X, L)$.
	\par 
	We now fix a positive Hermitian metric $h^L$ on $L$.
	This means that for the curvature $R^L$ of the Chern connection on $(L, h^L)$, the following closed real $(1, 1)$-differential form is positive
	\begin{equation}\label{eq_omega}
		\omega := c_1(L, h^L) := \frac{\imun}{2 \pi} R^L.
	\end{equation}
	\par 
	For smooth sections $f, f'$ of $L^p$, $p \in \nat$, we define the $L^2$-Hermitian product using the pointwise Hermitian products $\langle \cdot, \cdot \rangle_h$ on $L^p$, induced by $h^L$, as follows
	\begin{equation}\label{eq_l2_prod}
		\scal{f}{f'}_{L^2_p(X, h^L)} := \int_X \scal{f(x)}{f'(x)}_h dv_X(x),
	\end{equation}
	where $dv_X = \frac{1}{n!} \omega^n$.
	We also denote by $\| \cdot \|_{L^2_p(X, h^L)}$ the associated norm.
	We, finally, denote by $\| \cdot \|_{L^1_p(X, h^L)}$ and $\| \cdot \|_{L^{\infty}_p(X, h^L)}$ the $L^1$ and $L^{\infty}$-norms, defined as follows
	\begin{equation}
		\| f \|_{L^1_p(X, h^L)} := \int_X | f |_h dv_X(x), 
		\qquad 
		\| f \|_{L^{\infty}_p(X, h^L)} := \sup_{x \in X} |f(x)|_h.
	\end{equation}
	When restricted to the vector space of holomorphic sections, $H^0(X, L^p)$, we sometimes denote the $L^2$-norm by ${\rm{Hilb}}_p(h^L)$, and the $L^1$ and $L^{\infty}$-norms by ${\rm{Ban}}_p^{1}(h^L)$, ${\rm{Ban}}_p^{\infty}(h^L)$.
	\par
	\begin{sloppypar} 
	Over $R(X, L)$, for $q = 1, \infty$, we define the induced \textit{graded} norms
	\begin{equation}\label{defn_hilb_grad}
		{\rm{Hilb}}(h^L) := \sum_{k = 1}^{\infty} {\rm{Hilb}}_k(h^L),
		\qquad
		{\rm{Ban}}^{q}(h^L) := \sum_{k = 1}^{\infty} {\rm{Ban}}_k^{q}(h^L).
	\end{equation}
	The first goal of this paper is to study the metric properties of $(R(X, L), {\rm{Ban}}^1(h^L))$, $(R(X, L), {\rm{Hilb}}(h^L))$ and $(R(X, L), {\rm{Ban}}^{\infty}(h^L))$ in their relation with the multiplication map.
	Roughly, our first result states that one can inductively recover the approximate metric structure of the section ring from its multiplicative structure.
	\end{sloppypar} 
	\par 
	To explain this more precisely, let us recall some basic definitions from the theory of normed vector spaces.
	Let $V_1, V_2$ be two finitely dimensional vector spaces endowed with norms $N_i = \norm{\cdot}_i$, $i = 1, 2$.
	There are several natural constructions of a norm on the tensor product $V_1 \otimes V_2$.
	\par 
	Recall that the \textit{projective norm} $N_1 \otimes_{\pi} N_2 =  \norm{\cdot}_{N_1 \otimes_{\pi} N_2} = \norm{\cdot}_{\otimes_{\pi}}$ on $V_1 \otimes V_2$ is defined as 
	\begin{equation}\label{eq_defn_proj_norm}
		\norm{f}_{ \otimes_{\pi} }
		=
		\inf
		\Big\{
			\sum \| x_i \|_1 \cdot  \| y_i \|_2
			;
			\quad
			f = \sum x_i \otimes y_i
		\Big\},
	\end{equation}
	where the infimum is taken over different ways of partitioning $f$ into a sum of decomposable terms.
	\par 
	Recall that the \textit{injective norm} $N_1 \otimes_{\epsilon} N_2 = \norm{\cdot}_{N_1 \otimes_{\epsilon} N_2} = \norm{\cdot}_{ \otimes_{\epsilon} }$ on $V_1 \otimes V_2$ is defined as 
	\begin{equation}\label{eq_defn_inf_norm}
		\norm{f}_{ \otimes_{\epsilon} }
		=
		\sup
		\Big\{
			\big|
				(\phi \otimes \psi)(f)
			\big|
			;
			\quad
			\phi \in V_1^*, \psi \in V_2^*, \| \phi \|_{1}^* = \| \psi \|_{2}^* = 1
		\Big\}
	\end{equation}
	where $\| \cdot \|_{i}^*$, $i = 1, 2$, are the dual norms associated to $\| \cdot \|_{i}$.
	\par 
	When both $\| \cdot \|_{i}$ are induced by Hermitian products (then we say that $\| \cdot \|_{i}$ are Hermitian norms), we will denote by $N_1 \otimes N_2 = \| \cdot \|_{N_1 \otimes N_2} = \| \cdot \|_{ \otimes }$ the \textit{Hermitian tensor product norm} induced by the product of the respective Hermitian products.
	See Appendix \ref{app_norms} for relations between the above tensor norms and some of their basic properties, used sometimes implicitly below.
	\par 
	A norm $N_V = \| \cdot \|_V$ on a finitely dimensional vector space $V$ naturally induces the norm $\| \cdot \|_Q$ on any quotient $Q$, $\pi : V \to Q$ of $V$ through the following identity
	\begin{equation}\label{eq_defn_quot_norm}
		\| f \|_Q
		:=
		\inf \big \{
		 \| g \|_V
		 ;
		 \quad
		 g \in V, 
		 \pi(g) = f
		\},
		\qquad f \in Q.
	\end{equation}
	By a slight abuse of notations, we sometimes denote the quotient norm by $[ N_V ]$, i.e. without the reference to the quotient space. 
	This will not cause any trouble as the quotient will be explicit.
	\par 
	Now, it is standard that there is $p_0 \in \nat$, such that for any $k_1, \cdots, k_r \geq p_0$, the map ${\rm{Mult}}_{k_1, \cdots, k_r}$ is surjective, cf. Proposition \ref{prop_mult_surj}.
	Our first main result shows that the induced isomorphism quotient map is an approximate isometry (up to a multiplication by a constant) with respect to the $L^1, L^2$ and $L^{\infty}$-norms once the tensor product is endowed with an appropriate tensor product norm.
	More precisely, in Section \ref{sect_quot_mn}, we establish our first main result.
	\begin{thm}\label{thm_as_isom}
		There are $C > 0$, $p_1 \in \nat^*$, such that for any $k, l \geq p_1$, for the norms over $H^0(X, L^{k + l})$, under (\ref{eq_mult_map}), the following relation holds
		\begin{equation}\label{eq_as_isom1}
			1 - C \Big( \frac{1}{k} + \frac{1}{l} \Big)
			\leq 
			\frac{[{\rm{Hilb}}_k(h^L) \otimes {\rm{Hilb}}_l(h^L)]}{{\rm{Hilb}}_{k + l}(h^L)} 
			\cdot
			\Big( \frac{k \cdot l}{k + l} \Big)^{\frac{n}{2}}  
			\leq 
			1 + C \Big( \frac{1}{k} + \frac{1}{l} \Big).
		\end{equation}
		Similarly, for $L^1$ and $L^{\infty}$-norms the analogous results hold. There, instead of Hermitian tensor product norm, one needs to consider the projective and injective tensor product norms respectively.
		More precisely, in the notations of (\ref{eq_as_isom1}), we have
		\begin{equation}\label{eq_as_isom2}
		\begin{aligned}
			\frac{1}{4^n}
			\cdot
			\Big( 1 - C  \Big( \frac{1}{k} + \frac{1}{l} \Big) \Big)
			\leq 
			& 
			\frac{[{\rm{Ban}}_k^1(h^L) \otimes_{\pi} {\rm{Ban}}_l^1(h^L)]}{{\rm{Ban}}_{k + l}^1(h^L)} 
			\cdot
			\Big( \frac{k \cdot l}{k + l} \Big)^{n} 
			\cdot
			\frac{1}{2^n}
			\\
			&
			\qquad \qquad \qquad \qquad \qquad \qquad 
			\leq 
			4^n
			\cdot
			\Big( 1 + C  \Big( \frac{1}{k} + \frac{1}{l} \Big) \Big),
			\\
			1 - C  \Big( \frac{1}{k} + \frac{1}{l} \Big)
			\leq 
			& \frac{[{\rm{Ban}}_k^{\infty}(h^L) \otimes_{\epsilon} {\rm{Ban}}_l^{\infty}(h^L)]}{{\rm{Ban}}_{k + l}^{\infty}(h^L)} 
			\leq 
			1 + C  \Big( \frac{1}{k} + \frac{1}{l} \Big).
		\end{aligned}
		\end{equation}
	\end{thm}
	\begin{rem}
		a)
		According to Pisier \cite[Théorème 3.1]{PisierDecart}, cf. also \cite[Definition 2 and Theorem 6]{SzarekComp}, the projective and injective tensor norms differ by a value which tends to infinity as the dimensions of the vector spaces tends to infinity.
		Our result is, hence, sensible to the change of the projective norm by the injective norm.
		\par 
		b) The constant $4^n$ from the first equation of (\ref{eq_as_isom2}) is not optimal.
		It seems, however, to us that one cannot replace this constant by $1$, see Remark \ref{rem_dual_barg} for a related discussion.
		\par 
		c) One can guess the appearance of the injective and projective tensor products in (\ref{eq_as_isom2}) from the classical result on the tensor products of $L^1$ and $\ccal^0$-spaces, see Lemma \ref{lem_inj_proj_expl}.
	\end{rem}
	\par 
	In the second part of this paper, we show that the $L^1$, $L^2$ and $L^{\infty}$-norms are characterized by the multiplicative properties as in Theorem \ref{thm_as_isom}.
	Naturally, for our characterization, we need to consider non-smooth metrics.
	\par 
	Recall that a continuous metric $h^L$ on $L$ is called \textit{plurisubharmonic} (\textit{psh} for short) if the $(1, 1)$-current $c_1(L, h^L)$, see (\ref{eq_chern_class_defn}) for a definition, is positive.
	We denote by $\mathcal{H}^L$ the set of continuous psh metrics on $L$.
	Clearly, one can extend the definition of ${\rm{Ban}}^{\infty}(h^L)$ for any $h^L \in \mathcal{H}^L$.
	By using the Bedford-Taylor definition of Monge-Ampère operator, see \cite{BedfordTaylor} or \cite[Proposition I.3.2]{DemCompl}, cf. also (\ref{eq_bed_tay}), one can also extend the definition of the volume form $dv_X := \frac{1}{n!}c_1(L, h^L)^n$ as a non-negative measure on $X$ for any $h^L \in \mathcal{H}^L$.
	Using this measure, we extend the definition of ${\rm{Ban}}^1(h^L)$ and ${\rm{Hilb}}(h^L)$ for any $h^L \in \mathcal{H}^L$. See the discussion after (\ref{eq_bm_defn11}) for details.
	\par 
	Over finitely dimensional vector spaces $V$, we will later need to compare different norms. 
	For this, the Goldman-Iwahori distance (named after \cite{GoldIwah}, cf. also \cite[(1.1)]{BouckErik21}), $d_{\infty}(N_1, N_2)$, between two norms $N_1 = \norm{\cdot}_1$ and $N_2 = \norm{\cdot}_2$ over $V$ is defined as follows
	\begin{equation}\label{defn_gold_iwah}
		d_{\infty} ( N_1, N_2 )
		=
		\sup_{v \in V \setminus \{0\}} \big| \log \norm{v}_1 - \log \norm{v}_2 \big|.
	\end{equation}
	\par 
	Now, we are aiming to classify graded norms on the section ring up to the equivalence relation ($\sim$) defined as follows. We say that graded norms $N = \sum_{k = 1}^{\infty} N_k$ and $N' = \sum_{k = 1}^{\infty} N'_k$ are \textit{equivalent} if the graded pieces satisfy the following asymptotics
	\begin{equation}\label{eq_equiv_rel_defn}
		\frac{1}{k} d_{\infty}(N_k, N'_k)  \to 0.
	\end{equation} 
	One motivation for considering this equivalence relation is that it distinguishes elements from the image of the Hilbert map.
	See Theorem \ref{thm_hilb_leaf} for a more precise statement.
	\begin{sloppypar} 
	It is well-known that for smooth positive metrics $h^L$ on $L$, the multiplicative gap between ${\rm{Ban}}^1_k(h^L)$, ${\rm{Hilb}}_k(h^L)$, ${\rm{Ban}}^{\infty}_k(h^L)$ is at most polynomial in $k$.
	From the results of Berman-Boucksom-Witt Nystr{\"o}m \cite[Theorem 1.14]{BerBoucNys}, cf. Corollary \ref{cor_ident_maps}, we know that more generally, the graded norms ${\rm{Ban}}^1(h^L)$, ${\rm{Hilb}}(h^L)$, ${\rm{Ban}}^{\infty}(h^L)$ are equivalent for any $h^L \in \mathcal{H}^L$.
	Due to this, we concentrate now only on the study of Hermitian norms.
	\end{sloppypar}
	\begin{defn}\label{defn_mult_gen}
		We say that a graded Hermitian norm $N:= \sum_{k = 1}^{\infty} N_k$ on $R(X, L)$ is \textit{multiplicatively generated} if there is $p_0 \in \nat$ and a function $f : \nat_{\geq p_0} \to \real$, verifying $f(k) = o(k)$, as $k \to \infty$, such that for any $r \in \nat^*$, $k; k_1, \ldots, k_r \geq p_0$, $k_1 + \cdots + k_r = k$, under the map (\ref{eq_mult_map}), we have the following bound
		\begin{equation}\label{eq_mult_gen}
			 d_{\infty} \Big( N_k,  
			 \big[ 
			 N_{k_1} \otimes \cdots \otimes N_{k_r}
			 \big]
			 \Big)
			 \leq
			 f(k_1) + \cdots + f(k_r) + f(k).
		\end{equation}
	\end{defn}
	\begin{rem}
		The class of graded multiplicatively generated Hermitian norms is closed under the above equivalence relation.
	\end{rem}
	\par 
	From Theorem \ref{thm_as_isom}, ${\rm{Hilb}}(h^L)$ is multiplicatively generated for any smooth positive metric $h^L$ on $L$, see Remark \ref{rem_mult_gen_sm}. 
	More generally, in Section \ref{sect_mult_gen_psh}, we establish the following result.
	\begin{thm}\label{thm_mult_gen}
		For any $h^L \in \mathcal{H}^L$, the graded norm ${\rm{Hilb}}(h^L)$ is multiplicatively generated.
	\end{thm}
	\par 
	We denote by $\mathscr{N}^{R(X, L)}$ the space of graded multiplicatively generated norms on $R(X, L)$, and by $\mathscr{N}^{R(X, L)}_{\sim}$ the set of their equivalence classes.
	By an abuse of notation, let
	\begin{equation}\label{eq_hilb_map}
			{\rm{Hilb}} : \mathcal{H}^L \to \mathscr{N}^{R(X, L)}_{\sim}
	\end{equation}
	denotes the induced map.
	In Section \ref{sect_fin_ress_pf}, we establish our second main result.
	\par 
	\begin{thm}\label{thm_hilb_leaf0}
		The map ${\rm{Hilb}}$ from (\ref{eq_hilb_map}) is a bijection.
	\end{thm}
	\par 
	We will now refine Theorem \ref{thm_hilb_leaf0} by proving that the above bijection is actually an isometry with respect to the natural metric structures on the spaces $\mathscr{N}^{R(X, L)}_{\sim}$ and $\mathcal{H}^L$.
	For this, we define the $L^{\infty}$-pseudometric $d_{\infty} : \mathscr{N}^{R(X, L)} \times \mathscr{N}^{R(X, L)} \to [0, + \infty]$ as follows
	\begin{equation}\label{dist_gold_iwah}
		d_{\infty}(N, N')
		=
		\limsup_{k \to \infty} \frac{1}{k} d_{\infty}(N_k, N'_k).
	\end{equation}
	By Remark \ref{rem_fin_dist}, for any $N, N' \in \mathscr{N}^{R(X, L)}$, we always have $d_{\infty}(N, N') < + \infty$.
	Clearly, the Hausdorff quotient of $(\mathscr{N}^{R(X, L)}, d_{\infty})$ (i.e. the set of equivalence classes of $\mathscr{N}^{R(X, L)}$ with respect to the equivalence relation defined by the annihilation of $d_{\infty}$) coincides with $\mathscr{N}^{R(X, L)}_{\sim}$, and $d_{\infty}$, hence, yields the induced $L^{\infty}$-distance function 
	\begin{equation}\label{eq_d_inf_sim_sp}
		d_{\infty} : \mathscr{N}^{R(X, L)}_{\sim} \times \mathscr{N}^{R(X, L)}_{\sim} \to [0, + \infty[.
	\end{equation}
	We also define the $L^{\infty}$-distance $d_{\infty} : \mathcal{H}^L \times \mathcal{H}^L \to [0, + \infty[$ as follows
	\begin{equation}\label{dist_smooth}
		d_{\infty}(h^L_0, h^L_1)
		=
		\frac{1}{2}
		\sup \Big | \log \frac{h^L_0}{h^L_1} \Big |.
	\end{equation}
	Our next and final result, proved in Section \ref{sect_hilb}, is as follows.
	\begin{thm}\label{thm_hilb_leaf}
		The map ${\rm{Hilb}}$ from (\ref{eq_hilb_map}) preserves distances (\ref{eq_d_inf_sim_sp}) and  (\ref{dist_smooth}).
	\end{thm}
	\par 
	We will now describe an application of Theorem \ref{thm_hilb_leaf0}.
	We fix $h^L_0, h^L_1 \in \mathcal{H}^L$.
	Consider the weak Mabuchi geodesic $h^L_t \in \mathcal{H}^L$, $t \in [0, 1]$, connecting $h^L_0$ and $h^L_1$, see Section \ref{sect_hilb} for necessary definitions.
	Now, for any $k \in \nat^*$, consider a geodesic $H_{k, t}$, $t \in [0, 1]$, connecting ${\rm{Hilb}}_k(h^L_0)$ and ${\rm{Hilb}}_k(h^L_1)$ in the space of Hermitian norms on $H^0(X, L^k)$.
	Explicitly, if $A_k \in \enmr{H^0(X, L^k)}$ is the self-adjoint map, relating the scalar products on $H^0(X, L^k)$ associated to ${\rm{Hilb}}_k(h^L_0)$ and ${\rm{Hilb}}_k(h^L_1)$ as 
	$
		\scal{\cdot}{\cdot}_{L^2_k(X, h^L_1)}
		=
		\scal{A_k \cdot}{\cdot}_{L^2_k(X, h^L_0)},
	$
	then $H_{k, t}$ is the Hermitian norm associated to the scalar product $\scal{A^t_k \cdot}{\cdot}_{L^2_k(X, h^L_0)}$.
	Consider the associated graded norm $H_t := \sum H_{k, t}$ on $R(X, L)$.
	\begin{thm}\label{thm_pj_st_ref}
		For any $t \in [0, 1]$, the graded Hermitian norms $H_t$ and ${\rm{Hilb}}(h^L_t)$ are equivalent.
	\end{thm}
	\par 
	Let us now explain briefly the ideas of the proofs of the above results and their places in the context of previous works.
	\par 
	The central idea of our approach in the whole article is to give various interpretations of the multiplication map on the section ring in terms of the restriction morphisms.
	This is useful because the latter morphism can be studied by the methods developed previously by author, \cite{FinOTAs}, \cite{FinToeplImm}, \cite{FinOTRed}.
	\par 
	Theorem \ref{thm_as_isom} relies on the interpretation of the multiplication map as a restriction map from the product manifold $X \times X$ to the diagonal $\Delta$ inside of it, see (\ref{eq_comm_diag}), and on a combination of several statements.
	First, through Künneth theorem, one can view the space of holomorphic sections over $X \times X$ as the tensor product of the spaces of holomorphic sections over $X$, see (\ref{eq_nat_isom1}). 
	In Theorem \ref{thm_compar_tens_prod}, we then express using this isomorphism the $L^1$, $L^2$ and $L^{\infty}$-norms on $X \times X$ in terms of projective, Hermitian and injective tensor product norms respectively induced by the $L^1$, $L^2$ and $L^{\infty}$-norms on $X$.
	Second, in Theorem \ref{thm_mult_well_def}, we establish the lower bounds from Theorem \ref{thm_as_isom} by relying on some ideas from the proof of the semiclassical trace theorem from \cite[\S 4.3]{FinOTAs} applied for the pair of the product manifold $X \times X$ and the diagonal $\Delta$.
	The upper bounds are then established in Theorem \ref{thm_mult_surj} using some tools developed previously in the proof of the semiclassical Ohsawa-Takegoshi extension theorem in \cite{FinOTAs}, \cite{FinToeplImm}, \cite{FinOTRed}, applied here again for the pair $(X \times X, \Delta)$.
	\par 
	Theorem \ref{thm_mult_gen} is established using the result of Berman-Boucksom-Witt Nystr{\"o}m \cite[Theorem 1.14]{BerBoucNys}, cf. Theorem \ref{thm_bbw}, about the Bernstein-Markov property of the Monge-Ampère operator and the Ohsawa-Takegoshi extension theorem with the uniform constant, cf. Theorem \ref{thm_ot_expl_const}, applied to the product manifold and the diagonal in it.
	\par 
	To show Theorem \ref{thm_hilb_leaf0}, we construct explicitly an inverse of the map ${\rm{Hilb}}$.
	This is done using the Fubini-Study operator which associates through the Kodaira map for any norm $N_k$ on $H^0(X, L^k)$ (for $k$ big enough) a continuous psh metric over $L$, which we denote by $FS(N_k)^{\frac{1}{k}}$, see (\ref{eq_fs_defn}) for details.
	We study in Theorem \ref{thm_conv_mg} the convergence of the Fubini-Study metrics on $L$ associated to a multiplicatively generated norm $N$.
	The main content of Theorem \ref{thm_hilb_leaf0} is to show that when the graded Hermitian metric is multiplicatively generated, its equivalence class is determined by its Fubini-Study metric.
	Note that the existence of the limit of the Fubini-Study metric alone doesn't determine the equivalence class of a general graded Hermitian metric on $R(X, L)$, see Section \ref{sect_ex_mg_norm}.
	\par 
	In realms of non-Archimedean geometry, a result, similar to Theorem \ref{thm_hilb_leaf0}, was established by Boucksom-Jonsson \cite[Theorem 6.1]{BouckJohn21} and Reboulet \cite[Theorem A]{Reboulet21}.
	Remark that in the non-Archimedean setting, this statement is obtained as a consequence of certain algebraic results, see \cite[proof of Theorem 2.3]{BouckJohn21}. 
	One of them is the non-Archimedean GAGA principle, another is the study of spectral radius of norms over reduced k-affinoid algebras.
	Our methods, on the contrary, are purely analytic in nature and the only nontrivial ingredient in our proof is the semiclassical version of the Ohsawa-Takegoshi extension theorem from \cite{FinOTAs}.
	Interestingly enough, our proof of Theorem \ref{thm_hilb_leaf0} is built once again on an interpretation of the multiplication map as a version of the restriction map, see (\ref{eq_kod_map_comm_d}).
	But now the restriction is done with respect to the Kodaira embedding and not with respect to the diagonal in the product manifold.
	\par 
	We prove Theorem \ref{thm_hilb_leaf} by standard techniques, essentially relying on the Ohsawa-Takegoshi extension theorem.
	As we will show in Section \ref{sect_hilb} after explaining the necessary definitions, this result can be restated in a form which relates the limit of Goldman-Iwahori distance on the space of norms on $H^0(X, L^k)$, as $k \to \infty$, and the $L^{\infty}$-distance of the Mabuchi weak geodesic in $\mathcal{H}^L$. Hence, the former distance should be regarded as a quantization of the latter one.
	\par 
	This goes in line with the general philosophy that the geometry of $\mathcal{H}^L$ can be approximated by the geometry on the space of norms on $H^0(X, L^k)$, as $k \to \infty$, see Donaldson \cite{DonaldSymSp} and Section \ref{sect_hilb} for a more detailed review of related results.
	\par 
	The proof of Theorem \ref{thm_pj_st_ref}, which goes further in this philosophy, is based on the results of Phong-Sturm \cite[Theorem 1]{PhongSturm} and Berndtsson \cite[Theorem 1.2]{BerndtProb} about the convergence of the Fubini-Study metric associated to $H_t$, on the proof of Theorem \ref{thm_hilb_leaf0} and on some elementary statements about the monotonicity of geodesics in the space of Hermitian norms, which we establish by interpolation theory from functional analysis.
	Theorem \ref{thm_pj_st_ref} refines the mentioned result of Phong-Sturm and Berndtsson by Theorem \ref{thm_quant_hilb_conv}, Lemma \ref{lem_bnd_FS} and Proposition \ref{prop_expl_exmpl}.
	\par 
	This article is organized as follows.
	In Section \ref{sect_ext_thm_comp}, we introduce a set of techniques which will be later used to compare various norms on the section ring.
	In Section \ref{sect_quot_mn}, we study the relation between metric and algebraic structure of the section ring and establish Theorem \ref{thm_as_isom}.
	In Section \ref{sect_mult_gen_mn}, we prove Theorem \ref{thm_mult_gen}, give a classification of multiplicatively generated norms on the section ring and study the metric properties of this set. 
	In particular, we establish Theorems \ref{thm_hilb_leaf0}, \ref{thm_hilb_leaf}, \ref{thm_pj_st_ref}.
	In Appendix \ref{sect_unif_conv}, we develop the theory of almost sub-additive sequences, essential for Section \ref{sect_mult_gen_mn}.
	In Appendix \ref{app_norms}, we recall some classical results about normed vector spaces and the induced norms on the tensor products, which is of utmost importance to the study of $L^1$ and $L^{\infty}$-norms.
	\par 
	\textbf{Notation}.
	For vector bundles $E, F$ over $X$, the vector bundle $E \boxtimes F$ over $X \times X$ is defined as the tensor product $\pi_1^* E \otimes \pi_2^* F$, where $\pi_1, \pi_2$ are the projections $X \times X \to X$ on the first and the second components.
	\par 
	Let $\mu$ be an arbitrary measure on $X$.
	For a fixed Hermitian metric $h^L$ on $L$, we define the $L^2$-norm, which we denote by $\| \cdot \|_{L^2_k(X, h^L, \mu)}$ on $H^0(X, L^k)$ in the same way as in (\ref{eq_l2_prod}), but with the measure $dv_X$ replaced by $\mu$.
	When the metric $\mu$ is given by the symplectic volume form of a Kähler metric $\omega$, we denote the above norm by $\| \cdot \|_{L^2_k(X, h^L, \omega)}$.
	Similarly, we let ${\rm{Hilb}}_k(h^L, \mu)$ and ${\rm{Hilb}}_k(h^L, \omega)$ be the restrictions of the above norms to $H^0(X, L^k)$.
	Sometimes, to strengthen the dependence on the manifold $X$, we denote the above metrics by ${\rm{Hilb}}^X_k(h^L)$, etc.
	Similarly, we use the notations ${\rm{Ban}}^{1, X}(h^L)$ and ${\rm{Ban}}^{\infty, X}(h^L)$.
	When $h^L \in \mathcal{H}^L$, and $\mu = \frac{1}{n!} c_1(F, h^F)^n$ is constructed using Bedford-Taylor extension of the Monge-Ampère operator, cf. (\ref{eq_bed_tay}), we will omit $\mu$ from the notation and borrow the notation for psh metrics from the smooth situation $\| \cdot \|_{L^2_k(X, h^L)}$.
	Similarly, we extend the $L^1$-norm and use the analogous notation.
	\par 
	A Hermitian metric $h^F$ on a vector bundle $F$ is called \textit{bounded} if it can be bounded from above and below by some smooth metrics.
	For a complex manifold $X$, let $K_X$ be the canonical line bundle over it, given by $\wedge^n (T^{(1, 0)*}X)$.
	For a submanifold $Y$ in $X$, we denote by $\res_Y$ the restriction operator from $X$ to $Y$. 
	\par 
	\textbf{Acknowledgement}. 
	We thank wholeheartedly Sébastien Boucksom for many stimulating and fruitful discussions, which helped us to significantly improve this article.
	We thank Rémi Reboulet for several discussions we had after the first version of this paper was available.
	We also thank Yoshinori Hashimoto for an interesting conversation about Section \ref{sect_ex_mg_norm} and László Lempert for allowing us to reproduce his unpublished example from Proposition \ref{prop_ll_ex}.
	We, finally, thank the colleagues and staff of École Polytechnique, CMLS and CNRS for their support and wonderful working conditions.
	\par 
	{\bf{Dedication.}} This article is dedicated to the memory of Jean-Pierre Demailly, whose mathematics, teaching and personality have profoundly influenced me.

\section{Extension theorem and norms on the section ring}\label{sect_ext_thm_comp}
	The main goal of this section is to introduce a set of techniques to compare various norms on the section ring.
	This will be intimately related with Ohsawa-Takegoshi extension theorem, two versions of which we recall in Section \ref{sect_ohs}, and to Bernstein-Markov property, which we recall in Section \ref{sect_bern_m}.
	Those preliminaries will be used in Section \ref{sect_fs_psh}. 
	There we will recall the definition of the Fubini-Study operator, which will be our main technical tool in Section \ref{sect_mult_gen_mn}, and recall that any continuous psh metric can be obtained through approximation using this map.

\subsection{Two versions of Ohsawa-Takegoshi extension theorem}\label{sect_ohs}
	The main goal of this section is to recall a version of Ohsawa-Takegoshi extension theorem with uniform constant due to Demailly \cite[Theorem 2.8]{DemExtRed} and Ohsawa \cite{OhsawaV} and the semiclassical version due to the author \cite{FinOTAs}.
	\par 
	In the first part of this section, we follow closely the presentation from \cite{DemExtRed} and \cite{DemBookAnMet}. Only Theorem \ref{thm_ot_asymp} seems to be new.
	\par 
	Let $Y$ be a complex submanifold of a compact complex manifold $X$.
	Let $F$ be a holomorphic vector bundle $X$.
	Ohsawa-Takegoshi theorem addresses the following extension problem. 
	Given a holomorphic section $f$ of $F$ on $Y$, we would like to find a holomorphic extension $\tilde{f}$ of $f$ to $X$, together with a good estimate on the $L^2$-norm of $\tilde{f}$ in terms of the $L^2$-norm of $f$.
	It turns out that it is always possible modulo some positivity condition on the curvature of $F$.
	To state this result precisely, we need to fix some notations first.
	\par 
	\begin{defn}
		We call a function $\psi : X \to [-\infty, +\infty[$ on a complex manifold $X$ \textit{quasi-plurisubharmonic} (quasi-psh) if $\psi$ is locally the sum of a psh function and of a smooth function (or equivalently, if $\imun \partial \dbar \psi$ is locally bounded from below). 
		In addition, we say that $\psi$ has \textit{neat analytic singularities} if every point $x \in X$ possesses an open neighborhood $U$ on which $\psi$ can be written as 
		\begin{equation}
			\psi(z) = c \cdot \log \sum_{1 \leq j \leq N} \big| g_k(z) \big|^2 + w(z),
		\end{equation}
		where $c \geq 0$, $g_k$ are holomorphic over $U$ and $w$ is smooth.
	\end{defn}
	\begin{defn}\label{defn_mult_ideal}
		If $\psi$ is a quasi-psh function on a complex manifold $X$, the \textit{multiplier ideal sheaf} $\mathcal{J}(\psi)$ is the coherent analytic subsheaf of $\mathscr{O}_X$, cf. \cite[Proposition 5.7]{DemBookAnMet}, defined by 
		\begin{equation}
			\mathcal{J}(\psi) := \Big\{
				f \in \mathscr{O}_{X, x}; \quad \text{there exists } U \ni x, \int_U |f|^2 e^{- \psi} d \lambda < + \infty
			\Big\}.
		\end{equation}
		where $U$ is an open coordinate neighborhood of $x$, and $d \lambda$ the standard Lebesgue measure in the corresponding open chart of $\comp^n$. 
		We say that the singularities of $\psi$ are \textit{log canonical singularities} along the zero variety $Y = V (\mathcal{J}(\psi))$ if $\mathcal{J}(\psi) = \mathcal{J}_Y$, where $\mathcal{J}_Y$ is the sheaf of holomorphic sections vanishing along $Y$.
	\end{defn}
	\par 
	Let us define a bump function $\rho : \real \to [0, 1]$ as follows
	\begin{equation}\label{defn_rho_fun}
		\rho(x) =
		\begin{cases}
			1, \quad \text{for $|x| < \frac{1}{4}$},\\
			0, \quad \text{for $|x| > \frac{1}{2}$}.
		\end{cases}
	\end{equation}
	\par
	\begin{lem}\label{lem_delta_psh}
		Let $X$ be a compact Kähler manifold of dimension $n$, and $\omega$ be a Kähler form on $X$. 
		Let $Y$ be a closed submanifold of $X$ of dimension $m$.
		Then for $r > 0$ small enough, the function $\delta_Y : X \setminus Y \to \real$, defined as
		\begin{equation}\label{eq_delta_defn_y}
			\delta_Y(x) := 2 (n - m) \log \big(\dist_X(x, Y) \big) \cdot \rho 
			\Big(
				 \frac{\dist_X(x, Y)}{r} 
			\Big),
		\end{equation}
		is quasi-psh.
		Moreover, it has neat analytic log canonical singularities along $Y$. 
	\end{lem}
	\begin{proof}
		Remark first that only the statement about the quasi-plurisubharmonicity is non-trivial.
		For the proof of a similar statement, which works for general analytic submanifolds, see Demailly \cite[Proposition 1.4]{Dem82}.
		For exactly this statement, proved in a more general setting of manifolds of bounded geometry, a reader may consult \cite[Theorem 2.31]{FinOTAs}.
	\end{proof}	
	\par 
	Let us now consider quasi-psh function $\psi$ with log canonical singularities on $X$ such that $Y := V (\mathcal{J}(\psi))$ is a submanifold.
	Let $\mu_X$ be a fixed measure on $X$ with smooth positive density with respect to the Lebesgue measure on $X$.
	Let us associate in a canonical way to it a measure on $Y$.
	We fix $g \in \ccal^{\infty}_c(Y)$ and let $\tilde{g}$ be a compactly supported continuous extension of $g$ on $X$.
	We define the measure $\mu_Y[\psi]$ on $Y$ through the following limit
	\begin{equation}
		\int_{Y} g d \mu_Y[\psi]
		:=
		\limsup_{t \to -\infty} \int_{U_t} \tilde{g} e^{- \psi} d\mu_X,
	\end{equation}
	where $U_t \subset X$, $t \in \real$, is defined as $U_t := \{ x \in X;  t < \psi(x) < t + 1\}$.
	It is possible to see that the limit does not depend on the continuous extension $\tilde{g}$, and that one gets in this way a measure with smooth positive density with respect to the Lebesgue measure on $Y$, cf. \cite[Proposition 4.5]{DemExtRed}.
	\par 
	In the special case when $\mu$ corresponds to the Riemannian volume form $dV_{X, \omega}$ of Kähler manifold $(X, \omega)$ and $\psi$ equals to $\delta_Y$ from (\ref{eq_delta_defn_y}), an easy verification, cf. \cite[(2.6)]{DemExtRed}, shows that the volume form $\mu_Y[\psi]$ is related to the Riemannian volume form $dV_{Y, \omega|_Y}$ induced by $\omega|_Y$ as follows
	\begin{equation}\label{eq_mu_ind}
		\mu_Y[\psi] := \frac{2^{n - m + 1} \cdot \pi^{n - m}}{(n - m - 1)!} dV_{Y, \omega|_Y}.
	\end{equation}
	\par 
	Let us now introduce the function $\gamma : \real \to \real_+$ by the following formula
	\begin{equation}
		\gamma (x) = \begin{cases}
			\exp (- x ), & \text{if }\, x > 0,
			\\
			\frac{1}{1 + 4 x^2} & \text{if }\, x \leq 0.
		\end{cases}
	\end{equation}
	\par 
	\begin{thm}[ {\cite[Theorem 2.8, Remark 2.9b)]{DemExtRed} }]\label{thm_ot_expl_const}
	 	Let $X$ be a compact Kähler manifold, and $\omega$ be a Kähler form on $X$. 
	 	Let $(L, h^L)$ be a holomorphic line bundle equipped with a continuous Hermitian metric $h^L$ on $X$, and let $\psi : X \to [-\infty, +\infty[$ be a quasi-psh function on $X$ with neat analytic singularities. 
	 	Assume that the submanifold $Y$ of $X$ defined by $Y = V (\mathcal{J}(\psi))$ is smooth and $\psi$ has log canonical singularities. 
	 	Finally, assume that the $(1, 1)$-current 
	 	\begin{equation}\label{eq_cond_ot11}
	 		c_1(L, h^L) + \frac{\imun \alpha}{2 \pi} \partial \dbar \psi,
	 	\end{equation}
	 	is non-negative for any $\alpha \in [1, 2]$.
	 	Then for any section $f \in H^0(Y, (L \otimes K_X)|_Y)$, there is a holomorphic extension $\tilde{f} \in H^0(X, L \otimes K_X)$ of $f$, such that the following $L^2$-bound is satisfied
	 	\begin{equation}
	 		\int_X \gamma(\psi) |\tilde{f}|_{\omega, h}^2 e^{-\psi} dV_{X, \omega}
	 		\leq
	 		17
	 		\cdot
	 		\int_Y |f|_{\omega, h}^2  dV_{Y, \omega}[\psi],
	 	\end{equation}
	 	where the pointwise norm is calculated with respect to $h^L$ and $\omega$.
	 \end{thm}
	 \par 
	 Now, in the context of this article, the canonical line bundle and the condition (\ref{eq_cond_ot11}) do not appear naturally. Due to this, we will state an easy consequence of Theorem \ref{thm_ot_expl_const}, which will be better adapted for our needs.
	 \par 
	 We introduce first the notations for the following theorem.
	 Let $X$ be a projective manifold, $\omega$ is a Kähler metric on $X$ and $L$ is a fixed ample line bundle on $X$.
	 Let $Y$ be a closed submanifold of $X$.
	 Let $h^L$ be a continuous psh metric on $L$.
	 \begin{thm}\label{thm_ot_asymp}
		There are $C > 0$, $p_0 \in \nat$, such that for any $k \in \nat$, $k \geq p_0$, and any section $f \in H^0(Y, L|_Y^k)$, there is a holomorphic extension $\tilde{f} \in H^0(X, L^k)$ of $f$, such that the following $L^2$-bound is satisfied
	 	\begin{equation}\label{eq_ot_asymp}
	 		\| \tilde{f} \|_{L^2_k(X, h^L, \omega)}
	 		\leq
	 		C
	 		\| f \|_{L^2_k(Y, h^L, \omega)}.
	 	\end{equation}
	 \end{thm}
	 \begin{proof}
	 	Since $L$ is assumed to be ample, by Kodaira theorem it carries a smooth positive metric $h^L_0$.
	 	By Lemma \ref{lem_delta_psh} and positivity of $h^L_0$, we see that there is $p_0 \in \nat$, such that the following  $(1, 1)$-current is non-negative
	 	\begin{equation}
	 		p_0 \cdot c_1(L, h^L_0) 
	 		-
	 		c_1(K_X, h^{K_X}_{\omega})
	 		+
	 		\frac{\imun \alpha}{2 \pi} \partial \dbar \delta_Y,
	 	\end{equation}
	 	for any $\alpha \in [1, 2]$, where $h^{K_X}_{\omega}$ is the metric on $K_X$ induced by $\omega$.
	 	Hence for any $k \in \nat$, we may apply Theorem \ref{thm_ot_expl_const} for quasi-psh function $\delta_Y$ and line bundle $L^k \otimes L^{p_0} \otimes K_X^{-1}$, endowed with the metric $h_1 := (h^L)^k \otimes (h^L_0)^{p_0} \otimes (h^{K_X}_{\omega})^{-1}$.
	 	In particular, from Theorem \ref{thm_ot_expl_const} and (\ref{eq_mu_ind}) we see that for any $f \in H^0(Y, L|_Y^{k + p_0})$, there is an extension $\tilde{f} \in H^0(X, L^{k + p_0})$ of $f$ to $X$, such that the following estimate is satisfied
	 	\begin{equation}\label{eq_ot_asymp1}
	 		\int_X \gamma(\delta_Y) |\tilde{f}|_1^2 e^{-\delta_Y} dV_{X, \omega}
	 		\leq
	 		17
	 		\cdot
	 		\int_Y |f|_1^2  dV_{Y, \omega|_Y},
	 	\end{equation}
	 	where the pointwise norm $|\cdot|_1$ above is taken with respect to $h_1$.
	 	Remark now that since $h^L$ is bounded, there is $C > 0$, such that 
	 	\begin{equation}\label{eq_ot_asymp2}
	 		C^{-1} h^L_0 \leq h^L \leq C h^L_0.
	 	\end{equation}
	 	Now, remark that by taking $r > 0$ from Lemma \ref{lem_delta_psh} small, we may achieve that $\delta_Y$ is non-positive.
	 	Then by an easy fact that for $x > 0$, we have
	 	$
	 		\frac{e^x}{1 + 4 x^2} \geq \frac{1}{17},
	 	$
	 	we deduce the following inequality
	 	\begin{equation}\label{eq_ot_asymp3}
	 		\gamma(\delta_Y) e^{-\delta_Y} \geq \frac{1}{17}.
	 	\end{equation}
	 	By combining (\ref{eq_ot_asymp1}), (\ref{eq_ot_asymp2}) and (\ref{eq_ot_asymp3}), we obtain (\ref{eq_ot_asymp}).
	 \end{proof}
	 \par 
	 In particular, the previous result recovers the well-known fact that the restriction morphism
	 \begin{equation}\label{eq_res_y_mor}
	 	\res_Y : H^0(X, L^k) \to H^0(Y, L|_Y^k),
	 \end{equation}
	 is surjective.
	 The following result refines this statement on the metric level in more regular setting.
	 \par 
	 \begin{thm}[{\cite[Theorem 1.1]{FinOTAs}}]\label{thm_semicalss_ot}
	 	Let $h^L$ be a smooth positive metric on $L$.
	 	There are $C > 0$, $p_1 \in \nat$, such that under (\ref{eq_res_y_mor}), for any $k \geq p_1$, the following relation between the $L^2$-norm on the total manifold $X$ and the $L^2$-norm on the submanifold $Y$ is satisfied
		\begin{equation}
			1 - \frac{C}{k} \leq  \frac{[{\rm{Hilb}}^X_k(h^L)]}{{\rm{Hilb}}^Y_k(h^L)} \cdot k^\frac{n - m}{2} \leq 1 + \frac{C}{k}.
		\end{equation}
	\end{thm}
	\begin{rem}\label{rem_sem_jetss}
		See also \cite[Theorem 1.3]{FinOTRed} for a refinement on the level of holomorphic jets. 
	\end{rem}

\subsection{Bernstein-Markov property and comparison of natural norms}\label{sect_bern_m}
	The main goal of this section is to compare several natural norms on the section ring associated to psh metrics.
	We will first recall the classical construction due to Bedford-Taylor of Monge-Ampère operator on positive $(1, 1)$-currents with bounded potentials, thus extending the domain of operators ${\rm{Ban}}^{1}$, ${\rm{Hilb}}$ to continuous psh functions. Then we recall several results concerning the Bernstein-Markov property of the Bedford-Taylor version of Monge-Ampère operator.
	\par 
	Assume that a holomorphic line bundle $F$ over a complex manifold $Y$ is endowed with a bounded metric $h^F$.
	We define the $(1, 1)$-current $c_1(F, h^F)$ as follows
	\begin{equation}\label{eq_chern_class_defn}
		c_1(F, h^F)
		=
		c_1(F, h^F_{0})
		+
		\frac{ \partial \dbar}{2 \pi  \imun} \Big[ \log \Big( \frac{h^F}{h^F_0} \Big) \Big],
	\end{equation}
	where $h^F_{0}$ is any smooth Hermitian product on $F$, and by $[g]$ we mean the distribution associated to a locally integrable function $g$.
	Poincaré-Lelong equation, cf. \cite[Theorem V.13.9]{DemCompl}, shows that such a definition is independent of the choice of $h^F_{0}$.
	\par 
	Now, adapting the construction of Bedford-Taylor \cite{BedfordTaylor}, cf. \cite[\S III.3]{DemCompl}, in our situation, we can define inductively the positive $(k, k)$-currents $c_1(F, h^F)^k$, $k \in \nat$, for any bounded psh Hermitian metric $h^F$ as follows
	\begin{equation}\label{eq_bed_tay}
		c_1(F, h^F)^{k + 1}
		=
		c_1(F, h^F_{0}) \wedge
		c_1(F, h^F)^k
		+
		\frac{ \partial \dbar}{2 \pi  \imun} \Big[
		\log \Big( \frac{h^F}{h^F_0} \Big)  \cdot
		c_1(F, h^F)^k
		\Big].
	\end{equation}
	The term under the brackets in (\ref{eq_bed_tay}) is well-defined as a $(k, k)$-current since the logarithm function in it is bounded by our assumption on $h^F$ and $c_1(F, h^F)^k$ is represented by a non-negative $(k, k)$-form with measure coefficients.
	By (\ref{eq_bed_tay}), we see that the De Rham class of $c_1(F, h^F)^k$ coincides with $c_1(F)^k$, where the De Rham cohomology is extended to distributions in the usual way.
	\par 
	Following the standard terminology, cf. \cite{NgZer}, \cite{Siciak} and \cite{BerBoucNys}, we say that the pair $(h^L, \mu)$ satisfies the \textit{Bernstein-Markov property}, if for any $\epsilon > 0$ there is $p_1 \in \nat$, such that for any $k \geq p_1$, the following relation between the $L^2$ and $L^{\infty}$-norms holds 
	\begin{equation}\label{eq_bm_defn11}
		\| \cdot \|_{L^2_k(X, h^L, \mu)}
		\geq
		\exp(- \epsilon k)
		\cdot
		\| \cdot \|_{L^{\infty}_k(X, h^L)}.
	\end{equation}	 
	Following \cite{Siciak} and \cite[Definition 1.9]{BerBoucNys}, we say that the pair $(h^L, \mu)$ is \textit{Bernstein-Markov with respect to psh weights} if and only if for any $\epsilon > 0$, there is $C > 0$, such that for any $p \geq 1$ and any upper semi-continuous function $\psi : X \to [-\infty, \infty[$, such that the $(1, 1)$-current $\frac{\imun}{\pi} \partial \dbar \psi + c_1(L, h^L)$ is positive, we have 
	\begin{equation}\label{eq_sbm_defn11}
		\int  e^{p \psi} d \mu
		\geq
		C 
		\exp(-\epsilon p)
		\cdot
		\sup_{x \in X} e^{p \psi}(x).
	\end{equation}
	\par 
	\begin{lem}\label{lem_bm_rest}
		If a pair $(h^L, \mu)$ is Bernstein-Markov with respect to psh weights, then it has Bernstein-Markov property.
		Moreover, in the estimate (\ref{eq_bm_defn11}) one can replace the $L^2$-norm by the $L^1$-norm.
	\end{lem}
	\begin{proof}
		As it was observed in \cite[after Remark 1.10]{BerBoucNys}, to prove the first part of the statement, it suffices to take $p = 2k$ and $\psi = \frac{1}{k} \log |s|$ in (\ref{eq_sbm_defn11}), where $s \in H^0(X, L^k)$.
		The Poincaré-Lelong equation, cf. \cite[(13.2)]{DemCompl}, implies that the above $\psi$ verifies the positivity requirement stated before (\ref{eq_sbm_defn11}) and inequality (\ref{eq_sbm_defn11}) reduces to (\ref{eq_bm_defn11}).
		The proof of the second part of the statement is analogous to the proof of the first one, one only needs to put $p = k$ in (\ref{eq_sbm_defn11}) instead.
	\end{proof}
	\begin{thm}[{Berman-Boucksom-Witt Nystr{\"o}m \cite[Proposition 1.13 and Theorem 1.14]{BerBoucNys}}]\label{thm_bbw}
		For any continuous psh metric $h^L$ on $L$, the pair $(h^L, \frac{1}{n!} c_1(L, h^L)^n)$ satisfies Bernstein-Markov property with respect to psh weights.
	\end{thm}
	We will also need the following result, philosophically similar to Theorem \ref{thm_bbw}, yet simpler.
	\begin{prop}\label{prop_bm}
		The pair $(h^L, \mu)$ satisfies Bernstein-Markov property with respect to psh weights for any continuous psh metrics $h^L$ on $L$ and any measure $\mu$, majorized from below by some strictly positive volume form on $X$. 
	\end{prop}
	\begin{proof}
		It follows directly from \cite[Theorem 1.14]{BerBoucNys} or from mean-value inequality.
	\end{proof}
	\begin{cor}\label{cor_ident_maps}
		For any $h^L \in \mathcal{H}^L$, the following equivalences between the graded norms on $R(X, L)$ hold
		\begin{equation}
			{\rm{Ban}}^1(h^L) \sim {\rm{Hilb}}(h^L) \sim {\rm{Ban}}^{\infty}(h^L).
		\end{equation}
		In particular, the maps (\ref{eq_hilb_map}) are all identical.
		Also, for any measure $\mu$, majorized from below by some strictly positive volume form on $X$ and having finite volume over $X$, we have
		\begin{equation}
			{\rm{Hilb}}(h^L) \sim {\rm{Hilb}}(h^L, \mu).
		\end{equation}
	\end{cor}
	\begin{proof}
		From the fact that the De Rham class of $c_1(L, h^L)^n$ is fixed, we readily deduce that 
		\begin{equation}\label{eq_est_derham}
		\begin{aligned}
			&
			{\rm{Ban}}^1(h^L) \leq {\rm{Ban}}^{\infty}(h^L) \cdot \frac{1}{n!} c_1(L)^n,
			\\
			&
			{\rm{Hilb}}(h^L) \leq {\rm{Ban}}^{\infty}(h^L) \cdot \sqrt{ \frac{1}{n!} c_1(L)^n},
			\\
			&
			{\rm{Hilb}}(h^L, \mu) \leq {\rm{Ban}}^{\infty}(h^L) \cdot \sqrt{{\rm{vol}}(\mu)}.
		\end{aligned}
		\end{equation}
		This along with Theorem \ref{thm_bbw}, Lemma \ref{lem_bm_rest} and Proposition \ref{prop_bm} imply the result.
	\end{proof}
	As a sidenote, let us establish the following corollary. Compare it with Theorem \ref{thm_semicalss_ot}.
	\begin{cor}\label{cor_comp_res_1}
	 	Let $h^L$ be a continuous psh metric on an ample line bundle $L$.
	 	Then under (\ref{eq_res_y_mor}), the following equivalence of norms holds
		\begin{equation}
			[{\rm{Hilb}}^X(h^L)] \sim {\rm{Hilb}}^Y(h^L).
		\end{equation}
	\end{cor}
	\begin{proof}
		From Corollary \ref{cor_ident_maps} and the obvious fact $[{\rm{Ban}}^{\infty, X}(h^L)] \geq {\rm{Ban}}^{\infty, Y}(h^L)$, stipulating that the $L^{\infty}$-norm of a section is always no smaller than the $L^{\infty}$-norm of its restriction, we conclude that for any $\epsilon > 0$ there is $p_1 \in \nat$, such that for $k \geq p_1$, we have
		\begin{equation}\label{eq_comp_res_12}
			[{\rm{Hilb}}^X(h^L)] \geq \exp(- \epsilon k) \cdot {\rm{Hilb}}^Y(h^L).
		\end{equation}
		\par 
		Now, following the terminology of Theorem \ref{thm_ot_asymp}, we can restate Theorem \ref{thm_ot_asymp} in the following way.
		There are $C > 0$, $p_1 \in \nat$, such that for any $k \geq p_1$, we have
		\begin{equation}
			{\rm{Hilb}}^Y(h^L, \omega) \geq C [{\rm{Hilb}}^X(h^L, \omega)].
		\end{equation}
		From Corollary \ref{cor_ident_maps}, we see that it implies that for any $\epsilon > 0$, there is $p_1 \in \nat$, such that for any $k \geq p_1$, the following inequality is satisfied
		\begin{equation}\label{eq_comp_res_13}
			{\rm{Hilb}}^Y(h^L) \geq \exp(- \epsilon k) \cdot [{\rm{Hilb}}^X(h^L)].
		\end{equation}
		But Corollary \ref{cor_comp_res_1} means exactly that both inequalities (\ref{eq_comp_res_12}) and (\ref{eq_comp_res_13}) are satisfied.
	\end{proof}
	\begin{rem}
		a) From Corollary \ref{cor_ident_maps}, we see that the analogous result holds for $L^1$ and $L^{\infty}$-norms.
		\par 
		b) For positive smooth metrics, the upper bound from (\ref{eq_comp_res_13}) was independently obtained by Randriam \cite[Theorem 3.1.10]{RandriamTh} in a more general setting of reduced submanifold (i.e. for holomorphic jets), see also Remark \ref{rem_sem_jetss}.
		The $L^{\infty}$-version of this result, working more generally for certain non-compact manifolds $X$ and non-smooth $Y$ but only for strictly positive singular metrics $h^L$ was established by Randriam \cite{RandriamCrelle}.
	\end{rem}

\subsection{Fubini-Study operator and continuous psh metrics}\label{sect_fs_psh}
	In this section we recall the definition of the Fubini-Study operator. We will then show that the Hilbert map can be used to approximate any continuous psh metric.
	\par 
	First of all, recall that for $k$ sufficiently large so that $L^k$ is very ample, Fubini-Study operator associates for any Hermitian norm $N_k = \| \cdot \|_k$ on $H^0(X, L^k)$, a positive metric $FS(N_k)$ on $L$, constructed in the following way.
	Consider the Kodaira embedding 
	\begin{equation}\label{eq_kod}
		{\rm{Kod}}_k : X \hookrightarrow \mathbb{P}(H^0(X, L^k)^*),
	\end{equation}
	defined as $x \mapsto H^0(X, L^k \otimes \mathcal{J}_x)$, where $\mathcal{J}_x$ is the ideal sheaf of holomorphic functions vanishing at $x$, and we implicitly used the identification between the hyperplanes in $H^0(X, L^k)$ and lines in $H^0(X, L^k)^*$.
	The evaluation maps, $ev_x : H^0(X, L^k) \to L^k_x$, $x \in X$, provide the isomorphism
	\begin{equation}\label{eq_ev_kod_iso}
		L^{-k} \to {\rm{Kod}}_k^* \mathscr{O}(-1),
	\end{equation}
	where $\mathscr{O}(-1)$ is the tautological bundle over the projective space $\mathbb{P}(H^0(X, L^k)^*)$.
	Let us now endow $H^0(X, L^k)^*$ with the dual norm $N_k^*$ and induce from it a metric $h^{FS}(N_k)$ on the tautological line bundle $\mathscr{O}(-1)$ over $\mathbb{P}(H^0(X, L^k)^*)$. 
	We define the metric $FS(N_k)$ on $L^k$ through
	\begin{equation}\label{eq_fs_defn}
		FS(N_k) = {\rm{Kod}}_k^* ( h^{FS}(N_k)^* ).
	\end{equation}
	A statement below can be seen as an alternative definition.
	\begin{lem}\label{lem_fs_inf_d}
		For any $x \in X$, the metric $FS(N_k)$ satisfies the identity
		\begin{equation}
			\sum_{i = 1}^{l} \big| s_i(x) \big|^2_{FS(N_k)} = 1,
		\end{equation}
		where $s_1, \ldots, s_l$ is an orthonormal basis of $(H^0(X, L^k), N_k)$.
	\end{lem}
	\begin{proof}
		An easy verification, cf. Ma-Marinescu \cite[Theorem 5.1.3]{MaHol}.
	\end{proof}
	\par 
	Explicit evaluation shows that $c_1(\mathscr{O}(-1) , h^{FS}(N_k))$ coincides up to a negative constant with the Kähler form of the Fubini-Study metric on $\mathbb{P}(H^0(X, L^k)^*)$ induced by $N_k$.
	In particular, $(c_1(\mathscr{O}(-1)) , h^{FS}(N_k))$ is a negative $(1, 1)$-form.
	\par 
	We will now discuss a result concerning quantization of continuous psh metrics.
	For it, we claim no originality, but despite a huge amount of literature, we weren't able to find a precise reference for it.
	We present the proof below for the convenience of the reader.
	\begin{thm}\label{thm_quant_hilb_conv}
	 For any continuous psh metric $h^L$ on $L$, the sequence of metrics $FS({\rm{Hilb}}_k(h^L))^{\frac{1}{k}}$ on $L$ converges uniformly, as $k \to \infty$, to $h^L$.
	\end{thm}
	\begin{rem}\label{rem_tian}
		Theorem \ref{thm_quant_hilb_conv} generalizes the result of Tian \cite{TianBerg}, who proved the above statement for smooth positive metric $h^L$ on $L$.
		Moreover, he proved in this case that $FS({\rm{Hilb}}_k(h^L))^{\frac{1}{k}}$ converges to $h^L$ in $\ccal^2$-sense instead of $\ccal^0$.
		Later this was refined by Zelditch \cite{ZeldBerg}, who proved that the convergence is $\ccal^{\infty}$.
		For other related works, see Bouche \cite{Bouche}, Catlin \cite{Caltin}, Lu \cite{LuBergman}, Wang \cite{WangBergmKern} and Dai-Liu-Ma \cite{DaiLiuMa}.
	\end{rem}
	\par 
	For the proof of Theorem \ref{thm_quant_hilb_conv} and later on, we will need the following statement, which is a trivial consequence of Theorem \ref{thm_ot_asymp}.
	\begin{lem}\label{lem_ot_loc}
		Assume that the measure $\mu$ is bounded from above by a smooth volume form on $X$.
		Then for any continuous psh metric $h^L$ on an ample line bundle $L$, there are $C > 0$, $p_1 \in \nat$, such that for any $k \geq p_1$, $x \in X$, there is $f \in H^0(X, L^k)$, verifying 
		\begin{equation}
			|f(x)| \geq C \| f \|_{L^2_k(X, h^L, \mu)},
		\end{equation}
		where the norm of $f(x)$ is taken with respect to $h^L$.
	\end{lem}
	\begin{proof}[Proof of Theorem \ref{thm_quant_hilb_conv}.]
		An easy verification, following from Lemma \ref{lem_fs_inf_d}, shows that
		\begin{equation}\label{eq_fs_div_metr}
			\frac{FS({\rm{Hilb}}_k(h^L))}{(h^L)^k}
			=
			\rho_k,
		\end{equation}
		where the function $\rho_k : X \to [0, + \infty[$ is defined as follows 
		\begin{equation}
			\rho_k(x) := \sup \frac{|s(x)|}{\| s \|_{L^2_k(X, h^L)}}.
		\end{equation}
		By Theorem \ref{thm_bbw} and Proposition \ref{prop_bm}, we see that for any $\epsilon > 0$, there is $p_1 \in \nat$, such that for any $k \geq p_1$, we have
		\begin{equation}\label{eq_rho_bnd}
			\rho_k(x) \leq \exp(\epsilon k).
		\end{equation}
		Now, a combination of Theorem \ref{thm_bbw} and (\ref{eq_est_derham}) implies that for any $\epsilon > 0$, there is $p_1 \in \nat$, such that for any $k \geq p_1$, we have
		\begin{equation}\label{eq_rho_bnd2}
			\| \cdot \|_{L^2_k(X, h^L, \mu)}
			\geq
			\exp(- \epsilon k)
			\cdot
			\| \cdot \|_{L^2_k(X, h^L)}.
		\end{equation}	 
		From Lemma  \ref{lem_ot_loc} and (\ref{eq_rho_bnd2}), we conclude that for any $\epsilon > 0$, there is $p_1 \in \nat$, such that for any $k \geq p_1$, we have
		\begin{equation}\label{eq_rho_bnd3}
			\rho_k(x) \geq \exp(-\epsilon k).
		\end{equation}
		From (\ref{eq_fs_div_metr}), (\ref{eq_rho_bnd}) and (\ref{eq_rho_bnd3}), we deduce Theorem \ref{thm_quant_hilb_conv}. 
	\end{proof}

\section{Synergy of metric and algebraic structures on section ring}\label{sect_quot_mn}
	The main goal of this section is to study the relation between metric and algebraic structures of the section ring. 
	In particular, we prove Theorem \ref{thm_as_isom}.
	\par 
	More precisely, this section is organized as follows.
	In Section \ref{sect_mult_res}, we make a connection between the restriction map and the multiplication morphism and prove Theorem \ref{thm_as_isom} modulo some technical statements.
	Those statements will be then established in Sections \ref{sect_prod_norms}, \ref{sect_mult_def} and \ref{sect_as_opt_ext_ban}.

\subsection{Multiplication map as a restriction morphism, a proof of Theorem \ref{thm_as_isom}}\label{sect_mult_res}
	The main goal of this section is to give a proof of Theorem \ref{thm_as_isom} modulo some technical results, which will be later established in this article.
	Our main idea is to interpret the multiplication map of the section ring (\ref{eq_mult_map}) as an instance of the restriction morphism.
	This will allow us to use some of the methods developed in \cite{FinOTAs}, \cite{FinToeplImm}, \cite{FinOTRed} to prove Theorem \ref{thm_as_isom}.
	\par 
	More precisely, we conserve the notation from Section \ref{sect_intro}.
	Let us consider the product manifold $X \times  X$ and the diagonal submanifold in it, given by $\{(x, x) : x \in X \} =: \Delta \hookrightarrow X \times X$.
	Clearly, we have a natural isomorphism
	\begin{equation}\label{eq_nat_isom12}
		\Delta \to X, \qquad (x, x) \mapsto x.
	\end{equation}
	Also, under the isomorphism (\ref{eq_nat_isom12}), we have
	\begin{equation}\label{eq_nat_isom2}
		L^k \boxtimes L^l|_{\Delta} \to L^{k + l}.
	\end{equation}		
	Künneth theorem allows us to interpret the tensor product of cohomologies as the cohomology of the product manifold. In other words, it provides us with the following isomorphism
	\begin{equation}\label{eq_nat_isom1}
		H^0(X \times  X, L^k \boxtimes L^l) \to H^0(X, L^k) \otimes H^0(X, L^l).
	\end{equation}
	\par 
	A basic observation of this section is that those  isomorphisms can be put into the following commutative diagram 
	\begin{equation}\label{eq_comm_diag}
		\begin{CD}
			H^0(X \times  X, L^k \boxtimes L^l)  @> {\rm{Res}}_{\Delta} >> H^0(\Delta, L^k \boxtimes L^l|_{\Delta}) 
			\\
			@VV {} V @VV {} V
			\\
			H^0(X, L^k) \otimes H^0(X, L^l)  @> {\rm{Mult}}_{k, l} >> H^0(X, L^{k + l}).
		\end{CD}
	\end{equation}
	A similar observation was already used by Demailly-Ein-Lazarsfeld \cite[proof of Theorem 2.3]{DemEinLaz} in the proof of subadditivity theorem for multiplier ideals.
	\par 
	We will now explain how this observation allows us to establish Theorem \ref{thm_as_isom}.
	Before that, let us use it to prove the following well-known result.
	\begin{prop}\label{prop_mult_surj}
		There is $p_0 \in \nat$, such that for any $r \in \nat^*$, $k; k_1, \ldots, k_r \in \nat^*$, $k_1 + \cdots + k_r = k$,  $k_1, \cdots, k_r \geq p_0$, the map ${\rm{Mult}}_{k_1, \cdots, k_r}$ is surjetive.
	\end{prop}
	\begin{proof}
		Clearly, it is enough to establish the result for $r = 2$.
		We borrow the notation from Section \ref{sect_ohs}.
		We have a short exact sequence of sheaves
		\begin{equation}
			0 
			\to 
			\mathscr{O}_{X \times X}(L^k \boxtimes L^l) \otimes \mathcal{J}_{\Delta}
			\to
			\mathscr{O}_{X \times X}(L^k \boxtimes L^l)
			\xrightarrow[]{{\rm{Res}}_{\Delta}}
			\mathscr{O}_{\Delta}((L^k \boxtimes L^l)|_{\Delta})
			\to
			0.
		\end{equation}
		The associated long exact sequence in cohomology gives us
		\begin{multline}\label{eq_long_exact}
			\cdots
			\to 
			H^0(X \times X, L^k \boxtimes L^l) 
			\xrightarrow[]{{\rm{Res}}_{\Delta}}
			H^0(\Delta, (L^k \boxtimes L^l)|_{\Delta}) 
			\\
			\to
			H^1(X \times X, L^k \boxtimes L^l \otimes \mathcal{J}_{\Delta}) 
			\to
			\cdots.
		\end{multline}
		\par 
		Clearly, from (\ref{eq_comm_diag}), to prove the surjectivity of ${\rm{Mult}}_{k, l}$, it is enough to establish the surjectivity of ${\rm{Res}}_{\Delta}$.
		From (\ref{eq_long_exact}), we see that for the surjectivity of ${\rm{Res}}_{\Delta}$ it is enough to establish 
		\begin{equation}\label{eq_van_coh}
			H^1(X \times X, L^k \boxtimes L^l \otimes \mathcal{J}_{\Delta}) = \{ 0 \},
		\end{equation}
		which follows directly from Nadel vanishing theorem \cite[Theorem 5.11]{DemBookAnMet} and Lemma \ref{lem_delta_psh}.
	\end{proof}
	\par 
	Next we will explain that (\ref{eq_comm_diag}) allows us to use our previous ideas from \cite{FinOTAs}, \cite{FinToeplImm} and \cite{FinOTRed}, studying the general restriction maps, to prove Theorem \ref{thm_as_isom}.
	We decompose our proof into a series of more tractable statements.
	By the isomorphism (\ref{eq_nat_isom1}), we can endow the right-hand side of (\ref{eq_nat_isom1}) with the natural norms $\| \cdot \|_{L^q_{k, l}(X \times  X, h^L \boxtimes h^L)}$, $q = 1, 2, \infty$, constructed by the restriction of the corresponding $L^q$-norms associated to the symplectic volume form of $\omega_{X \times X}$ to the vector space of holomorphic sections.
	Those norms will be of utmost importance.
	\par
	By using the interpretation (\ref{eq_comm_diag}) of the multiplication map as an instance of restriction morphism and relying on some ideas from the proof of the semiclassical trace theorem from \cite[\S 4.3]{FinOTAs}, we establish in Section \ref{sect_mult_def} the following result.
	\begin{thm}\label{thm_mult_well_def}
		There are $C > 0$, $p_1 \in \nat$, such that for any $k, l \geq p_1$, $f \in H^0(X, L^k)$, $g \in H^0(X, L^l)$, we have
		\begin{equation}\label{eq_mult_well_def}
		\begin{aligned}
			&			
			\big \| 
				f \cdot g
			\big \|_{{\rm{Hilb}}_{k + l}(h^L)}
			\leq
			\Big(
				1 + C  \Big( \frac{1}{k} + \frac{1}{l} \Big)
			\Big)
			\cdot
			\Big( \frac{k \cdot l}{k + l} \Big)^{\frac{n}{2}}  
			\cdot
			\big \| 
				f
			\big \|_{{\rm{Hilb}}_k(h^L)}
			\cdot
			\big \| 
				g
			\big \|_{{\rm{Hilb}}_l(h^L)},
			\\
			&
			\big \| 
				f \cdot g
			\big \|_{{\rm{Ban}}_{k + l}^1(h^L)}
			\leq
			\Big(
				1 + C  \Big( \frac{1}{k} + \frac{1}{l} \Big)
			\Big)
			\cdot
			\Big( \frac{k \cdot l}{k + l} \Big)^{n}
			\cdot
			4^n
			\cdot
			\big \| 
				f
			\big \|_{{\rm{Ban}}_k^{1}(h^L)}
			\cdot
			\big \| 
				g
			\big \|_{{\rm{Ban}}_l^{1}(h^L)}.
		\end{aligned}
		\end{equation}
		More precisely, we have the following estimates for operator norms
		\begin{equation}\label{eq_mult_well_def2}
		\begin{aligned}
			&
			\big\|
				{\rm{Mult}}_{k, l} 
			\big\|_{L^2_{k, l}(X \times  X, h^L \boxtimes h^L) \to L^2_{k + l}(X, h^L)}
			\leq
			\Big(
				1 + C  \Big( \frac{1}{k} + \frac{1}{l} \Big)
			\Big)
			\cdot
			\Big( \frac{k \cdot l}{k + l} \Big)^{\frac{n}{2}},
			\\
			&
			\big\|
				{\rm{Mult}}_{k, l} 
			\big\|_{L^1_{k, l}(X \times  X, h^L \boxtimes h^L) \to L^1_{k + l}(X, h^L)}
			\leq
			\Big(
				1 + C  \Big( \frac{1}{k} + \frac{1}{l} \Big)
			\Big)
			\cdot
			\Big( \frac{k \cdot l}{k + l} \Big)^{n}
			\cdot
			4^n,
		\end{aligned}
		\end{equation}
		where $\|
				\cdot
			\|_{L^q_{k, l}(X \times  X, h^L \boxtimes h^L) \to L^q_{k + l}(X, h^L)}$, $q = 1, 2, \infty$, is the operator norm of an operator from $H^0(X, L^k) \otimes H^0(X, L^l)$ to $H^0(X, L^{k + l})$, calculated when the domain is endowed with the norm $\| \cdot \|_{L^q_{k, l}(X \times  X, h^L \boxtimes h^L)}$ and the image is endowed with the norm $\| \cdot \|_{L^q_{k + l}(X, h^L)}$.
	\end{thm}
	\begin{rem}\label{rem_mult_well_def}
		a) As the $L^{\infty}$-norm of a product of two sections is no bigger than the product of the $L^{\infty}$-norms of the sections, we trivially have the identity $\big\|
				{\rm{Mult}}_{k, l} 
			\big\|_{L^{\infty}_{k, l}(X \times  X, h^L \boxtimes h^L) \to L^{\infty}_{k + l}(X, h^L)} \leq 1$.
		\par 
		b) The constant $4^n$ in the second estimates from (\ref{eq_mult_well_def}) and (\ref{eq_mult_well_def2}) is not optimal. It is possible to improve it through our methods, but as this improvement will not affect qualitatively the final result, we decided to omit it for brevity.
		\par 
		c) In the proof we show that the first estimate from (\ref{eq_mult_well_def2}) is sharp.
	\end{rem}
	\par
	By using again the interpretation of the multiplication map as a version of the restriction morphism, cf. (\ref{eq_comm_diag}), and establishing a version of asymptotic Ohsawa-Takegoshi theorem, following the methods from \cite{FinOTAs}, \cite{FinToeplImm}, we prove in Section \ref{sect_as_opt_ext_ban} the following result.
	\begin{thm}\label{thm_mult_surj}
		There are $C > 0$, $p_0 \in \nat$, such that for any $k, l \geq p_0$, $f \in H^0(X, L^{k + l})$, there is $h \in H^0(X, L^k) \otimes H^0(X, L^l)$, verifying  ${\rm{Mult}}_{k, l}(h) = f$ and
		\begin{equation}\label{eq_mult_surj}
		\begin{aligned}
			&
			\big \| 
				h
			\big \|_{L^1_{k, l}(X \times  X,  h^L \boxtimes h^L)}
			\leq
			\Big(
				1 + C  \Big( \frac{1}{k} + \frac{1}{l} \Big)
			\Big)
			\cdot
			\Big( \frac{k + l}{k \cdot l} \Big)^{n}  
			\cdot
			2^n
			\cdot
			\big \| 
				f
			\big \|_{{\rm{Ban}}_{k + l}^{1}(h^L)},
			\\
			&
			\big \| 
				h
			\big \|_{L^2_{k, l}(X \times  X,  h^L \boxtimes h^L)}
			\leq
			\Big(
				1 + C  \Big( \frac{1}{k} + \frac{1}{l} \Big)
			\Big)
			\cdot
			\Big( \frac{k + l}{k \cdot l} \Big)^{\frac{n}{2}}  
			\cdot
			\big \| 
				f
			\big \|_{{\rm{Hilb}}_{k + l}(h^L)},
			\\
			&
			\big \| 
				h
			\big \|_{L^{\infty}_{k, l}(X \times  X,  h^L \boxtimes h^L)}
			\leq
			\Big(
				1 + C  \Big( \frac{1}{k} + \frac{1}{l} \Big)
			\Big)
			\cdot
			\big \| 
				f
			\big \|_{{\rm{Ban}}_{k + l}^{\infty}(h^L)}.
		\end{aligned}
		\end{equation}
	\end{thm}
	\begin{rem}
		From Theorem \ref{thm_mult_well_def}, we see that the second and the third estimates in (\ref{eq_mult_surj}) are optimal.
		Through some additional work, one can see that the first estimate is optimal as well. 
	\end{rem}
	\par 
	From Theorems \ref{thm_mult_well_def} and \ref{thm_mult_surj}, a natural question arises if it is possible to interpret the norms $L^q_{k, l}(X \times  X,  h^L \boxtimes h^L)$, $k, l \in \nat$, $q = 1, 2, \infty$, in terms of the corresponding norms $L^q_{k}(X, h^L)$, $k \in \nat$. 
	The following result, which we establish in Section \ref{sect_prod_norms}, shows that this is indeed the case.
	\begin{thm}\label{thm_compar_tens_prod}
		The following estimate between the norms on $H^0(X, L^k) \otimes H^0(X, L^l)$ holds
		\begin{equation}\label{eq_compar_tens_prod111}
			\frac{1}{4^n}
			\cdot
			{\rm{Ban}}_k^{1}(h^L) \otimes_{\pi} {\rm{Ban}}_l^{1}(h^L)
			\leq
			\| \cdot \|_{L^1_{k, l}(X \times  X, h^L \boxtimes h^L)}
			\leq
			{\rm{Ban}}_k^{1}(h^L) \otimes_{\pi} {\rm{Ban}}_l^{1}(h^L).
		\end{equation}
		Also, the following identity holds
		\begin{equation}\label{eq_compar_tens_prod1112}
			\| \cdot \|_{L^{\infty}_{k, l}(X \times  X, h^L \boxtimes h^L)}
			=
			{\rm{Ban}}_k^{\infty}(h^L) \otimes_{\epsilon} {\rm{Ban}}_l^{\infty}(h^L).
		\end{equation}
	\end{thm}
	\begin{rem}\label{rem_compar_tens_prod}
		Clearly, the following identity 
		\begin{equation}
			\scal{f_1 \boxtimes g_1}{f_2 \boxtimes g_2}_{L^2_{k, l}(X \times  X, h^L \boxtimes h^L)} = \scal{f_1}{f_2}_{L^2_k(X, h^L)} \cdot \scal{g_1}{g_2}_{L^2_l(X, h^L)},
		\end{equation}
		shows that ${\rm{Hilb}}_k(h^L) \otimes {\rm{Hilb}}_l(h^L)$ coincides with the norm $L^2_{k, l}(X \times  X,  h^L \boxtimes h^L)$.
	\end{rem}
	\begin{proof}[Proof of Theorem \ref{thm_as_isom}]
		Once we unravel the definition of the quotient norm, we conclude directly from Theorems \ref{thm_mult_well_def} and \ref{thm_mult_surj} that there are $C > 0$, $p_1 \in \nat$, such that for any $k, l \geq p_1$, we have
		\begin{equation}\label{eq_as_isom1121}
		\begin{aligned}
			\frac{1}{4^n}
			\cdot
			\Big( 1 - C  \Big( \frac{1}{k} + \frac{1}{l} \Big)\Big)
			\leq 
			& 
			\frac{[\| \cdot \|_{L^1_{k, l}(X \times  X,  h^L \boxtimes h^L)}]}{{\rm{Ban}}_{k + l}^1(h^L)} 
			\cdot
			\Big( \frac{k \cdot l}{k + l} \Big)^{n} 
			\cdot
			\frac{1}{2^n}
			\\
			& \qquad \qquad \qquad \qquad \qquad \qquad
			\leq 
			\Big( 
			1 + C  \Big( \frac{1}{k} + \frac{1}{l} \Big)
			\Big),
			\\
			1 - C  \Big( \frac{1}{k} + \frac{1}{l} \Big)
			\leq 
			&
			\frac{[\| \cdot \|_{L^2_{k, l}(X \times  X,  h^L \boxtimes h^L)}]}{{\rm{Hilb}}_{k + l}(h^L)} 
			\cdot
			\Big( \frac{k \cdot l}{k + l} \Big)^{\frac{n}{2}}  
			 \leq 
			1 + C  \Big( \frac{1}{k} + \frac{1}{l} \Big),
			\\
			1 - C  \Big( \frac{1}{k} + \frac{1}{l} \Big)
			\leq 
			& \frac{[\| \cdot \|_{L^{\infty}_{k, l}(X \times  X,  h^L \boxtimes h^L)}]}{{\rm{Ban}}_{k + l}^{\infty}(h^L)} 
			 \leq 
			1 + C  \Big( \frac{1}{k} + \frac{1}{l} \Big).
		\end{aligned}
		\end{equation}
		The result now follows from Theorem \ref{thm_compar_tens_prod} and (\ref{eq_as_isom1121}).
	\end{proof}
	
\subsection{Norms on the product manifold, injective and projective tensor norms}\label{sect_prod_norms}
	The main goal of this section is to study the functional-analytic properties of the normed spaces  $(H^0(X, L^k), {\rm{Ban}}_{k}^{1}(h^L))$, $(H^0(X, L^k), {\rm{Ban}}_{k}^{\infty}(h^L))$.
	From this analysis, we establish Theorem \ref{thm_compar_tens_prod}.
	\par 
	Our main idea of the proof of Theorem \ref{thm_compar_tens_prod} is to first use the general theory of injective tensor products to settle down the corresponding part of Theorem \ref{thm_compar_tens_prod}.
	Then we will obtain the version for the projective tensor products from the already established version for the injective tensor products and a careful study of duality on the normed spaces $(H^0(X, L^k), {\rm{Ban}}_{k}^{1}(h^L))$, $(H^0(X, L^k), {\rm{Ban}}_{k}^{\infty}(h^L))$.
	More precisely, the main result of this section goes as follows.
	\begin{thm}\label{thm_compar_tens_prod2}
		There are constants $c > 0$, $p_0 \in \nat$, such that for any $k \geq p_0$, the following estimate between the norms on $H^0(X, L^k)$ holds
		\begin{equation}\label{eq_compar_tens_prod}
			\frac{1}{2^n}
			\cdot
			\Big(
				1 - \frac{C}{k}
			\Big)
			\cdot
			{\rm{Ban}}_k^{1}(h^L)
			\leq
			({\rm{Ban}}_k^{\infty}(h^L))^*
			\leq
			{\rm{Ban}}_k^{1}(h^L),
		\end{equation}
		where we implicitly identified  $H^0(X, L^k)$ and  $H^0(X, L^k)^*$ using the $L^2$-scalar product.
		Similarly, for the norms over $H^0(X \times X, L^k \boxtimes L^l)$, under a similar identification, we have the following comparison
		\begin{equation}\label{eq_compar_tens_prod2}
			\frac{1}{4^n}
			\cdot
			\Big(
				1 - C  \Big( \frac{1}{k} + \frac{1}{l} \Big)
			\Big)
			\cdot
			\| \cdot \|_{L^1_{k, l}(X \times X, h^L \boxtimes h^L)}
			\leq
			\| \cdot \|_{L^{\infty}_{k, l}(X \times X, h^L \boxtimes h^L)}^*
			\leq
			\| \cdot \|_{L^1_{k, l}(X \times X, h^L \boxtimes h^L)}.
		\end{equation}
	\end{thm}
	\begin{rem}\label{rem_dual_barg}
		We remark that on the left-hand side of (\ref{eq_compar_tens_prod}), the constant $\frac{1}{2^n}$ coincides with the analogous constant obtained in the study of dual of $L^p$-Bargmann space by Gryc-Kemp \cite[Theorem 1.2]{DualBarg}, cf. also  Janson-Peetre-Rochberg \cite{HankFock} and Zhu \cite{ZhuFock} for related results.
		It seems, however, that the constant is not optimal, see \cite[Theorem 1.5 and discussion before]{DualBargSharp}.
	\end{rem}
	Before establishing Theorem \ref{thm_compar_tens_prod2}, let us see how it implies Theorem \ref{thm_compar_tens_prod}.
	\begin{proof}[Proof of Theorem \ref{thm_compar_tens_prod}.]
		Let us first establish the result for the ${\rm{Ban}}^{\infty}$ norm.
		Remark that we have the following isometric embedding 
		\begin{equation}\label{eq_h0_isom_emb}
			(H^0(X, L^k), {\rm{Ban}}_k^{\infty}(h^l)) \hookrightarrow (\ccal^0(X, L^k), \| \cdot \|_{L^{\infty}_k(X, h^L)}).
		\end{equation}
		The statement (\ref{eq_compar_tens_prod1112}) now follows from Lemmas \ref{lem_quot_proj} and \ref{lem_inj_proj_expl}. 
		\par 
		Let us establish the upper bound of (\ref{eq_compar_tens_prod111}).
		For any $f \in H^0(X, L^k)$, $g \in H^0(X, L^l)$, we have
		\begin{equation}
			\| f \otimes g \|_{L^1_{k, l}(X \times  X, h^L \boxtimes h^L)}
			=
			\| f \|_{L^1_{k}(X, h^L)}
			\cdot
			\| g \|_{L^1_{l}(X, h^L)}.
		\end{equation}
		By the triangle inequality and the definition of the projective tensor norm, the above inequality implies the upper bound of (\ref{eq_compar_tens_prod111}).
		\par 
		Now, by taking the dual of (\ref{eq_compar_tens_prod1112}) and using Lemma \ref{lem_dual_proj_inj}, we obtain
		\begin{equation}\label{eq_compar_tens_prod111232}
			\| \cdot \|_{L^{\infty}_{k, l}(X \times  X, h^L \boxtimes h^L)}^*
			=
			{\rm{Ban}}_k^{\infty}(h^L)^* \otimes_{\pi} {\rm{Ban}}_l^{\infty}(h^L)^*.
		\end{equation}
		The lower bound of (\ref{eq_compar_tens_prod111}) now follows directly from Theorem \ref{thm_compar_tens_prod2}, monotonicity of the projective tensor norm and (\ref{eq_compar_tens_prod111232}).
	\end{proof}
	To establish Theorem \ref{thm_compar_tens_prod2}, we need to recall some basic properties of the \textit{Bergman projector} $B_{k}^{X}$, which is the orthogonal projection from $L^2_k(X, h^L)$ to $H^0(X, L^k)$.
	Recall that the \textit{Bergman kernel} is defined as the Schwartz kernel, $B_k^X(x_1, x_2) \in L^k_{x_1} \otimes (L^k_{x_2})^*$, $x_1, x_2 \in X$, of the Bergman projection $B_k^X$, evaluated with respect to the Riemannian volume form $dv_X$, i.e. 
	\begin{equation}
		(B_k^X s) (x_1) = \int_X B_k^X(x_1, x_2) \cdot s(x_2) dv_X(x_2), \qquad s \in L^2_k(X, h^L).
	\end{equation}
	The first result we will need concerns the study of the Bergman kernel for distant parameters. It says that the Bergman kernel has an exponential decay.
	\begin{thm}[{Ma-Marinescu \cite[Theorem 1]{MaMarOffDiag}}]\label{thm_bk_off_diag}
		There are $c, C > 0$, $p_1 \in \nat^*$ such that for any $k \geq p_1$, $x_1, x_2 \in X$, we have
		\begin{equation}\label{eq_bk_off_diag}
			\Big|  B_k^X(x_1, x_2) \Big| \leq C k^n \cdot \exp \big(- c \sqrt{k} \cdot \dist(x_1, x_2) \big).
		\end{equation}
	\end{thm}
	\par 
	The next result we will need is about the near-diagonal behavior of the Bergman kernel.
	To state it, we will need to fix some further notation.
	\par 
	We will first introduce a special coordinate system in a neighborhood of a given point.
	Let us fix $x_0 \in X$ and an orthonormal frame $(e_1, \ldots, e_{2n})$ of $(T_{x_0}X, g^{TX}_{x_0})$.
	We identify $Z \in \real^{2n}$ with an element from $T_{x_0}X$ using this frame and define the exponential coordinates $\phi_{x_0}^X : \real^{2n} \to X$, $x_0 \in X$, as follows
	\begin{equation}\label{eq_phi_defn}
		\phi_{x_0}^{X}(Z) := \exp^{X}_{x_0}(Z).
	\end{equation}
	\par 
	Let us now construct an orthonormal trivialization of $(L, h^L)$ in a neighborhood of $x_0$.
	We fix an orthonormal vector $l_0  \in L_{x_0}$, and define the local orthonormal frame $\tilde{l}'_0$ of $L$ around $x_0$ by the parallel transport of $l_0$ with respect to the Chern connection $\nabla^L$, done along the curve $\phi_{x_0}^X(tZ)$, $t \in [0, 1]$, $Z \in T_{x_0}X$, $|Z| < r_X$, where $r_X$ is the infimum of the injectivity radius of $X$.
	Now, for $Z, Z' \in \real^{2n}$, $|Z|, |Z'| \leq \epsilon$, we will denote by $B_k^X \big(\phi_{x_0}(Z), \phi_{x_0}(Z') \big) \in \real$ the Bergman kernel, evaluated at $(\phi_{x_0}(Z), \phi_{x_0}(Z'))$, and viewed as a real number using the frame $\tilde{l}'_0$.
	\par 
	Define the function $\mathscr{P}_{n}$, cf. \cite[(4.1.84)]{MaHol}, as follows
	\begin{equation}\label{eq_berg_k_expl}
		\mathscr{P}_n(Z, Z') = \exp \Big(
			-\frac{\pi}{2} \sum_{i = 1}^{n} \big( 
				|z_i|^2 + |z'_i|^2 - 2 z_i \overline{z}'_i
			\big)
		\Big), \quad \text{for } Z, Z' \in \comp^n.
	\end{equation}	
	From \cite[\S 4.1.6]{MaHol}, $\mathscr{P}_n$ is the Schwartz kernel of the orthogonal projection from the space of $L^2$-integrable functions with respect to the Gaussian measure onto the Bargmann space.
	\par 
	Next result shows that the behavior of the Bergman kernel on any manifold $X$ for close parameters is governed by the function $\mathscr{P}_{n}$.
	\begin{thm}\label{thm_berg_off_diag}
		There are $\epsilon, c, C > 0$, $p_1 \in \nat^*$, such that for any $x_0 \in X$, $k \geq p_1$, $Z, Z' \in \real^{2n}$, $|Z|, |Z'| \leq \epsilon$, the following estimate holds
			\begin{multline}\label{eq_berg_off_diag}
				\bigg| 
						\frac{1}{k^n} B_k^X \big(\phi_{x_0}^X(Z), \phi_{x_0}^X(Z') \big)
						-					
						\mathscr{P}_{n}(\sqrt{k} Z, \sqrt{k} Z')
				\bigg|
				\\
				\leq
				C k^{-1}
				\Big(1 + \sqrt{k}|Z| + \sqrt{k} |Z'| \Big)^{3n + 2}
				\exp\Big(- c \sqrt{k} |Z - Z'| \Big).
			\end{multline}
	\end{thm}
	\begin{proof}
		The result follows from the first two terms of the asymptotic expansion of Dai-Liu-Ma \cite{DaiLiuMa} and a result of Ma-Marinescu \cite[Remark 4.1.26]{MaHol}, which states that the second term of the asymptotic expansion vanishes.
		The $\kappa$ function, see \cite[(4.1.28)]{MaHol} for a definition, doesn't appear in our expansion since our expansion only uses the two first terms of the expansion of Dai-Liu-Ma, and according to \cite[(1.2.19)]{MaHol}, the appearance of the $\kappa$-term doesn't affect them.
	\end{proof}
	We are now ready to draw the first consequence of those statements.
	\begin{prop}\label{prop_norm_berg_proj}
		For $k \in \nat$ big enough, the operator $B_k^X$ extends using the Bergman kernel as a bounded operator to the space of $L^p$-integrable sections of $L^k$, for any $p \in [1, + \infty]$. 
		Moreover, there is $C > 0$, such that for any $k \in \nat^*$, $p \in [1, + \infty]$, the following bound holds
		\begin{equation}\label{eq_norm_berg_proj}
			\big\|
				B_k^X
			\big\|_{L^p_k(X, h^L) \to L^p_k(X, h^L)}
			\leq
			2^n \Big(1 + \frac{C}{k} \Big),
		\end{equation}
		where $\| \cdot \|_{L^p_k(X, h^L) \to L^p_k(X, h^L)}$ is the operator norm of an operator from the space of $L^p$-integrable sections of $L^k$ to itself.
	\end{prop}
	\begin{rem}
		Similar results were obtained for Bargmann space, see \cite{HankFock}.
		In fact, according to \cite[Theorem 1.1]{DualBarg}, the constant $2^n$ in this case is sharp for $p = 1$.
	\end{rem}
	\begin{proof}
		From Theorems \ref{thm_bk_off_diag}, \ref{thm_berg_off_diag} and (\ref{eq_exp_int}), we conclude that there is $C > 0$, such that for any $x_0 \in X$, for $k \geq p_1$, we have
		\begin{equation}\label{eq_bound_int_berg}
		\begin{aligned}
			&
			\int_{X} 
				\big| B_k^X(x_0, x) \big| dv_X(x) 
			\leq 
			2^n \Big(1 + \frac{C}{k} \Big), 
			\\
			&
			\int_{X} 
				\big| B_k^X(x, x_0) \big| dv_X(x) 
			\leq 
			2^n \Big(1 + \frac{C}{k} \Big).
		\end{aligned}
		\end{equation}
		Directly from (\ref{eq_bound_int_berg}) and Young's inequality for integral operators, cf. \cite[Theorem 0.3.1]{SoggBook} applied for $q = p$, $r = 1$ in the notations of \cite{SoggBook}, we deduce (\ref{eq_norm_berg_proj}).
	\end{proof}
	\begin{proof}[Proof of Theorem \ref{thm_compar_tens_prod2}.]
		Instead of (\ref{eq_compar_tens_prod}), let us establish the following inequalities
		\begin{equation}\label{eq_compar_tens_prod_22}
			\frac{1}{2^n}
			{\rm{Ban}}_k^{\infty}(h^L)
			\leq
			({\rm{Ban}}_k^{1}(h^L))^*
			\leq
			{\rm{Ban}}_l^{\infty}(h^L).
		\end{equation}
		By taking duals, and using the fact that any norm on a finitely dimensional vector space is equal to its double dual, we see that (\ref{eq_compar_tens_prod_22}) is equivalent to (\ref{eq_compar_tens_prod}).
		\par 
		\begin{sloppypar}
		First of all, by the definition of the dual norm, for any $f \in H^0(X, L^k)$, we have
		\begin{equation}\label{eq_compar_tens_prod_2232}
			\|
				f
			\|_{L^1_k(X, h^L)}^*
			=
			\sup_{g \in H^0(X, L^k)} \frac{|\scal{f}{g}_{L^2_k(X, h^L)}|}{\| g \|_{L^{1}_k(X, h^L)}}.
		\end{equation}
		From this, the fact that $H^0(X, L^k)$ is a subspace of $L^1(X, L^k)$ and the well-known fact that the dual of $(L^1(X, L^k), \| \cdot \|_{L^1_k(X, h^L)})$ coincides with $(L^{\infty}(X, L^k), \| \cdot \|_{L^{\infty}_k(X, h^L)})$, we deduce the right estimate from (\ref{eq_compar_tens_prod_22}).
		\end{sloppypar}
		\par 
		Now, for a fixed $\alpha > 0$ and $f \in H^0(X, L^k)$, consider an element $g_{\alpha} := B_k^X (|f|^{\alpha} \cdot f) \in H^0(X, L^k)$.
		Then according to Proposition \ref{prop_norm_berg_proj}, we have 
		\begin{equation}\label{eq_g_al_nm1}
			\|
				g_{\alpha}
			\|_{L^{1}_k(X, h^L)}
			\leq
			2^n
			\Big(
				1 + \frac{C}{k}
			\Big)
			\cdot
			\int_X
			|f|^{1 + \alpha}(x) dv_X(x).
		\end{equation}
		Remark also that we have
		\begin{equation}\label{eq_g_al_nm2}
			\scal{f}{g_{\alpha}}_{L^2_k(X, h^L)}
			=
			\scal{f}{|f|^{\alpha} \cdot f}_{L^2_k(X, h^L)}
			=
			\int_X
			|f|^{2 + \alpha}(x) dv_X(x).
		\end{equation}
		In particular, from (\ref{eq_g_al_nm2}), we see that $g_{\alpha} \neq 0$.
		Recall, finally, that for any continuous non-zero function $h : X \to \real$, we have
		\begin{equation}\label{eq_g_al_nm3}
			\lim_{\alpha \to +\infty}
			\frac{\int_X
			|h|^{2 + \alpha}(x) dv_X(x)}
			{\int_X
			|h|^{1 + \alpha}(x) dv_X(x)}
			=
			\sup_{x \in X} |h|(x).
		\end{equation}
		The left estimate from (\ref{eq_compar_tens_prod_22}) now follows from (\ref{eq_g_al_nm1}), (\ref{eq_g_al_nm2}) and (\ref{eq_g_al_nm3}) by putting $g := g_{\alpha}$ in (\ref{eq_compar_tens_prod_2232}) and letting $\alpha \to \infty$.
		\par 
		The second estimate from Theorem \ref{thm_compar_tens_prod2} is obtained in an analogous way.
	\end{proof}
	
\subsection{Bounds on products of holomorphic sections, a proof of Theorem \ref{thm_mult_well_def}}\label{sect_mult_def}
	The main goal of this section is to give a bound on the norm of a product of two holomorphic sections in terms of their norms. In particular, we prove Theorem \ref{thm_mult_well_def}.
	The main caveat of the proof is an introduction of a sequence of operators, called multiplicative defect.
	They would allow us to extend our methods from \cite{FinOTAs}, \cite{FinToeplImm}, \cite{FinOTRed} in the current setting.
	\par 
	Before all, let us prove that Theorem \ref{thm_mult_well_def} holds for $L^1$-norm.
	For this we will rely solely on the results on the study of the asymptotic behavior of the Bergman kernel from Section \ref{sect_prod_norms}.
	\begin{sloppypar}
	\begin{proof}[Proof of Theorem \ref{thm_mult_well_def} for the $L^1$-norm]
		Let us first recall that the Bergman projector $B_{k, l}^{X \times X}$ is defined as the orthogonal projection from $L^2(X \times X, L^k \boxtimes L^l)$ to $H^0(X \times X, L^k \boxtimes L^l)$.
		Clearly, under the identification
		\begin{equation}\label{eq_l2_times_otimes}
			L^2(X \times X, L^k \boxtimes L^l)
			\to
			L^2(X, L^k)
			\otimes
			L^2(X, L^l),
		\end{equation}
		where $\otimes$ means here the tensor product of Hilbert spaces, and the identification (\ref{eq_nat_isom1}), we can write 
		\begin{equation}\label{eq_bergm_prod}
			B_{k, l}^{X \times X} = B_{k}^{X} \otimes B_{l}^{X}.
		\end{equation}
		\par
		We extend ${\rm{Mult}}_{k, l}$ to $L^2(X, L^k) \otimes L^2(X, L^l)$ by precomposing it with the Bergman projection.
		Then, under the isomorphisms from (\ref{eq_comm_diag}) and (\ref{eq_l2_times_otimes}), we have the following identity
		\begin{equation}\label{eq_mult_bergm_relat}
			{\rm{Mult}}_{k, l} = \res_{\Delta} \circ B_{k, l}^{X \times X}.
		\end{equation}
		\par 
		Similarly to (\ref{eq_bound_int_berg}), from (\ref{eq_mult_bergm_relat}), we conclude that there are $C > 0$, $p_1 \in \nat$, such that for any $x_0, x_1 \in X$, for $k, l \geq p_1$, we have
		\begin{equation}\label{eq_bound_int_mult}
		\begin{aligned}
			&
			\int_{X} 
				\Big| {\rm{Mult}}_{k, l} \big(x, (x_0, x_1) \big) \Big|  dv_X(x) 
			\leq 
			2^n  \Big( \frac{k \cdot l}{\max \{ k, l \}} \Big)^n \cdot \Big(1 + \frac{C}{k} + \frac{C}{l} \Big), 
			\\
			&
			\int_{X \times X} 
				\Big| {\rm{Mult}}_{k, l} \big(x_0, (x, x') \big) \Big| dv_X(x)dv_X(x') 
			\leq 
			4^n \Big(1 + \frac{C}{k} + \frac{C}{l} \Big).
		\end{aligned}
		\end{equation}
		Directly from (\ref{eq_bound_int_mult}) and Young's inequality for integral operators, cf. \cite[Theorem 0.3.1]{SoggBook} applied for $q = p = 1$, $r = 1$ in the notations of \cite{SoggBook}, we deduce the second inequality from (\ref{eq_mult_well_def2}).
		\par 
		Let us now show that the first part of the statement for $L^1$-norm follows from the second one.
		Remark that for any $f \in H^0(X, L^k)$, $g \in H^0(X, L^l)$, we have
		\begin{equation}\label{eq_prod_l2_nm}
			\big \| 
				f \otimes g
			\big \|_{L^1_{k, l}(X \times X, h^L \boxtimes h^L)}
			=
			\big \| 
				f
			\big \|_{L^1_k(X, h^L)}
			\cdot
			\big \| 
				g
			\big \|_{L^1_l(X, h^L)}.
		\end{equation}
		By the definition of the operator norm and (\ref{eq_prod_l2_nm}), we deduce that
		\begin{equation}\label{eq_prod_l2_nm333}
			\big \| 
				f \cdot g
			\big \|_{L^1_{k + l}(X, h^L)}
			\leq
			\|
				{\rm{Mult}}_{k, l}
			\|_{L^1_{k, l}(X \times  X, h^L \boxtimes h^L) \to L^1_{k + l}(X, h^L)}
			\cdot
			\big \| 
				f \otimes g
			\big \|_{L^1_{k, l}(X \times X, h^L \boxtimes h^L)}.
		\end{equation}
		We conclude by (\ref{eq_mult_well_def2}), (\ref{eq_prod_l2_nm}) and (\ref{eq_prod_l2_nm333}).
	\end{proof}
	\end{sloppypar}
	\par 
	Now, to prove Theorem \ref{thm_mult_well_def} for the $L^2$-norm, more work has to be done.
	This is due to the fact that the bound provided by the Young's inequality will no longer coincide asymptotically with the optimal bound.
	To remedy this problem, using the surjectivity statement from Proposition \ref{prop_mult_surj}, we define the \textit{optimal decomposition map} 
	\begin{equation}\label{eq_opt_dec}
		{\rm{Dec}}_{k, l} : H^0(X, L^{k + l})
		\to
		H^0(X, L^k) \otimes H^0(X, L^l),
	\end{equation}
	as follows.
	We let ${\rm{Dec}}_{k, l}(f) = h$ if $h$ satisfies ${\rm{Mult}}_{k, l}(h) = f$, and $h$ has the minimal $L^2$-norm among all $h \in H^0(X, L^k) \otimes H^0(X, L^l)$, verifying such an identity (the uniqueness of such $h$ follows from the non-degeneracy of the $L^2$-Hermitian product).
	Clearly, from the commutative square (\ref{eq_comm_diag}), we see that the optimal decomposition map is just a version of the optimal extension map, studied in \cite{FinOTAs}, associated for the pair $(X \times X, \Delta)$. 
	The difference between our situation here from \cite{FinOTAs} is that there we consider only the tensor powers and not $\boxtimes$ products.
	This being said, most of our techniques from \cite{FinOTRed} transfer in the current setting almost directly.
	\par 
	We will now use the methods of \cite{FinOTRed} to study ${\rm{Dec}}_{k, l}$.
	Our first goal is to prove the existence of a sequence of operators, relating the restriction maps with optimal decomposition maps.
	\begin{thm}\label{thm_mult_def}
		There is $p_1 \in \nat^*$, such that for any $k, l \geq p_1$, there is a unique operator $A_{k, l} \in \enmr{H^{0}(X, L^{k + l})}$, verifying under the isomorphisms (\ref{eq_nat_isom12}) and (\ref{eq_nat_isom2}) the following identity
		\begin{equation}\label{eq_resp_ap_lem}
			(\res_{\Delta} \circ B_{k, l}^{X \times X})^{*} = {\rm{Dec}}_{k, l} \circ A_{k, l}.
		\end{equation}
	\end{thm}
	\begin{rem}
		The sequence of operators $A_{k, l}$, will be later called \textit{multiplicative defect}.
		An analogous sequence of operators for more general immersions (and tensor powers of the line bundles instead of $\boxtimes$-products) was introduced in \cite[Theorem 4.3]{FinToeplImm} in their relation with the problem of transitivity of optimal holomorphic extensions.
		It was later studied in \cite[Theorem 4.1]{FinToeplImm} in a more general situation of holomorphic jets along submanifold, and used as a tool to extend our previous work \cite{FinOTAs} on the semiclassical extension theorem for holomorphic jets instead of sections.
	\end{rem}
	\begin{proof}
		It follows from the equality of kernels and images of $(\res_{\Delta} \circ B_{k, l}^{X \times X})^{*}$ and ${\rm{Dec}}_{k, l}$, which relies on Proposition \ref{prop_mult_surj}.
		See \cite[Theorem 4.3]{FinToeplImm} or  \cite[Theorem 4.1]{FinOTRed} for details.		
	\end{proof}
	Now, we will show that, $A_{k, l}$ coincides asymptotically with the identity operator up to a multiplication by an explicit constant.
	At least for $k = l$, this is quite an expected result.
	\par 
	In fact, from the commutative diagram (\ref{eq_comm_diag}) and (\ref{eq_resp_ap_lem}), the operator $A_{k, l}$ relates the extension and the adjoint of the restriction operators. 
	But from the calculations in \cite[(3.10)]{FinToeplImm} or \cite[(3.15)]{FinOTRed}, we know that those operators coincide (in a non-asymptotic sense) for the model Bargmann space.
	Moreover, by the results of \cite[Theorem 1.6]{FinOTAs} and \cite[Lemma 5.9]{FinOTRed}, we know that the extension and restriction operators for general manifolds coincide asymptotically with those model operators after an appropriate rescaling. 
	So over general manifolds the extension and the adjoint of the restriction operators should asymptotically coincide as well.
	\par 
	The following result shows that despite the fact that we do not posses a general version of extension theorem, which would cover not only the tensor powers, but also line bundles of the form $L^k \boxtimes L^l$, the property of $A_{k, l}$, described above, persists.
	\begin{thm}\label{thm_akl_comp}
		There are $p_1 \in \nat^*$, $C > 0$, such that for any $k, l \geq p_1$, we have
		\begin{equation}\label{eq_akl_comp}
			\Big\|
			A_{k, l}
			-
			\Big( \frac{k \cdot l}{k + l} \Big)^{n}
			\cdot
			{\rm{Id}}_{H^0(X, L^{k + l})}
			\Big\|
			\leq
			C
			\cdot
			\Big(
				\frac{1}{k} + \frac{1}{l}
			\Big)
			\cdot
			\Big( \frac{k \cdot l}{k + l} \Big)^n,
		\end{equation}
		where the operator norm of the operators on $H^0(X, L^{k + l})$ is taken with respect to ${\rm{Hilb}}_{k + l}(h^L)$.
	\end{thm}
	In the proof of Theorem \ref{thm_akl_comp}, the following result will be of utmost importance.
	To state it, we extend the domain of $A_{k, l}$ to $L^2_{k + l}(X, h^L)$ by precomposing it with the Bergman projection.
	Then $A_{k, l}$ is a smoothing operator and hence has a smooth Schwartz kernel, $A_{k, l}(x_1, x_2)$, $x_1, x_2 \in X$,  with respect to $dv_X$.
	Next lemma provides a comprehensive study of this Schwartz kernel.
	\begin{lem}\label{lem_mult_def_exp_dec_diag}
		There are $c, C > 0$, $p_1 \in \nat^*$ such that for any $k, l \geq p_1$, $x_1, x_2 \in X$, we have
		\begin{equation}\label{eq_akl_off_diag}
			\Big|  A_{k, l}(x_1, x_2) \Big| \leq C (k \cdot l)^{n} \cdot \exp \big(- c ( \sqrt{k} + \sqrt{l} ) \cdot \dist(x_1, x_2) \big).
		\end{equation}
		\par 
		Moreover, there are $\epsilon, c, C > 0$, $p_1 \in \nat^*$, such that for any $x_0 \in X$, $k, l \geq p_1$, $Z, Z' \in \real^{2n}$, $|Z|, |Z'| \leq \epsilon$, we have
		\begin{multline}\label{eq_akl_off_diag2}
			\bigg| 
					A_{k, l} \big(\phi_{x_0}^X(Z), \phi_{x_0}^X(Z') \big)
					-	
					\Big( \frac{k \cdot l}{k + l} \Big)^{n}  	
					B_{k + l}^X \big(\phi_{x_0}^X(Z), \phi_{x_0}^X(Z') \big)
			\bigg|
			\\
			\leq
			C (k \cdot l)^n \cdot \Big( \frac{1}{k} + \frac{1}{l} \Big) \cdot
			\Big(1 + \sqrt{k + l}|Z| + \sqrt{k + l} |Z'| \Big)^{6n + 4}
			\cdot
			\\
			\cdot
			\exp\Big(- c (\sqrt{k} +  \sqrt{l}) \cdot  |Z - Z'| \Big).
		\end{multline}
	\end{lem}
	\begin{proof}
		First of all, from (\ref{eq_resp_ap_lem}), we see that the following identity holds
		\begin{equation}\label{eq_akl_form}
			A_{k, l}
			=
			\res_{\Delta} \circ (\res_{\Delta} \circ B_{k, l}^{X \times X})^{*}.
		\end{equation}
		Remark also the trivial fact that the embedding $\Delta \hookrightarrow X \times X$ is totally geodesic once we endow $X \times X$ with the product metric.
		From this, we deduce that 
		\begin{equation}\label{eq_dist_equiv}
			d_{X \times X}(z, z') = d_{\Delta}(z, z'), \qquad \text{for } z, z' \in \Delta,
		\end{equation}
		where the metric on $\Delta$ is taken with respect to the Kähler form $\omega_{1, 1}$.
		Also, for any $x_0 \in X$, $Z \in \real^{2n}$, $|Z| < r_X$, we have
		\begin{equation}\label{eq_phi_equiv}
			\phi_{x_0 \times x_0}^{X \times X}(Z, Z)
			=
			(\phi_{x_0}^{X}(Z), \phi_{x_0}^{X}(Z)).
		\end{equation}
		From Theorem \ref{thm_bk_off_diag}, (\ref{eq_bergm_prod}), (\ref{eq_akl_form}) and (\ref{eq_dist_equiv}), we deduce (\ref{eq_akl_off_diag}).
		\par 
		Let us now prove (\ref{eq_akl_off_diag2}). 
		From Theorem \ref{thm_berg_off_diag}, (\ref{eq_bergm_prod}), (\ref{eq_akl_form}) and (\ref{eq_phi_equiv}), we see that there are $\epsilon, c, C > 0$, $p_1 \in \nat^*$, such that for any $x_0 \in X$, $k, l \geq p_1$, $Z, Z' \in \real^{2n}$, $|Z|, |Z'| \leq \epsilon$, the following estimate holds
		\begin{multline}\label{eq_akl_off_diag1111}
			\bigg| 
					A_{k, l} \big(\phi_{x_0}^X(Z), \phi_{x_0}^X(Z') \big)
					-					
					(k \cdot l)^n \cdot \mathscr{P}_{n}(\sqrt{k} Z, \sqrt{k} Z')
					\cdot \mathscr{P}_{n}(\sqrt{l} Z, \sqrt{l} Z')
			\bigg|
			\\
			\leq
			C \Big( \frac{1}{k} + \frac{1}{l} \Big) \cdot (k \cdot l)^n \cdot
			\Big(1 + \sqrt{k + l}|Z| + \sqrt{k + l} |Z'| \Big)^{6n + 4}
			\cdot
			\\
			\cdot
			\exp\Big(- c  ( \sqrt{k} +  \sqrt{l}) \cdot  |Z - Z'| \Big).
		\end{multline}
		We now remark that from (\ref{eq_berg_k_expl}), we have
		\begin{equation}\label{eq_pn_mult_prop11}
			\mathscr{P}_{n}(\sqrt{k} Z, \sqrt{k} Z') \cdot \mathscr{P}_{n}(\sqrt{l} Z, \sqrt{l} Z')
			=
			\mathscr{P}_{n}(\sqrt{k + l} Z, \sqrt{k + l} Z').
		\end{equation}
		From Theorem \ref{thm_berg_off_diag}, (\ref{eq_akl_off_diag1111}) and (\ref{eq_pn_mult_prop11}), we conclude that (\ref{eq_akl_off_diag2}) holds.
	\end{proof}
	\begin{proof}[Proof of Theorem \ref{thm_akl_comp}]
		Remark first, cf. \cite[Proposition 3.2]{FinOTAs}, that there is $C > 0$, such that for any $x_0 \in X$, we have the following bound
		\begin{equation}\label{eq_exp_int}
			\int_{X} \exp \big(- c (\sqrt{k} +  \sqrt{l}) \cdot \dist(x_0, x) \big) dv_{X}(x) < \frac{C}{(k + l)^{n}}.
		\end{equation}
		\par 
		From Theorem \ref{thm_bk_off_diag}, (\ref{eq_akl_off_diag2}) and (\ref{eq_exp_int}), we conclude that there are $C > 0$, $p_1 \in \nat^*$, such that for any $x_0 \in X$, for $k, l \geq p_1$, we have
		\begin{equation}\label{eq_bound_int_berg2}
		\begin{aligned}
			&
			\int_{X} 
			\Big| \Big(
				A_{k, l} - \Big( \frac{k \cdot l}{k + l} \Big)^{n} 	
				B_{k + l}^X 
			\Big) (x_0, x) \Big|  dv_X(x) 
			\leq 
			C \Big( \frac{1}{k} + \frac{1}{l} \Big) \cdot \Big( \frac{k \cdot l}{k + l} \Big)^n, 
			\\
			&
			\int_{X} 
			\Big| \Big(
				A_{k, l} - \Big( \frac{k \cdot l}{k + l} \Big)^{n} 	 	
				B_{k + l}^X 
			\Big) (x, x_0)\Big|  dv_X(x) 
			\leq 
			C \Big( \frac{1}{k} + \frac{1}{l} \Big) \cdot \Big( \frac{k \cdot l}{k + l} \Big)^n.
		\end{aligned}
		\end{equation}
		Directly from (\ref{eq_bound_int_berg2}) and Young's inequality for integral operators, cf. \cite[Theorem 0.3.1]{SoggBook}, applied for $p, q = 2$, $r = 1$ in the notations of \cite{SoggBook}, we deduce the result.
	\end{proof}
	We are now finally ready to finish the proof of the main result of this section.
	\par 
	\begin{proof}[Proof of Theorem \ref{thm_mult_well_def} for the $L^2$-norm]
		In this proof all the operator norms are taken with respect to the $L^2$-norm.
		First of all, similarly to the proof of Theorem \ref{thm_mult_well_def} for the $L^1$-norm, it suffices to establish the second part of the statement. We will concentrate on this from now on.
		Directly from Theorem \ref{thm_akl_comp}, we see that there are $p_1 \in \nat^*$, $C > 0$, such that for any $k, l \geq p_1$, we have
		\begin{equation}\label{eq_akp_norm}
			\Big|
				\big\|
				A_{k, l} 
				\big\|
				-
				\Big( \frac{k \cdot l}{k + l} \Big)^{n}
			\Big|
			\leq
			C \Big( \frac{1}{k} + \frac{1}{l} \Big) \cdot \Big( \frac{k \cdot l}{k + l} \Big)^n.
		\end{equation}
		From Theorem \ref{thm_mult_def}, under the isomorphisms from (\ref{eq_comm_diag}), we have the following identity
		\begin{equation}\label{eq_comp_ext_oper}
			{\rm{Mult}}_{k, l} \circ ({\rm{Mult}}_{k, l})^{*} = A_{k, l}.
		\end{equation}
		Clearly, we have $ \| {\rm{Mult}}_{k, l} \circ ({\rm{Mult}}_{k, l})^{*}\| = \|  {\rm{Mult}}_{k, l} \|^2$.
		The first estimate from (\ref{eq_mult_well_def2}) now follows from this observation, (\ref{eq_akp_norm}) and (\ref{eq_comp_ext_oper}).
	\end{proof}

	\subsection{Asymptotics of optimal decomposition map, a proof of Theorem \ref{thm_mult_surj}}\label{sect_as_opt_ext_ban}
	The main goal of this section is to study more precisely the asymptotics of optimal decomposition map, introduced in (\ref{eq_opt_dec}).
	Through this study we will establish Theorem \ref{thm_mult_surj}.
	\par 
	We will begin by showing that for the $L^2$-norm, Theorem \ref{thm_mult_surj} already follows from the analysis on the multiplicative defect from Section \ref{sect_mult_def}.
	\begin{proof}[Proof of Theorem \ref{thm_mult_surj} for the $L^2$-norm]
		Directly from Theorem \ref{thm_akl_comp}, we see that there is $p_1 \in \nat^*$, such that $A_{k, l}$ is invertible on $H^0(X, L^{k + l})$ for any $k, l \geq p_1$.
		We, moreover, deduce that there are $p_1 \in \nat^*$, $C > 0$, such that for any $k, l \geq p_1$, the following inequality holds
		\begin{equation}\label{eq_akp_norm2}
			\Big|
				\big\|
				A_{k, l}^{-1}
				\big\|
				-
				\Big( \frac{k + l}{k \cdot l} \Big)^{n}
			\Big|
			\leq
			C \Big( \frac{1}{k} + \frac{1}{l} \Big) \cdot \Big( \frac{k + l}{k \cdot l} \Big)^n,
		\end{equation}
		where the operator norm here (as well as below in this proof) is taken with respect to the $L^2$-norm.
		From Theorem \ref{thm_mult_def}, under the isomorphisms from (\ref{eq_comm_diag}), we have the following identity
		\begin{equation}\label{eq_comp_ext_oper2}
			({\rm{Dec}}_{k, l})^* \circ {\rm{Dec}}_{k, l} = \big( (A_{k, l})^* \big)^{-1}.
		\end{equation}
		Clearly, we have $ \| ({\rm{Dec}}_{k, l})^* \circ {\rm{Dec}}_{k, l}  \| = \| {\rm{Dec}}_{k, l}  \|^2$.
		Theorem \ref{thm_mult_surj} for the $L^2$-norm now follows from this observation, (\ref{eq_akp_norm2}) and (\ref{eq_comp_ext_oper2}).
	\end{proof}
	\par 
	To establish Theorem \ref{thm_mult_surj} for $L^1$ and $L^{\infty}$-norms, we have to do more refined analysis.
	This is due to the fact that $L^1$ and $L^{\infty}$-norms are not Hermitian ones, and so the above argument using the dual mappings wouldn't work.
	\par 
	To overcome this, we develop analysis similar in spirit to \cite[proof of Theorem 1]{FinOTRed}, yet technically a bit more involved.
	In particular, we will give an asymptotic formula for the optimal decomposition map (and not only for its norm or the norm of the inverse as it was done in proof of Theorems \ref{thm_mult_well_def} and \ref{thm_mult_surj} for the $L^2$-norm).
	To describe this asymptotics, we introduce some further notation to study the geometry of the diagonal embedding.
	\par 
	For $\alpha, \beta \geq 0$, not vanishing at the same time, let us consider the orthogonal complement $N_{\alpha, \beta}$ of $T \Delta$ in $T(X \times X)$ with respect to the non-negative $(1, 1)$-form
	\begin{equation}
		\omega_{\alpha, \beta} := \alpha \pi_1^* \omega + \beta \pi_2^* \omega
	\end{equation}
	on $X \times X$, where $\pi_1, \pi_2$ are the projections from $X \times X$ to the first and the second factor respectively.
	Clearly, $N_{\alpha, \beta}$ is well-defined even if either $\alpha$ or $\beta$ is equal to $0$, as $\omega_{\alpha, \beta}$ is always positive-definite along $T \Delta$ as long as $\alpha, \beta$ do not vanish at the same time. Moreover, $N_{\alpha, \beta}$ has constant dimension along $\Delta$ and hence forms a (smooth) vector bundle over $\Delta$.
	\par 
	For $y = (x_0, x_0) \in \Delta$, $Z_N \in N_{\alpha, \beta}$, let $\real \ni t \mapsto \exp_y^{X \times X}(tZ_N) \in X$ be the geodesic in $X \times X$ in direction $Z_N$. The geodesic is taken with respect to the Kähler form $\omega_{1, 1}$.
	In other words, for $Z_N = (Z_1, Z_2)$, we have $\exp_y^{X \times X}(Z_N) = (\exp_{x_0}^X(Z_1), \exp_{x_0}^X(Z_2))$.
	\par 
	Clearly, for a certain $r_{\perp} > 0$, the above exponential map induces a diffeomorphism of $r_{\perp}$-neighborhood of the zero section in $N_{\alpha, \beta}$ (with respect to the norm $\| \cdot \|$ induced by $\omega_{1, 1}$) with a tubular neighborhood $U_{\alpha, \beta}$ of $\Delta$ in $X \times X$.
	Such $r_{\perp}$ can be chosen independently of $\alpha, \beta \geq 0$ since $N_{\alpha, \beta}$ can be seen as a piecewise smooth vector bundle over $\Delta \times \Delta_1$, where $\Delta_1 := \{ (\gamma, \delta) \in \real^2 : |\gamma| + |\delta| = 1 \}$ and the projection over $\Delta_1$ is induced by $(\alpha, \beta) \mapsto \frac{1}{|\alpha| + |\beta|} (\alpha, \beta)$.
	\par 
	We denote by $\pi_{\alpha, \beta} : U_{\alpha, \beta} \to \Delta$ the natural projection $(y, Z_N) \mapsto y$. 
	Over $U_{\alpha, \beta}$, we identify $\pi_i^* L$, $i = 1, 2$, to $\pi_{\alpha, \beta}^* (\pi_i^* L|_{\Delta})$ by the parallel transport with respect to the respective Chern connections along the geodesic $[0, 1] \ni t \mapsto (y,t  Z_N) \in X$, $|Z_N| < r_{\perp}$, previously described. 
	This gives us a local trivialization in normal directions of the vector bundle $L^k \boxtimes L^l$ in the neighborhood of $\Delta$.
	\par 
	\begin{sloppypar}
	Using the above identifications, we define the operator ${\rm{Dec}}_{k, l}^0 : L^2(X, L^{k + l}) \to L^2(X \times X, L^k \boxtimes L^l)$ as follows.
	For $g \in L^2(X, L^{k + l})$, we let $({\rm{Dec}}_{k, l}^0 g)(x) = 0$, $x \notin U_{k, l}$, and in $U_{k, l}$, we put
	\begin{equation}\label{eq_ext0_op}
		({\rm{Dec}}_{k, l}^{0} g)(y, Z_N) = (B_{k + l}^X g)(y) \exp \Big(- \frac{\pi}{2} |Z_N|_{k, l}^2 \Big) \rho \Big(  \frac{|Z_N|}{r_{\perp}} \Big),
	\end{equation}
	where $| \cdot |_{\alpha, \beta}$ is the norm induced on $N_{\alpha, \beta}$ by the Kähler form $\omega_{\alpha, \beta}$ and $\rho : \real_{+} \to [0, 1]$ satisfies
	\begin{equation}\label{defn_rho_fun}
		\rho(x) =
		\begin{cases}
			1, \quad \text{for $x < \frac{1}{4}$},\\
			0, \quad \text{for $x > \frac{1}{2}$}.
		\end{cases}
	\end{equation}
	Now, for $g \in H^0(X, L^{k + l})$, the section ${\rm{Dec}}_{k, l}^{0} g$ satisfies $({\rm{Dec}}_{k, l}^{0} g)|_{\Delta} = g$, but ${\rm{Dec}}_{k, l}^{0} g$ is not holomorphic over $X \times X$ in general.
	The main result of this sections, nevertheless, says that this operator approximates well the optimal decomposition map.
	\end{sloppypar}
	\begin{thm}\label{thm_high_term_ext}
		There are $C > 0$, $p_1 \in \nat^*$, such that for any $k, l \geq p_1$, we have 
		\begin{equation}\label{eq_ext_as}
		\begin{aligned}
			& \Big\| {\rm{Dec}}_{k, l} - {\rm{Dec}}_{k, l}^{0} \Big\|_{L^1_{k + l}(X, h^L) \to L^1_{k, l}(X \times  X, h^L \boxtimes h^L)} \leq  C  \Big( \frac{1}{k} + \frac{1}{l} \Big) \cdot \Big(\frac{k + l}{k \cdot l}\Big)^n,
			\\
			& \Big\| {\rm{Dec}}_{k, l} - {\rm{Dec}}_{k, l}^{0} \Big\|_{L^{\infty}_{k + l}(X, h^L) \to L^{\infty}_{k, l}(X \times  X, h^L \boxtimes h^L)} \leq C  \Big( \frac{1}{k} + \frac{1}{l} \Big).
		\end{aligned}
		\end{equation}
	\end{thm} 
	\begin{rem}
		For the $L^2$-norm, an analogous result can be proved by similar methods. 
		It would give us an alternative proof of Theorem \ref{thm_mult_surj} for the $L^2$-norm.
	\end{rem}
	Before establishing Theorem \ref{thm_high_term_ext}, let us see how useful it is in proving Theorem \ref{thm_mult_surj}.
	\begin{proof}[Proof of Theorem \ref{thm_mult_surj} for $L^1$ and $L^{\infty}$-norms]
	\begin{sloppypar}
		Directly from (\ref{eq_ext0_op}), the fact that $({\rm{Dec}}_{k, l}^{0} g)|_{\Delta} = g$ and $\exp(-x) \leq 1$ for $x \geq 0$, we see that for any $g \in H^0(X, L^{k + l})$, we have
		\begin{equation}\label{eq_inf_exxt_triv_bnd}
			\big\| {\rm{Dec}}_{k, l}^{0} g \big\|_{L^{\infty}_{k, l}(X \times  X, h^L \boxtimes h^L)}
			=
			\big\| g \big\|_{L^{\infty}_{k + l}(X, h^L)}.
		\end{equation}
		Then from the second bound of (\ref{eq_ext_as}) and (\ref{eq_inf_exxt_triv_bnd}), we establish that Theorem \ref{thm_mult_surj} holds for the $L^{\infty}$-norm, and one can moreover take $h := {\rm{Dec}}_{k, l} f$ in the notations of Theorem \ref{thm_mult_surj}.
	\end{sloppypar}
		\par 
		Let us now treat the $L^1$-norm.
		Remark first that the volume form $dv_{k, l}$ on $X \times X$ associated to $\omega_{k, l}$ relates to $dv_{1, 1}$ as follows
		\begin{equation}
			dv_{k, l}
			=
			(k \cdot l)^n \cdot dv_{1, 1}.
		\end{equation}
		Remark also that the restriction of $\omega_{k, l}$ to $\Delta$ coincides with $(k + l) \cdot \omega$ under the isomorphism (\ref{eq_nat_isom12}).
		Finally, we see that for the Lebesgue measure $d_{k, l} Z_N$ on $N_{k, l, y}$, calculated with respect to the norm on $N_{k, l}$ induced by $\omega_{k, l}$, we have
		\begin{equation}\label{eq_inf_exxt_triv_bnd2}
			\int_{Z_N \in N_{k, l, y}} \exp \Big(- \frac{\pi}{2} |Z_N|_{k, l}^2 \Big) d_{k, l}Z_N
			=
			2^n.
		\end{equation}
		Directly from those remarks, we see that there are $p_1 \in \nat$, $C > 0$, such that for any $k, l \geq p_1$, $g \in H^0(X, L^{k + l})$, we have
		\begin{multline}\label{eq_inf_exxt_triv_bnd3}
			\Big(
				1 - C \Big( \frac{1}{k} + \frac{1}{l} \Big)
			\Big)
			\cdot
			2^n \cdot \Big( \frac{k + l}{k \cdot l} \Big)^n
			\cdot
			\big\| g \big\|_{L^1_{k + l}(X, h^L)}
			\leq
			\big\| {\rm{Dec}}_{k, l}^{0} g \big\|_{L^1_{k, l}(X \times  X, h^L \boxtimes h^L)}
			\\
			\leq
			\Big(
				1 + C \Big( \frac{1}{k} + \frac{1}{l} \Big)
			\Big)
			\cdot
			2^n \cdot \Big( \frac{k + l}{k \cdot l} \Big)^n
			\cdot
			\big\| g \big\|_{L^1_{k + l}(X, h^L)}.
		\end{multline}
		We conclude by the first inequality from Theorem \ref{thm_high_term_ext} and (\ref{eq_inf_exxt_triv_bnd3}) that Theorem \ref{thm_mult_surj} holds for the $L^1$-norm, and one can moreover take $h := {\rm{Dec}}_{k, l} f$ in the notations of Theorem \ref{thm_mult_surj}.
	\end{proof}
	To establish Theorem \ref{thm_high_term_ext}, we proceed in several steps.
	First, we show that the Schwartz kernels of both operators, ${\rm{Dec}}_{k, l}$ and ${\rm{Dec}}_{k, l}^0$, have exponential decay.
	Second, we study the near-diagonal behavior of the Schwartz kernels of those operators and infer that they have the same principal asymptotic terms. 
	This would imply Theorem \ref{thm_high_term_ext} by some standard techniques.
	\par 
	More precisely, in the first step we prove the following exponential decay estimates.
	\begin{lem}\label{lem_ext_exp_dc}
		There are $c, C > 0$, $p_1 \in \nat^*$, such that for any $k, l \geq p_1$, $x_0, x_1, x_2 \in X$, the following estimate holds
		\begin{equation}\label{eq_ext_exp_dc}
			\Big| R_{k, l} \big( (x_1, x_2), x_0 \big) \Big| 
			\leq C (k + l)^n  \cdot \exp \Big(- c \cdot \big( \sqrt{k} \dist(x_1, x_0) + \sqrt{l} \dist(x_2, x_0) \big) \Big),
		\end{equation}
		where $R_{k, l}$ is ether ${\rm{Dec}}_{k, l}$ or ${\rm{Dec}}_{k, l}^0$.
	\end{lem}
	\begin{proof}
		Directly from Theorem \ref{thm_akl_comp}, we see that there is $p_1 \in \nat^*$, such that $A_{k, l}$ is invertible on $H^0(X, L^{k + l})$ for any $k, l \geq p_1$.
		From this and Theorem \ref{thm_mult_def}, we conclude that
		\begin{equation}\label{eq_dec_form_inv_a}
			{\rm{Dec}}_{k, l}
			=
			(\res_{\Delta} \circ B_{k, l}^{X \times X})^{*} \circ A_{k, l}^{-1}.
		\end{equation}
		From Theorem \ref{thm_bk_off_diag}, Lemma \ref{lem_mult_def_exp_dec_diag} and the usual laws of composition and inversion of operators with exponential decay, cf. \cite[Lemma 3.1]{FinOTAs}, \cite[Lemma 4.5]{FinToeplImm}, we deduce (\ref{eq_ext_exp_dc}) for $R_{k, l} = {\rm{Dec}}_{k, l}$.
		\par 
		Directly from Theorem \ref{thm_bk_off_diag} and (\ref{eq_ext0_op}), we see that 
		\begin{equation}\label{eq_dec_triv_exp_kl}
			\Big| {\rm{Dec}}_{k, l}^0 \big( (x_1, x_2), x_0 \big) \Big| 
			\\
			\leq C (k + l)^n \cdot \exp \Big(- c \cdot \dist_{k, l} \big((x_1, x_2), x_0 \big) \Big),
		\end{equation}
		where $\dist_{k, l}((x_1, x_2), x_0)$ is the distance between points $(x_1, x_2)$ and $(x_0, x_0)$ in $X \times X$ with respect to the Kähler form $\omega_{k, l}$.
		An easy verification shows that
		\begin{equation}
			2 \dist_{k, l}((x_1, x_2), x_0)
			\geq
			\sqrt{k} \dist(x_1, x_0) + \sqrt{l} \dist(x_2, x_0).
		\end{equation}
		This with (\ref{eq_dec_triv_exp_kl}) clearly finish the proof.
	\end{proof}
	\par
	Theorem \ref{lem_ext_exp_dc} implies that to understand fully the asymptotics of the Schwartz kernel of the optimal decomposition map, it suffices to do so in a neighborhood of a fixed point $(x_0, x_0) \in \Delta \subset X \times X$.
	Our next result calculates the asymptotics of this Schwartz kernel.
	\par 
	To state this result precisely, we need to introduce a version of \textit{Fermi coordinate system} in $X \times X$ around $\Delta$.
	We fix $k, l \in \nat^*$ and $y_0 \in X$.
	Let $P^N_{k, l} : T(X \times X)|_{\Delta} \to N_{k, l}$, be the orthogonal projection.
	For $Z = (Z_H, Z_N)$, $Z_H \in T_{y_0} X$, $Z_N \in N_{k, l, y_0}$, $|Z_H| \leq r_X$, $|Z_N| \leq r_{\perp}$, we define a coordinate system $\psi_{y_0}^{k, l} : B_0^{T_{y_0} X}(r_X) \times B_0^{N_{k, l, y_0}}(r_{\perp}) \to X \times X$ by 
	\begin{equation}\label{eq_defn_fermi}
		\psi_{y_0}^{k, l}(Z_H, Z_N) := \exp_{\exp_{y_0}^{X}(Z_H)}^{X \times X}(Z_N(Z_H)),
	\end{equation}
	where $Z_N(Z_H)$ is the parallel transport of $Z_N \in N_{k, l, y_0}$ along the path $(\exp_{y_0}^{X}(t Z_H), \exp_{y_0}^{X}(t Z_H))$, $t = [0, 1]$, with respect to the connection $\nabla^{N}_{k, l} := P^N_{k, l} \nabla^{T(X \times X)} P^N_{k, l}$ on $N_{k, l}$, and the exponential map on $X \times X$ is taken with respect to the Kähler form $\omega_{X \times X}$.
	We will implicitly fix an orthonormal basis in $T_{y_0} X$ and trivialize $N_{k, l, y_0}$ using it.
	In this way, we may view $\psi_{y_0}^{k, l}$ as a map defined on an open neighborhood of $0$ in $\real^n \times \real^n$.
	\par 
	Let us now fix an element $f_0 \in L_{y_0}$ of unit norm.
	We define the local orthonormal frame $\tilde{f}$ of $L$ around $y_0$ by the parallel transport of $f$ with respect to $\nabla^L$, done first along $\psi^{k, l}_{y_0}(t Z_H, 0)$, $t \in [0, 1]$, and then along $\psi^{k, l}_{y_0}(Z_H, tZ_N)$, $t \in [0, 1]$, $Z_H \in T_{y_0} X$, $Z_N \in N_{k, l, y_0}$, $|Z_H| < r_X$, $|Z_N| < r_{\perp}$.
	\par 
	Recall that the exponential coordinates $\phi^X$ and the frame $\tilde{f}'$ of $L$ in a neighborhood of $y_0 \in X$ were defined in (\ref{eq_phi_defn}).
	For $g \in \ccal^{\infty}(X \times X, L^k \boxtimes L^l)$, by an abuse of notation, we write $g(\phi_{y_0}^X(Z), \phi_{y_0}^X(Z')) \in L_{y_0}^{k+l}$, $Z, Z' \in Z_H \in T_{y_0} X$, $|Z|, |Z'| \leq R$, for coordinates of $g$ in the frame $\tilde{f}'$, identified with an element in $L_{y_0}^{k+l}$ using the frame $f$.
	Similarly, we denote by $g(\psi_{y_0}^{k, l}(Z_H, Z_N)) \in L_{y_0}^{k+l}$ the coordinates in the frame $\tilde{f}$, identified with an element from $L_{y_0}^{k+l}$.
	\par 
	Our first result gives a near-diagonal expansion of the optimal decomposition map.
	\begin{lem}\label{lem_dec_asmp}
		There are $\epsilon, c, C > 0$, $p_1 \in \nat^*$, such that for any $y_0 \in X$, $k, l \geq p_1$, $Z_0, Z_1, Z_2 \in T_{y_0}X$, $|Z_0|, |Z_1|, |Z_2| \leq \epsilon$, we have
		\begin{multline}\label{eq_dec_asmp}
			\bigg| 
					\frac{1}{(k + l)^n} {\rm{Dec}}_{k, l} \big((\phi_{y_0}^X(Z_1), \phi_{y_0}^X(Z_2)), \phi_{y_0}^X(Z_0) \big)
					-
					\mathscr{P}_n \big(\sqrt{k} Z_1, \sqrt{k} Z_0 \big)
					\cdot
					\mathscr{P}_n \big(\sqrt{l} Z_2, \sqrt{l} Z_0 \big)
			\bigg|
			\\
			\leq
			C \Big( 
			\frac{1}{k} + \frac{1}{l}
			\Big)
			\cdot
			\Big(1 + \sqrt{k}(|Z_1| + |Z_0|) +  \sqrt{l}(|Z_2| + |Z_0|) \Big)^{12n + 8}
			\cdot
			\\
			\cdot
			\exp\Big(- c \big( \sqrt{k} |Z_1 - Z_0| + \sqrt{l} |Z_2 - Z_0|\big) \Big),
		\end{multline}
		where $\mathscr{P}_n$ was defined in (\ref{eq_berg_k_expl}).
 	\end{lem}	
 	\begin{proof}
 		It follows directly from Theorems \ref{thm_bk_off_diag}, \ref{thm_berg_off_diag}, Lemma \ref{lem_mult_def_exp_dec_diag}, (\ref{eq_dec_form_inv_a}) and the usual composition rules for the operators with Taylor-type expansions and exponential decay of the Schwartz kernel, cf. \cite[Lemma 3.5]{FinOTAs}.
 	\end{proof}
 	We will now state the analogous result for the trivial decomposition map.
 	\begin{lem}\label{lem_dec_asmp2}
		There are $\epsilon, c, C > 0$, $p_1 \in \nat^*$, such that for any $y_0 \in X$, $k, l \geq p_1$, $Z_H, Z_0 \in T_{y_0}X$, $Z_N \in N_{k, l, y_0}$, $|Z_H|, |Z_N|, |Z_0| \leq \epsilon$, we have
		\begin{multline}\label{eq_ext_as_expk}
			\bigg| 
					\frac{1}{(k + l)^n} {\rm{Dec}}_{k, l}^0 \big(\psi_{y_0}^{k, l}(Z_H, Z_N), \phi_{y_0}^X(Z_0) \big)
					-
					\mathscr{P}_n \big(\sqrt{k + l} Z_H, \sqrt{k + l} Z_0 \big)
					\cdot
					\\
					\cdot
					\exp\Big( - \frac{\pi}{2} |Z_N|^2_{k, l} \Big) 
			\bigg|
			\leq
			C \Big( 
			\frac{1}{k} + \frac{1}{l}
			\Big)
			\cdot
			\Big(1 + \sqrt{k + l}(|Z_H| + |Z_0|)\Big)^{3n + 2}
			\cdot
			\\
			\cdot
			\exp\Big(- c \cdot \sqrt{k + l} \cdot |Z_H - Z_0| - c \cdot \frac{kl}{k + l} \cdot |Z_N|^2 \Big).
		\end{multline}
 	\end{lem}
 	\begin{proof}
 		The proof follows directly from Theorem \ref{thm_berg_off_diag} and (\ref{eq_ext0_op}).
 	\end{proof}
 	We, finally need the following general result for the estimation of operator's norms.
 	Let assume that a sequence of operators $R_{k, l} : L^2(X, L^{k + l}) \to L^2(X \times X, L^k \boxtimes L^l)$, $k, l \in \nat$, is such that its Schwartz kernel satisfies both the estimate (\ref{eq_ext_exp_dc}) and the estimate below.
 	There are $\epsilon, c, C, Q > 0$, $p_1 \in \nat^*$, such that for any $y_0 \in X$, $k, l \geq p_1$, $Z_0, Z_1, Z_2 \in T_{x_0}X$, $|Z_0|, |Z_1|, |Z_2| \leq \epsilon$, we have
	\begin{multline}\label{eq_dec_asmp}
		\bigg| 
				\frac{1}{(k + l)^n} R_{k, l} \big((\phi_{y_0}^X(Z_1), \phi_{y_0}^X(Z_2)), \phi_{y_0}^X(Z_0) \big)
		\bigg|
		\\
		\leq
		C \Big( 
		\frac{1}{k} + \frac{1}{l}
		\Big)
		\cdot
		\Big(1 + \sqrt{k}(|Z_1| + |Z_0|) +  \sqrt{l}(|Z_2| + |Z_0|) \Big)^{Q}
		\cdot
		\\
		\cdot
		\exp\Big(- c \big( \sqrt{k} |Z_1 - Z_0| + \sqrt{l} |Z_2 - Z_0|\big) \Big).
	\end{multline}
 	\begin{lem}\label{lem_norm_r_bnd_gen}
 		There is $p_1 \in \nat$, such that for $k, l \geq p_1$, the operator $R_{k, l}$ extends using Schwartz kernel as an operator $R_{k, l} : L^q(X, L^{k + l}) \to L^q(X \times X, L^k \boxtimes L^l)$ for $q = 1, \infty$, and we have
 		\begin{equation}\label{eq_rest_term_sch_bnd}
 		\begin{aligned}
 			&
 			\big\| R_{k, l} \big\|_{L^1_{k + l}(X, h^L) \to L^1_{k, l}(X \times X, h^L \boxtimes h^L)} \leq
 			C 
 			\Big(
 				\frac{1}{k} + \frac{1}{l}
 			\Big)
 			\cdot
 			\Big(
 				\frac{k + l}{k \cdot l}
 			\Big)^n,
 			\\
 			&
 			\big\| R_{k, l} \big\|_{L^{\infty}_{k + l}(X, h^L) \to L^{\infty}_{k, l}(X \times X, h^L \boxtimes h^L)} \leq
 			C \Big(
 				\frac{1}{k} + \frac{1}{l}
 			\Big)
 			.
 		\end{aligned}
 		\end{equation}
 	\end{lem}
 	\begin{proof}
 		First of all, a calculation of Gaussian integrals shows that there is $C > 0$, such that for any $k, l \in \nat^*$, $f \in L^1(\comp^n)$, we have 
 		\begin{multline}\label{eq_norm_r_bnd_gen1}
 			\int_{Z_2 \in \comp^n} \int_{Z_1 \in \comp^n} \int_{Z_0 \in \comp^n} 
 			\exp\Big(- c \big( \sqrt{k} |Z_1 - Z_0| + \sqrt{l} |Z_2 - Z_0|\big) \Big)
 			\cdot
 			|f(Z_0)|
 			dZ_0 dZ_1 dZ_2
 			\\
 			\leq
 			\frac{C}{(k \cdot l)^n}
 			\| f \|_{L^1}.
 		\end{multline}
 		Now, remark that the triangle inequality implies
 		\begin{equation}
 			\sqrt{k} |Z_1 - Z_0| + \sqrt{l} |Z_2 - Z_0| 
 			\geq 
 			\big( \sqrt{k} + \sqrt{l} \big) 
 			\cdot 
 			\Big|
 			\frac{\sqrt{k} Z_1 + \sqrt{l} Z_2}{\sqrt{k} + \sqrt{l}}
 			- 
 			Z_0
 			\Big|.
 		\end{equation}
 		This with a calculation of Gaussian integrals shows that there is $C > 0$, such that for any $k, l \in \nat^*$, $g \in L^{\infty}(\comp^n)$, we have 
 		\begin{multline}\label{eq_norm_r_bnd_gen2}
 			\sup_{Z_2 \in \comp^n} \sup_{Z_1 \in \comp^n} \int_{Z_0 \in \comp^n} 
 			\exp\Big(- c \big( \sqrt{k} |Z_1 - Z_0| + \sqrt{l} |Z_2 - Z_0|\big) \Big)
 			\cdot
 			|g(Z_0)|
 			dZ_0
 			\\
 			\leq
 			\frac{C}{(k + l)^n}
 			\cdot
 			\| g \|_{L^{\infty}}.
 		\end{multline}
 		Exponential decay assumption on the kernel of $R_{k, l}$ implies that only local contributions of the Schwartz kernel of $R_{k, l}$ matter in the study of the $L^1$ and $L^{\infty}$-operator norms of $R_{k, l} f$ for $f \in H^0(X, L^{k + l})$.
 		According to (\ref{eq_rest_term_sch_bnd}), estimates (\ref{eq_norm_r_bnd_gen1}) and (\ref{eq_norm_r_bnd_gen2}), done in local coordinates, show that $R_{k, l}$ extends to an operator $R_{k, l} : L^q(X, L^{k + l}) \to L^q(X \times X, L^k \boxtimes L^l)$ for $q = 1, \infty$, and the announced bounds on the resulting norms hold.
 	\end{proof}
 	\begin{proof}[Proof of Theorem \ref{thm_high_term_ext}]
 		First of all, let us denote by $R_{k, l} : L^2(X, L^{k + l}) \to L^2(X \times X, L^k \boxtimes L^l)$ the operator, given by the difference
 		\begin{equation}
 			R_{k, l} := {\rm{Dec}}_{k, l} - {\rm{Dec}}_{k, l}^0.
 		\end{equation}
 		Then Lemma \ref{lem_ext_exp_dc} implies that the Schwartz kernel of $R_{k, l}$ has exponential decay as in (\ref{eq_ext_exp_dc}).
 		\par 
 		Now, for a fixed $x_0 \in X$ and arbitrary $Z_1, Z_2 \in T_{x_0}X$, we decompose the vector $(Z_1, Z_2) \in T_{x_0 \times x_0}(X \times X)$ into tangential and normal part
 		\begin{equation}
 			(Z_1, Z_2)
 			=
 			Z_H + Z_N,
 		\end{equation}
 		where $Z_H = (Z_H^0, Z_H^0) \in T \Delta$ and $Z_N \in N_{k, l}$ for a given $k, l \in \nat^*$.
 		Clearly, we then have
 		\begin{equation}
 			Z_H^0 = \frac{k Z_1 + l Z_2}{k + l},
 			\qquad
 			Z_N = \Big( \frac{l ( Z_1 - Z_2)}{k + l}, \frac{k ( Z_2 - Z_1)}{k + l} \Big).
 		\end{equation}
 		An easy calculation, using the formula for $\mathscr{P}_n$ and the above formulas for $Z_H^0, Z_N$ shows that
 		\begin{multline}\label{eq_pn_hor_vert_dec}
 			\mathscr{P}_n \big(\sqrt{k} Z_1, \sqrt{k} Z_0 \big)
			\cdot
			\mathscr{P}_n \big(\sqrt{l} Z_2, \sqrt{l} Z_0 \big)
			\\
			=
			\mathscr{P}_n \big(\sqrt{k + l} Z_H^0, \sqrt{k + l} Z_0 \big)
			\cdot
			\exp\Big( - \frac{\pi}{2} |Z_N|^2_{k, l} \Big).
 		\end{multline}
 		Also, up to the second order, the following identity holds 
 		\begin{equation}
 			\psi_{y_0}^{k, l}(Z_H^0, Z_N)
 			=
 			(\phi_{y_0}^X(Z_1), \phi_{y_0}^X(Z_2)).
 		\end{equation}
 		Hence, by (\ref{eq_pn_hor_vert_dec}) and Lemmas \ref{lem_ext_exp_dc}, \ref{lem_dec_asmp}, \ref{lem_dec_asmp2}, we see that all the assumptions of Lemma \ref{lem_norm_r_bnd_gen} are satisfied for the operator $R_{k, l}$.
 		An application of Lemma \ref{lem_norm_r_bnd_gen}, hence, finishes the proof.
 	\end{proof}

\section{The set of multiplicatively generated norms}\label{sect_mult_gen_mn}
	The main goal of this section is to study asymptotically the set of multiplicatively generated norms on the section ring.
	In particular, we establish Theorems \ref{thm_mult_gen}, \ref{thm_hilb_leaf0}, \ref{thm_hilb_leaf} and \ref{thm_pj_st_ref}.
	The main tool in this study is the Fubini-Study operator from Section \ref{sect_fs_psh}.
	\par The section is organized as follows.
	In Section \ref{sect_mult_gen_psh}, we establish that plurisubharmonic metrics give rise to multiplicatively generated norms. 
 	More precisely, we show Theorem \ref{thm_mult_gen}.
 	\par 
 	In Section \ref{sect_hilb}, we study the metric properties of the maps ${\rm{Ban}}^1$, ${\rm{Hilb}}$, ${\rm{Ban}}^{\infty}$.
	More precisely, we prove Theorem \ref{thm_hilb_leaf}.
	We will also explain how this result falls into a more general picture of approximations of $\mathcal{H}^L$ by the space of norms on $H^0(X, L^k)$, as $k \to \infty$.
	This section is essentially independent from the rest of the text, except for Section \ref{sect_ph_st}.
	\par 
	In Section \ref{sect_conv}, we study the convergence properties of Fubini-Study operator for multiplicatively generated norms and establish the following result.
	\begin{thm}\label{thm_conv_mg}
		For any graded multiplicatively generated norm $N:= \sum_{k = 1}^{\infty} N_k$ on $R(X, L)$, the sequence of metrics $FS(N_k)^{\frac{1}{k}}$ converges uniformly, as $k \to \infty$, to a (continuous psh) metric on $L$, which we denote by $FS(N) \in \mathcal{H}^L$.
	\end{thm} 
	Theorem \ref{thm_conv_mg} motivates the following natural question: does the convergence of the Fubini-Study operator of a graded norm of the section ring determines the equivalence class of the norm?
	The answer is negative, see Section \ref{sect_ex_mg_norm} for an explicit example.
	\par 
	Nevertheless, the following result from Section \ref{sect_fin_ress_pf} shows that the answer becomes positive if we restrict our attention to a special class of graded norms.
	\begin{thm}\label{thm_tame_conv}
		Assume that for a graded norm $N:= \sum_{k = 1}^{\infty} N_k$ on $R(X, L)$, the sequence of metrics $FS(N_k)^{\frac{1}{k}}$ converges uniformly, as $k \to \infty$, to a (continuous psh) metric $FS(N)$ on $L$, and there is $p_0 \in \nat$ and a function $f : \nat_{\geq p_0} \to \real$, verifying $f(k) = o(k)$, as $k \to \infty$, such that for any $r \in \nat^*$, $k; k_1, \ldots, k_r \geq p_0$, $k_1 + \cdots + k_r = k$, under the map (\ref{eq_mult_map}), we have the following bound
		\begin{equation}\label{eq_half_mult_gen}
			N_k
			\leq
			\big[ 
			 N_{k_1} \otimes \cdots \otimes N_{k_r}
			 \big]
			 \cdot
			 \exp \Big( f(k_1) + \cdots + f(k_r) + f(k) \Big).
		\end{equation}
		Then the following equivalence of graded norms on $R(X, L)$ holds $N \sim {\rm{Hilb}}(FS(N))$.
 	\end{thm}
 	Clearly, Theorems \ref{thm_conv_mg} and \ref{thm_tame_conv} imply Theorem \ref{thm_hilb_leaf0}.
 	We then show in Section \ref{sect_ph_st}, that from Theorem \ref{thm_tame_conv} and some elementary facts from the theory of interpolation of Hilbert spaces, one can establish Theorem \ref{thm_pj_st_ref}.

\subsection{Multiplicatively generated norms from plurisubharmonic metrics}\label{sect_mult_gen_psh}
	The main goal of this section is to establish that plurisubharmonic metrics give rise to multiplicatively generated norms through the Hilbert map.
	In other words, we prove Theorem \ref{thm_mult_gen}.
	\par 
	The proof of this result is decomposed into two parts: establishing the lower bound and the upper bound.
	For the upper bound, the following simple observation is crucial.
	\begin{lem}\label{lem_ban_inf_upp_bnd}
		For any $h^L \in \mathcal{H}^L$, $r \in \nat^*$, $k; k_1, \ldots, k_r \geq p_0$, $k_1 + \cdots + k_r = k$, under the map (\ref{eq_mult_map}), the following inequality holds
		\begin{equation}
			{\rm{Ban}}_k^{\infty}(h^L) \leq \Big[
				{\rm{Ban}}_{k_1}^{\infty}(h^L)
				\otimes_{\epsilon}
				\cdots
				\otimes_{\epsilon}
				{\rm{Ban}}_{k_r}^{\infty}(h^L)
			\Big].
		\end{equation}
	\end{lem}
	\begin{proof}
		From Theorem \ref{thm_compar_tens_prod}, we see that under the Künneth isomorphism
		\begin{equation}
			H^0(X \times \cdots \times X, L^{k_1} \boxtimes \cdots \boxtimes L^{k_r})
			\to
			H^0(X, L^{k_1})
			\otimes
			\cdots
			\otimes 
			H^0(X, L^{k_r}),
		\end{equation}
		the norm ${\rm{Ban}}_{k_1}^{\infty}(h^L)
				\otimes_{\epsilon}
				\cdots
				\otimes_{\epsilon}
				{\rm{Ban}}_{k_r}^{\infty}(h^L)$
		gets pulled-back to the $L^{\infty}$-norm on $H^0(X \times \cdots \times X, L^{k_1} \boxtimes \cdots \boxtimes L^{k_r})$.
		The result now follows directly from this, the fact that multiplication operator can be seen as the the restriction operator, cf. (\ref{eq_comm_diag}), and the trivial fact that the $L^{\infty}$-norm of the restriction of a given section is always no bigger than the $L^{\infty}$-norm of the section.
	\end{proof}
	For the lower bound, we will need the following slightly more complicated result.
	\begin{lem}\label{lem_hilb_low_bnd}
		Assume that the measure $\mu$ on $X$ can be majorized from below and above by some strictly positive volume forms on $X$. 
		Then there are $C > 0$, $p_0 \in \nat$, such that for any $h^L \in \mathcal{H}^L$, $r \in \nat^*$, $k; k_1, \ldots, k_r \geq p_0$, $k_1 + \cdots + k_r = k$, under the map (\ref{eq_mult_map}), the following inequality holds
		\begin{equation}
			\Big[
				{\rm{Hilb}}_{k_1}(h^L, \mu)
				\otimes
				\cdots
				\otimes
				{\rm{Hilb}}_{k_r}(h^L, \mu)
			\Big] \leq C^r \cdot {\rm{Hilb}}_k(h^L, \mu).
		\end{equation}
	\end{lem}
	\begin{proof}
		Clearly, by Corollary \ref{cor_quot_ass}, it is enough to establish this result for $r = 2$, as for other $r$ it would follow by induction.
		Once the definition of the quotient norm is unraveled, it becomes clear that it suffices to prove that for any $f \in H^0(X, L^{k_1 + k_2})$, there is a tensor $h \in H^0(X, L^{k_1}) \otimes H^0(X, L^{k_2})$, verifying ${\rm{Mult}}_{k_1, k_2}(h) = f$, such that
		\begin{equation}
			\| h \|_{L^2_{k_1}(X, h^L) \otimes L^2_{k_2}(X, h^L)} \leq C \| f \|_{L^2_{k_1 + k_2}(X, h^L)}.
		\end{equation}
		This, however, follows directly from the interpretation of the multiplication map of the section ring as an instance of the restriction map, see (\ref{eq_comm_diag}), and a statement similar to Theorem \ref{thm_ot_asymp}, where instead of $L^k$, we consider $L^{k_1} \boxtimes L^{k_2}$.
		The proof of the latter statement is done analogously to the proof of Theorem \ref{thm_ot_asymp}, and is left to the interested reader.
	\end{proof}
	\begin{proof}[Proof of Theorem \ref{thm_mult_gen}.]
		From Corollary \ref{cor_ident_maps}, Lemma \ref{lem_ban_inf_upp_bnd}, (\ref{eq_est_derham}) and the fact that the injective tensor norm is majorized from above by the Hermitian tensor norm, cf. Lemma \ref{lem_tens_reas}, we conclude that there is a function $f : \nat \to \real$, $f(k) = o(k)$, such that in the notations of Lemma \ref{lem_hilb_low_bnd}, we have
		\begin{equation}\label{eq_mult_gen1}
			{\rm{Hilb}}_k(h^L) \leq \Big[
				{\rm{Hilb}}_{k_1}(h^L)
				\otimes
				\cdots
				\otimes
				{\rm{Hilb}}_{k_r}(h^L)
			\Big]
			\cdot
			\exp \Big( f(k_1) + \cdots + f(k_r) \Big).
		\end{equation}
		Similarly, from Lemma \ref{lem_hilb_low_bnd} and the second part of Corollary \ref{cor_ident_maps}, we conclude that there is a function $g : \nat \to \real$, $g(k) = o(k)$, such that in the notations of Lemma \ref{lem_hilb_low_bnd} we have
		\begin{equation}\label{eq_mult_gen2}
			{\rm{Hilb}}_k(h^L) \geq \Big[
				{\rm{Hilb}}_{k_1}(h^L)
				\otimes
				\cdots
				\otimes
				{\rm{Hilb}}_{k_r}(h^L)
			\Big]
			\cdot
			\exp \Big( g(k_1) + \cdots + g(k_r) + g(k) \Big).
		\end{equation}
		We deduce Theorem \ref{thm_mult_gen} from (\ref{eq_mult_gen1}) and (\ref{eq_mult_gen2}).
	\end{proof}

\subsection{Metric properties of the Hilbert map and Mabuchi supremum geometry}\label{sect_hilb}
	The main goal of this section is to study the asymptotic metric properties of the map ${\rm{Hilb}}$.
	More precisely, we compare the norms ${\rm{Hilb}}_k(h^L_0), {\rm{Hilb}}_k(h^L_1)$ for $h^L_0, h^L_1 \in \mathcal{H}^L$,  as $k \to \infty$.
	\begin{proof}[Proof of Theorem \ref{thm_hilb_leaf}.]
		By Corollary \ref{cor_ident_maps}, we can study the distance between the images of the map ${\rm{Ban}}^{\infty}$ instead of ${\rm{Hilb}}^{\infty}$.
		Trivially, for any metrics $h^L_0$, $h^L_1$ on $L$, we obtain that 
		\begin{equation}\label{eq_isemd_aux_0}
			d_{\infty}({\rm{Ban}}^{\infty}(h^L_0), {\rm{Ban}}^{\infty}(h^L_1))
			\leq
			d_{\infty}( h^L_0, h^L_1).
		\end{equation}
		\par 
		Let us now establish the opposite bound.
		Our idea is to construct, using the Ohsawa-Takegoshi extension theorem, a holomorphic section of $L^k$, with the mass concentrated in a neighborhood of a fixed point.
		More precisely, by Theorem \ref{thm_bbw}, Proposition \ref{prop_bm} and Lemmas \ref{lem_bm_rest}, \ref{lem_ot_loc}, we conclude that for any $\epsilon > 0$, there is $p_1 \in \nat$, such that for any $k \geq p_1$, $x \in X$, there is $f \in H^0(X, L^k)$, verifying the following inequality
		\begin{equation}\label{eq_isemd_aux_1}
			|f(x)|_1 \geq \exp(- \epsilon k) \| f \|_{L^{\infty}_k(X, h^L_1)},
		\end{equation}
		where $| \cdot |_i$, $i = 1, 2$, are the pointwise norms induced by $h^L_i$.
		Remark, however, that 
		\begin{equation}\label{eq_isemd_aux_2}
			\| f \|_{L^{\infty}_k(X, h^L_0)}
			\geq
			|f(x)|_0,
		\end{equation}
		and we clearly have 
		\begin{equation}\label{eq_isemd_aux_3}
			|f(x)|_0
			=
			\Big(\frac{h^L_0}{h^L_1} \Big)^{\frac{k}{2}} (x)
			\cdot
			|f(x)|_1.
		\end{equation}	 
		From (\ref{eq_isemd_aux_1}), (\ref{eq_isemd_aux_2}) and (\ref{eq_isemd_aux_3}), we conclude that for any $\epsilon > 0$, there is $p_1 \in \nat$, such that for any $k \geq p_1$, $x \in X$, there is $f \in H^0(X, L^k)$, such that we have 
		\begin{equation}\label{eq_isemd_aux_33}
			\| f \|_{L^{\infty}_k(X, h^L_0)} \geq \exp(- \epsilon k) \cdot \Big(\frac{h^L_0}{h^L_1} \Big)^{\frac{k}{2}} (x) \cdot \| f \|_{L^{\infty}_k(X, h^L_1)}.
		\end{equation}		
		Moreover, one can interchange the roles of $h^L_0, h^L_1$ in (\ref{eq_isemd_aux_33}).		
		In particular, we obtain that
		\begin{equation}\label{eq_isemd_aux_4}
			d_{\infty}({\rm{Ban}}^{\infty}(h^L_0), {\rm{Ban}}^{\infty}(h^L_1))
			\geq
			d_{\infty}( h^L_0, h^L_1).
		\end{equation}
		From (\ref{eq_isemd_aux_0}) and (\ref{eq_isemd_aux_4}), we deduce Theorem \ref{thm_hilb_leaf}.
	\end{proof}
	\par 
	Theorem \ref{thm_hilb_leaf} goes in line with the general philosophy that the geometry of $\mathcal{H}^L$ can be approximated by the geometry of the space of norms on $H^0(X, L^k)$, as $k \to \infty$, see Donaldson \cite{DonaldSymSp}.
	To explain this precisely, we recall the basics of the geometry on the space of Kähler potentials.
	To make our presentation compatible with the notations introduced previously, we will only work with Kähler potentials in the integral class $c_1(L)$, but everything we say below can be trivially extended to the case of a general (non-integral) Kähler class.
	\par 
	First of all, remark that the tangent space $T \mathcal{H}^L$ of continuous psh metrics $\mathcal{H}^L$ on $L$ can be seen as a subbundle of the trivial vector bundle 
	\begin{equation}\label{eq_tan_kah_pot}
		T \mathcal{H}^L \subset \mathcal{H}^L \times \ccal^0(X),
	\end{equation}
	as follows.
	Let $h_t^L$, $t \in [0, 1]$, be a $\mathcal{C}^1$-path in $\mathcal{H}^L$. 
	Then $\dot h_t^L :=  (h_t^L)^{-1} \frac{d}{dt} h_t^L \in \ccal^0(X)$ gives the needed tangent vector in the identification.
	\par 
	For $p \in [1, + \infty]$, let us define the $L^p$-norm $\| \cdot \|_p$ for $u \in T \mathcal{H}^L$ as follows
	\begin{equation}
		\| u \|_p
		=
		\Big(
		\frac{1}{\int c_1(L)^n}
		\int_X | v |^p(x) \cdot c_1(L, h^L)^n \Big)^{\frac{1}{p}},
	\end{equation}
	where $u = (h^L, v)$ in the identification (\ref{eq_tan_kah_pot}) and $c_1(L, h^L)^n$ is the measure constructed from the procedure of Bedford-Taylor, see (\ref{eq_chern_class_defn}).
	\par 
	In particular, in the most relevant case for this article, $p = + \infty$, we have
	\begin{equation}
		\| u \|_{\infty}
		=
		{\rm{ess \, sup}} |v|,
	\end{equation}
	where the essential supremum over $X$ is taken with respect to the measure $c_1(L, h^L)^n$.
	\par 
	For $p = 2$, this (Hermitian) norm in the context of Kähler potentials was defined by Mabuchi. 
	For other $p$, it was studied by Darvas in \cite{DarvasFinEnerg}.
	\par 
	We define the $L^{\infty}$-length ${\rm{len}}(l)$, of a piece-wise $\ccal^1$-curve $l(t) := h^L_t \in \mathcal{H}^L$, $t \in [0, 1]$, as 
	\begin{equation}
		{\rm{len}}(l) := \frac{1}{2} \int_0^1 \big\| \dot h_t^L \big\|_{\infty} dt.
	\end{equation}
	The normalizing factor $\frac{1}{2}$, usually omitted in this setting, is here to make things compatible with our previous normalizations, adapted to the study of norms and not metrics.
	We define the supremum ($L^{\infty}$)-pseudodistance $\dist(h^L_0, h^L_1)$ between $h^L_0, h^L_1 \in \mathcal{H}^L$ as follows
	\begin{equation}\label{eq_dist_len_min}
		\dist(h^L_0, h^L_1)
		=
		\inf {\rm{len}}(l),
	\end{equation}
	where the infimum is taken among all possible piece-wise $\ccal^{1}$-paths $l$ from $h^L_0$ to $h^L_1$.
	\par 
	The fact that $\dist$ is actually a distance (i.e. separates distinct points) is a nontrivial result in this infinite dimensional setting.
	Since $L^{\infty}$-distance introduced above obviously dominates $L^p$-distances for any $p \in [1, \infty[$, for smooth endpoints, the fact that $\dist$ is actually a distance follows from the work of Chen \cite{ChenGeodMab}, who proved that the $L^2$-pseudodistance is actually a distance. For general endpoints it follows from Darvas \cite{DarvasFinEnerg}, who refined this statement for $p = 1$.
	\par 
	Our interest in this metric comes from the following result.
	\begin{thm}\label{thm_dist_equiv}
		The distance $\dist$ from (\ref{eq_dist_len_min}), and the $d_{\infty}$ from (\ref{dist_smooth}), coincide.
	\end{thm}
	\begin{rem}
		If we combine Theorems \ref{thm_hilb_leaf} and \ref{thm_dist_equiv}, we see that the limit of the Goldman-Iwahori distances on the space of norms on $H^0(X, L^k)$ coincides with the $L^{\infty}$-Mabuchi distance on $\mathcal{H}^L$.
		It complements the results of Chen-Sun \cite{ChenSunQuant}, who proved a similar result relating the distance induced by the Hilbert-Shmidt norm on the space of norms on $H^0(X, L^k)$ and the Mabuchi distance on $\mathcal{H}^L$, Berndtsson \cite{BerndtProb} who generalized the last statement for the $L^p$-distances with smooth positive exteremities, $1 \leq p < + \infty$, and Darvas-Lu-Rubinstein \cite{DarvLuRub}, who generalized the statement of Berndtsson for paths with non-smooth extremities.
	\end{rem}
	The proof of this result is decomposed into a number of statements, which essentially follow from the works of Berndtsson, Darvas and Lempert.
	To explain it in details, we need to recall the notion of weak geodesics.
	Let $h^L_0, h^L_1$ be two smooth positive metrics on $L$.
	A geodesic equation (with respect to $L^2$-metric) for a smooth curve $h^L_t$ connecting $h^L_0$ and $h^L_1$ was found by Mabuchi \cite{Mabuchi}.
	Later, Semmes \cite{Semmes} and Donaldson \cite{DonaldSymSp} independently discovered that this equation can be understood as a complex Monge-Ampère equation.
	\par 
	More precisely, they proved that if $h^L_t$, $t \in [0, 1]$, is a geodesic segment in the sense of Mabuchi, then for a Hermitian metric $\tilde{h}^L$ on $L$, over $X \times S_{1, 2}$, $S_{1, 2} := \{ s \in \comp : 1 \leq \Re s \leq 2 \}$, defined as $\tilde{h}^L|_{X \times u} := h^L_{\Re u - 1}$, $u \in S_{1, 2}$, the following identity over $X \times S_{1, 2}$ holds
	\begin{equation}\label{eq_geod_mab}
		c_1(L, \tilde{h}^L)^{n + 1} = 0.
	\end{equation}
	Analogously to the smooth setting, a segment $h^L_t \in \mathcal{H}^L$, $t \in [0, 1]$, is called \textit{weak geodesic} if the complexification of it, $\tilde{h}^L$, constructed as above, satisfies (\ref{eq_geod_mab}) in the Bedford-Taylor sense.
	\par 
	The existence and uniqueness of weak geodesics $h^L_t \in \mathcal{H}^L$, $t \in [0, 1]$, with smooth positive extremities $h^L_0$, $h^L_1$, was established by Chen \cite{ChenGeodMab}.
	Chen, moreover, proved that in the above notations, $c_1(L, \tilde{h}^L)$ is a bounded $(1, 1)$-form.
	By standard regularity results, it implies that the path $h^L_t$ viewed as a metric over $X \times [0, 1]$ is $\ccal^{1, \alpha}$, for any $0 \leq \alpha < 1$.
	From the general theory of Dirichlet problem for complex Monge-Ampère equations of Bedford-Taylor \cite[Theorem D]{BedTayDirProb}, weak geodesics exist more generally for $h^L_0, h^L_1 \in \mathcal{H}^L$.
	Those geodesics are moreover unique by the “minimum principle", see \cite[Theorem A]{BedTayDirProb}.
	The following result explains their relevance for us.
	\begin{lem}\label{lem_lemp}
		Let $h^L_t \in \mathcal{H}^L$, $t \in [0, 1]$, be a weak geodesic.
		Then $\dist(h^L_0, h^L_1) = \frac{1}{2} \| \dot h_0^L \|_{\infty}$.
	\end{lem}
	\begin{proof}
		It is a direct consequence of Lempert \cite[Theorem 8.1]{LempLeastAction}.
		In fact, in the notations of Lempert, the Lagrangian $\| \cdot \|_{\infty}$ is convex and invariant under strict rearrangements, see \cite[Definition 2.1]{LempLeastAction} for necessary definitions.
		Hence, \cite[Theorem 8.1]{LempLeastAction} applies.
	\end{proof}
	The next result we need relates the right-hand side from Lemma \ref{lem_lemp} with weak geodesics.
	\begin{lem}\label{lem_weak_geod_dist}
		Let $h^L_t \in \mathcal{H}^L$, $t \in [0, 1]$, be a weak geodesic with smooth positive extremities $h^L_0$, $h^L_1$.
		Then the essential supremums of the velocities of $h^L_t$ are constant in $t \in [0, 1]$, and, despite the fact that $h^L_t$ are not necessarily strictly positive, they coincide with the supremums.
		More precisely, for any $t_0, t_1 \in [0, 1]$, we have
		\begin{equation}
			\| \dot h_{t_0}^L \|_{\infty}
			=
			\sup |\dot h_{t_0}^L|
			=
			\sup |\dot h_{t_1}^L|
			=
			\| \dot h_{t_1}^L \|_{\infty}
			=
			2 \cdot d_{\infty}(h^L_0, h^L_1).
		\end{equation}
	\end{lem}
	\begin{proof}
		First of all, by an estimate of Darvas \cite[Theorem 3.4]{DarvWeakGeod}, for any $t \in [0, 1]$, we have
		\begin{equation}\label{eq_dar_00}
			2 \cdot d_{\infty}(h^L_0, h^L_1) \geq \sup |\dot h_t^L|.
		\end{equation}
		\par 
		Now, since the path $h^L_t$ is $\ccal^{1}$, for any $x \in X$, we have 
		\begin{equation}
			\log \frac{h^L_1(x)}{h^L_0(x)}
			=
			\int_0^1 \dot h_t^L(x).
		\end{equation}
		In particular, we obtain that
		\begin{equation}\label{eq_dar_0}
			\Big| \log \frac{h^L_1(x)}{h^L_0(x)} \Big|
			\leq
			\sup_{t \in [0, 1]} |\dot h_t^L(x)|.
		\end{equation}
		By combining (\ref{eq_dar_00}) with (\ref{eq_dar_0}) and using the fact that $\sup_{x \in X} |\dot h_t^L(x)|$ is continuous for $t \in [0, 1]$, we obtain that for any $t \in [0, 1]$, we have
		\begin{equation}\label{eq_dar_1}
			2 \cdot d_{\infty}(h^L_0, h^L_1) = \sup_{x \in X} |\dot h_t^L(x)|.
		\end{equation}
		\par 
		On the other hand, since $h^L_0$ is smooth and positive, $c_1(L, h^L_0)^n$ is a smooth (strictly) positive volume form.
		Hence the essential supremum of $|\dot h_0^L|$ coincides with the supremum.
		In other words, we have the following identity
		\begin{equation}\label{eq_dar_2}
			\sup_{x \in X} |\dot h_0^L(x)|
			=
			\| \dot h_0^L \|_{\infty}.
		\end{equation}
		Recall, however, that by a result of Berndtsson \cite[Proposition 2.2]{BerndtProb}, for any function $f \in \ccal^0(\real)$, and any $t_1, t_2 \in [0, 1]$, the following conservation law holds
		\begin{equation}
			\int_X f(\dot h_{t_1}^L) c_1(L, h_{t_1}^L)^n
			=
			\int_X f(\dot h_{t_2}^L) c_1(L, h_{t_2}^L)^n.
		\end{equation}
		In particular, by choosing a sequence of functions $f$ appropriately, we obtain that 
		\begin{equation}\label{eq_dar_3}
			\| \dot h_{t_1}^L \|_{\infty}
			=
			\| \dot h_{t_2}^L \|_{\infty}.
		\end{equation}
		A combination of (\ref{eq_dar_1}), (\ref{eq_dar_2}) and (\ref{eq_dar_3}) finishes the proof.
	\end{proof}
	\begin{proof}[Proof of Theorem \ref{thm_dist_equiv}]
		A combination of Lemmas \ref{lem_lemp} and \ref{lem_weak_geod_dist} gives the result for smooth positive extremities.
		Remark, however, that any $h^L \in \mathcal{H}^L$ can be approximated uniformly by smooth positive metrics $h^L_k$, $k \in \nat$.
		To see this, consider the approximation $h^L_k := FS({\rm{Hilb}}_k(h^L))$, and use Theorem \ref{thm_quant_hilb_conv}.
		The result now follows by the fact that Theorem \ref{thm_dist_equiv} holds for $h^L_k$, by approximation and triangle inequality.
	\end{proof}

\subsection{Fubini-Study operator on graded multiplicatively generated norms}\label{sect_conv}
	The main goal of this section is to study the convergence of the Fubini-Study operator associated to graded pieces of multiplicatively generated norms, i.e. to prove Theorem \ref{thm_conv_mg}.
	\par 
	The proof of this result decomposes into a series of small statements.
	Our first lemma, which is a generalization of Theorem \ref{thm_as_isom}, was the starting point for the formulation of Definition \ref{defn_mult_gen}.
	\begin{lem}\label{lem_mult_gen_sm}
		For any smooth positive metric $h^L$ on $L$, there are $p_0 \in \nat$, $C > 0$, such that for any $r \in \nat^*$, $k; k_1, \ldots, k_r \geq p_0$, $k_1 + \cdots + k_r = k$, under the map (\ref{eq_mult_map}), we have
		\begin{multline}\label{eq_mult_gen222}
			\prod_{i = 1}^{r} \Big(1 - \frac{C}{k_i}\Big)
			\leq 
			\frac{[{\rm{Hilb}}_{k_1}(h^L) \otimes \cdots \otimes {\rm{Hilb}}_{k_r}(h^L)]}{{\rm{Hilb}}_{k}(h^L)} 
			\cdot
			\Big( \frac{k_1 \cdots k_r}{k} \Big)^{\frac{n}{2}}  
			\leq 
			\prod_{i = 1}^{r} \Big(1 + \frac{C}{k_i}\Big).
		\end{multline}
	\end{lem}
	\begin{rem}\label{rem_mult_gen_sm}
		In particular, the graded norms ${\rm{Hilb}}(h^L)$ are multiplicatively generated for smooth positive $h^L$, and we can take  $f(k) := f'(k) = \frac{n}{2} \log(k) + 2 + \log \dim H^0(X, L^k)$ in Definition \ref{defn_mult_gen}.
	\end{rem}
	\begin{proof}
		It follows from Theorem \ref{thm_as_isom} and Corollary \ref{cor_quot_ass} by induction.
	\end{proof} 
	\par 
	In the next lemma we study how the Fubini-Study maps associated for different parameters $k$ relate to each other.
	\begin{lem}\label{lem_segre}
		Assume that $p_0 \in \nat$ is such that $L^{p_0}$ is very ample. Then for any Hermitian graded norm $N = \sum_{k = 1}^{\infty} N_k$ on $R(X, L)$ and for any $r \in \nat^*$, $k; k_1, \ldots, k_r \geq p_0$, $k_1 + \cdots + k_r = k$, the following identity is satisfied
		\begin{equation}
			{\rm{FS}}(N_{k_1}) \cdots {\rm{FS}}(N_{k_r}) = {\rm{FS}}([N_{k_1} \otimes \cdots \otimes N_{k_r}]).
		\end{equation}
	\end{lem}
	\begin{proof}
		For simplicity, let us establish the result for $N = 2$, as the general case is an easy adaptation of this particular one.
		To simplify the notation, let us fix two sufficiently big numbers $k, l \in \nat$.
		Recall that the Segre embedding 
		\begin{equation}\label{eq_seg_emb11}
			{\rm{Seg}}_{k, l} : \mathbb{P}(H^0(X, L^k)^*) \times \mathbb{P}(H^0(X, L^l)^*) 
			\to
			\mathbb{P}(H^0(X, L^k)^* \otimes H^0(X, L^l)^*),
		\end{equation}
		is defined as follows $[x] \times [y] \mapsto [x \otimes y]$.
		Under (\ref{eq_seg_emb11}), the tautological line bundle over $\mathbb{P}(H^0(X, L^k)^* \otimes H^0(X, L^l)^*)$ gets pulled back to the $\boxtimes$-product of the tautological line bundles on $\mathbb{P}(H^0(X, L^k)^*)$ and $\mathbb{P}(H^0(X, L^l)^*)$.
		By Lemma \ref{lem_tens_reas}, this map is an isometry once the tautological line bundle over $\mathbb{P}(H^0(X, L^k)^* \otimes H^0(X, L^l)^*)$ is endowed with the metric induced by the Fubini-Study metric associated to $N_k^* \otimes N_l^*$ and the tautological line bundles on $\mathbb{P}(H^0(X, L^k)^*)$ and $\mathbb{P}(H^0(X, L^l)^*)$ are induced by the Fubini-Study metrics associated to $N_k^*$ and $N_l^*$ respectively.
		\par 
		An easy verification shows that the following diagram is commutative
		\begin{equation}\label{eq_comm_diag2}
			\begin{CD}
				X  @> {\rm{Kod}}_k \times {\rm{Kod}}_l  >> \mathbb{P}(H^0(X, L^k)^*) \times \mathbb{P}(H^0(X, L^l)^*)
				\\
				@VV {{\rm{Kod}}_{k + l} } V @VV {{\rm{Seg}}_{k, l}} V
				\\
				\mathbb{P}(H^0(X, L^{k + l})^*)  @> ({\rm{Mult}}_{k, l})^* >> \mathbb{P}(H^0(X, L^k)^* \otimes H^0(X, L^l)^*).
			\end{CD}
		\end{equation}
		Moreover, the map $({\rm{Mult}}_{k, l})^*$ induces the isomorphism of the corresponding tautological line bundles. 
		This isomorphism becomes an isometry if the tautological line bundle over $\mathbb{P}(H^0(X, L^{k + l})^*)$ is induced by the Fubini-Study metric associated to the metric on $H^0(X, L^{k + l})^*$ obtained from $N_k^* \otimes N_l^*$  and the inclusion $({\rm{Mult}}_{k, l})^*$, and the tautological line bundle over $\mathbb{P}(H^0(X, L^k)^* \otimes H^0(X, L^l)^*)$ is induced by the Fubini-Study metric associated to $N_k^* \otimes N_l^*$.
		Remark, however, that by Lemmas \ref{lem_dual_incl_eq} and \ref{lem_dual_proj_inj}, the dual of the above norm on $H^0(X, L^{k + l})^*$ coincides with $[N_k \otimes N_l]$.
		This now finishes the proof by the commutativity of square (\ref{eq_comm_diag2}) and the above isometries.
	\end{proof}	
	\par 
	Now, to finally establish Theorem \ref{thm_conv_mg}, we need the following simple lemma.
	\begin{lem}\label{lem_bnd_FS}
		Assume that $k \in \nat^*$ is such that $L^k$ is very ample and the norms $N_k$, $N'_k$ over $H^0(X, L^k)$ satisfy the inequality $N_k \leq C^k \cdot N'_k$ for a given constant $C > 0$.
		Then we have
		\begin{equation}
			{\rm{FS}}(N_k) \leq C^k \cdot {\rm{FS}}(N'_k).
		\end{equation}
		If the graded norms $N = \sum_{k = 1}^{\infty} N_k$, $N' = \sum_{k = 1}^{\infty} N'_k$ on $R(X, L)$ are moreover equivalent, then, as $k \to \infty$, the following convergence holds uniformly
		\begin{equation}
			\Big( \frac{FS(N_k)}{FS(N'_k)} \Big)^{\frac{1}{k}} \to 1.
		\end{equation}
	\end{lem}
	\begin{proof}
		It follows trivially from the definition of the Fubini-Study operator from (\ref{eq_fs_defn}).
	\end{proof}
	\begin{proof}[Proof of Theorem \ref{thm_conv_mg}]
		Let us denote by $f$ the function from Definition \ref{defn_mult_gen} associated to the norm $N$.
		From Lemmas \ref{lem_segre} and \ref{lem_bnd_FS}, we conclude that there is $p_0 \in \nat$, such that for any $r \in \nat^*$, $k; k_1, \ldots, k_r \geq p_0$, $k_1 + \cdots + k_r = k$, we have
		\begin{equation}\label{eq_conv_1}
			{\rm{FS}}(N_{k_1}) \cdots {\rm{FS}}(N_{k_r}) \leq {\rm{FS}}(N_k) \cdot \exp \Big(f(k_1) + \cdots + f(k_r) + f(k) \Big).
		\end{equation}
		\par 
		Now, let $h^L$ be an arbitrary smooth positive Hermitian metric on $L$.
		According to Lemmas \ref{lem_mult_gen_sm} and \ref{lem_bnd_FS}, for the function $f'$ as in Remark \ref{rem_mult_gen_sm}, we see that there is $p_1 \geq p_0$, such that for any $r \in \nat^*$, $k; k_1, \ldots, k_r \geq p_1$, $k_1 + \cdots + k_r = k$, we have
		\begin{multline}\label{eq_conv_2}
			{\rm{FS}}({\rm{Hilb}}_{k_1}(h^L)) \cdots {\rm{FS}}({\rm{Hilb}}_{k_r}(h^L)) \\
			\geq {\rm{FS}}({\rm{Hilb}}_k(h^L)) \cdot \exp \Big(f'(k_1) + \cdots + f'(k_r) + f'(k) \Big).
		\end{multline}
		From (\ref{eq_conv_1}) and (\ref{eq_conv_2}), we conclude that the sequence of functions $g_k : X \to ]-\infty, +\infty[$, $k \in \nat$, defined as follows
		\begin{equation}\label{eq_gk_defn}
			g_k := \log \Big( \frac{{\rm{FS}}({\rm{Hilb}}_k(h^L))}{{\rm{FS}}(N_k)} \Big),
		\end{equation}
		is uniformly almost sub-additive in the sense of Section \ref{sect_unif_alm_sub}.
		By symmetry, the sequence $g_k$ is also uniformly almost super-additive.
		By a version of Fekete's subadditive lemma and Dini's lemma, see Lemma \ref{lem_sub_super_unif}, we conclude that there is a continuous function $g : X \to ]- \infty, + \infty[$, such that $\frac{g_k}{k}$ converges to $g$ uniformly.
		By this and Tian's theorem, cf. Remark \ref{rem_tian}, we conclude from (\ref{eq_gk_defn}) that, as $k \to \infty$, the following uniform convergence of metrics on $L$ holds
		\begin{equation}
			{\rm{FS}}(N_k)^{\frac{1}{k}} \to e^{-g} \cdot h^L.
		\end{equation}		 
		Since ${\rm{FS}}(N_k)$ are psh, we conclude that $e^{-g} \cdot h^L$ is psh as well.
	\end{proof}
	\par 
	Finally, motivated by the methods of this section, let us establish that graded multiplicatively generated norms decart at most exponentially one from another.
	\begin{lem}\label{lem_mg_equiv}
		Let $N:= \sum_{k = 1}^{\infty} N_k$, $N' := \sum_{k = 1}^{\infty} N'_k$ be two graded multiplicatively generated norms on $R(X, L)$.
		Then there is $C > 1$, such that for any $k \in \nat$, we have
		\begin{equation}\label{eq_mg_eqv}
			C^{-k} \cdot N'_k  \leq N_k \leq C^k \cdot N'_k.
		\end{equation}
	\end{lem}
	\begin{rem}\label{rem_fin_dist}
		In particular, we have $d_{\infty}(N, N') < + \infty$.
	\end{rem}
	\begin{proof}
		Assume for simplicity that we can take $p_0 = 1$ in Definition \ref{defn_mult_gen}. 
		Then we see that there is a function $f : \nat \to \real$, verifying $f(k) = o(k)$, as $k \to \infty$, and such that 
		\begin{equation}
			N_k \leq N_1 \otimes \cdots \otimes N_1 \cdot \exp(f(n) + n f(1)).
		\end{equation}
		Similarly, for $N'$ we can write the opposite inequality
		\begin{equation}
			N'_1 \otimes \cdots \otimes N'_1 \leq N'_k \cdot \exp(f'(n) + n f'(1)),
		\end{equation}
		where the function $f' : \nat \to \real$ verifies $f'(k) = o(k)$, as $k \to \infty$.
		In particular, by the assumption on the growth of $f$, and the fact that all norms on $H^0(X, L)$ are equivalent, we see that there is a constant $C > 0$, such that for any $k \in \nat$, the left side of the inequality (\ref{eq_mg_eqv}) holds.
		By symmetry, we establish also the inequality from the right side.
	\end{proof}

\subsection{An example of a graded Hermitian norm on the section ring}\label{sect_ex_mg_norm}
	The main goal of this section is emphasize that the equivalence class of a graded Hermitian metric cannot be determined from the convergence of the  Fubini-Study operator.
	\begin{prop}\label{prop_expl_exmpl}
		There is a graded Hermitian norm $N = \sum_{k = 1}^{\infty} N_k$ on $R(\mathbb{P}^1, \mathscr{O}(1))$, for which $FS(N_k)^{\frac{1}{k}}$ converge uniformly, as $k \to \infty$, to a fixed smooth positive metric on $\mathscr{O}(1)$, but $N$ is not equivalent to ${\rm{Hilb}}(h^L)$ for any continuous psh metric $h^L$ on $\mathscr{O}(1)$.
	\end{prop}
	\begin{proof}
		Let us identify $\mathbb{P}^1$ to $\mathbb{P}(V^*)$, where $V$ is a vector space generated by two elements: $x$ and $y$.
		Let us consider a Hermitian metric $H$ on $V$, which makes $x$ and $y$ an orthonormal basis, and denote by $h^{FS}$ the induced Fubini-Study metric on $\mathscr{O}(1)$.
		Recall that for any $k \in \nat$, we have a classical isomorphism 
		\begin{equation}\label{eq_sym_isom}
			{\rm{Sym}}^k(V) \to H^0(\mathbb{P}(V^*), \mathscr{O}(k)).
		\end{equation}
		Under this isomorphism, for any $a, b \in \nat$, $a + b = k$, an easy calculation shows that we have
		\begin{equation}\label{eq_l2_norm_calc}
			\big\| x^a \cdot y^b \big\|_{{\rm{Hilb}}_k(h^{FS})}^2
			=
			\frac{a! b!}{(k + 1)!}.
		\end{equation}
		We draw the attention of the reader that the above norm doesn't coincide with the norm $\| \cdot  \|_{{\rm{Sym}}^k(H)}$ on ${\rm{Sym}}^k(V)$ induced by $H$. In fact, an easy calculation gives
		\begin{equation}
			\big\| x^a \cdot y^b \big\|_{{\rm{Sym}}^k(H)}^2
			=
			\frac{a! b!}{k!}.
		\end{equation}
		\par
		Let us consider a Hermitian norm $H_k$ on $H^0(\mathbb{P}(V^*), \mathscr{O}(k))$, such that the basis $x^a \cdot y^b$ is orthogonal and in the above notations, we have
		\begin{equation}
			\big\| x^a \cdot y^b \big\|_{H_k}
			=
			\begin{cases}
				1, &\text{if } |a - b| \leq 1 \\
				\big\| x^a \cdot y^b \big\|_{{\rm{Hilb}}_k(h^{FS})} ,  &\text{otherwise}.
			\end{cases}
		\end{equation}
		We will now verify that $H_k$ satisfies the assumptions of the proposition.
		For simplicity, we only work with metrics $H_k$ with $k$ even.
		By Stirling's formula and (\ref{eq_l2_norm_calc}), we have
		\begin{equation}
			\frac{\big\| x^k \cdot y^k \big\|^2_{H_{2k}}}{\big\| x^k \cdot y^k \big\|^2_{{\rm{Hilb}}_{2k}(h^{FS})}}
			\geq 2^k.
		\end{equation}
		Hence, the graded norm $\sum_{k = 1}^{\infty} H_k$ is not equivalent to ${\rm{Hilb}}(h^{FS})$.
		We will now show that the Fubini-Study metric of $H_k$ converges uniformly, as $k \to \infty$, to $h^{FS}$.
		This would clearly imply by Theorem \ref{thm_quant_hilb_conv} that $\sum_{k = 1}^{\infty} H_k$ satisfies the assumptions of the proposition.
		\par 
		For this, from Lemma \ref{lem_fs_inf_d}, we see that for any $a, b \in \comp$, not simultaneously equal to zero, the following identity between the Fubini-Study metrics of $H_k$ and ${\rm{Hilb}}_{k}(h^{FS})$ holds
	 	\begin{equation}\label{eq_expl_exmpl11}
	 		1 - \frac{FS({\rm{Hilb}}_{2k}(h^{FS}))}{FS(H_{2k})} \Big([a x^* + b y^*] \Big)
	 		=
	 		\Big( \frac{(2k)!}{k! k!} - 1 \Big)
	 		\cdot
	 		\frac{|a b|^k }{(|a| + |b|)^{2k}}.
	 	\end{equation}
	 	However, by mean inequality and Stirling's formula, we obtain that there is $C > 0$, such that for any $k \in \nat^*$, we have
	 	\begin{equation}\label{eq_expl_exmpl22}
	 		\Big| 
	 			\Big( \frac{(2k)!}{k! k!} - 1 \Big)
	 			\cdot
	 			\frac{|a b|^k }{(|a| + |b|)^{2k}}
	 		\Big|
	 		\leq
	 		\frac{1}{\sqrt{k}}.
	 	\end{equation}
	 	From (\ref{eq_expl_exmpl11}), (\ref{eq_expl_exmpl22}) and Theorem \ref{thm_quant_hilb_conv}, $FS(H_k)^{\frac{1}{k}}$ converges uniformly, as $k \to \infty$, to $h^{FS}$.
	\end{proof}
	As the following example of László Lempert shows, a phenomenon similar to Proposition \ref{prop_expl_exmpl} holds in the non-asymptotic regime as well.
	 \begin{prop}[László Lempert]\label{prop_ll_ex}
	 	There are sequences of Hermitian metrics $H_{\delta}$, $H'_{\delta}$, $\delta \in ]0, 1]$ over $H^0(\mathbb{P}^1, \mathscr{O}(2))$, such that, as $\delta \to \infty$, $\frac{FS(H_{\delta})}{FS(H'_{\delta})} \to 1$, but there is $c > 0$ such that $d_{\infty}(H_{\delta}, H'_{\delta}) > c$ for any $\delta \in ]0, 1]$. 
	 \end{prop}
	 \begin{rem}
	 	See also Lempert \cite{LempFsInj} for a related result about injectivity of the Fubini-Study map.
	 \end{rem}
	 \begin{proof}
	 	We use similar notations as in the proof of Proposition \ref{prop_expl_exmpl}.
	 	Let $H_{\delta}$ (resp. $H'_{\delta}$) be so that the basis $(s_{1, \delta}, s_{2, \delta}, s_{3, \delta}) := (xy, 2xy+2 \delta y^2, x^2+y^2)$  (resp. $(2s_{1, \delta}, \frac{1}{2} s_{2, \delta}, s_{3, \delta})$) is orthonormal with respect to $H_{\delta}$ (resp. $H'_{\delta}$). We now verify that such choice satisfies the assumptions of the proposition.
	 	The second condition is obvious since $H_{\delta}(s_{1, \delta}, s_{1, \delta}) = 1$, but $H'_{\delta}(s_{1, \delta}, s_{1, \delta}) = \frac{1}{4}$.
	 	Now, from Lemma \ref{lem_fs_inf_d}, we have
	 	\begin{equation}\label{eq_ll_ex111}
	 		\frac{FS(H_{\delta})}{FS(H'_{\delta})}
	 		=
	 		\frac{ 4 |s_{1, \delta}|^2 + \frac{1}{4} |s_{2, \delta}|^2 + |s_{3, \delta}|^2}{|s_{1, \delta}|^2 + |s_{2, \delta}|^2 + |s_{3, \delta}|^2},
	 	\end{equation}
	 	where $| \cdot |$ is any norm on $\mathscr{O}(2)$.
	 	From (\ref{eq_ll_ex111}), $\frac{FS(H_{\delta})}{FS(H'_{\delta})}$ tends to $1$ uniformly, as $\delta \to 0$.
	\end{proof}

	\subsection{Inductive construction of the Hilbert map, a proof of Theorem \ref{thm_tame_conv}}	\label{sect_fin_ress_pf}
	The main goal of this section is to establish Theorem \ref{thm_tame_conv}.
 	The proof is done using the following result providing the inductive construction of the Hilbert map.
 	\par 
 	Assume for simplicity that $L$ is very ample and the multiplication map ${\rm{Mult}}_{1, 1}$ is surjective (hence, all multiplication maps are surjective as well).
 	Let $H_1$ be any Hermitian norm on $H^0(X, L)$.
 	Then using the map (\ref{eq_mult_map}) by the surjectivity of the multiplication maps, one can endow $H^0(X, L^k)$ with the Hermitian norm $H_k = [H_1 \otimes \cdots \otimes H_1]$, where the tensor product is repeated $k$ times.
 	We introduce the graded Hermitian norm $\mathcal{H} := \sum_{k = 1}^{\infty} H_k$ on $R(X, L)$.
 	\begin{thm}\label{thm_induc}
 		The graded Hermitian norms $\mathcal{H}$ and ${\rm{Hilb}}(FS(H_1))$ on $R(X, L)$ are equivalent.
 	\end{thm} 
 	\begin{rem}
 		a) The Fubini-Study operators for both of those norms coincide by Theorem \ref{thm_quant_hilb_conv} and Lemma \ref{lem_segre}.
 		It doesn't, however, imply the equivalence of the graded norms by Section \ref{sect_ex_mg_norm}.
 		\par 
 		b) 
 		It is tempting to think that Theorem \ref{thm_induc} follows from Lemma \ref{lem_mult_gen_sm}.
 		This is, however, not the case.
 		Not only in Lemma \ref{lem_mult_gen_sm}, we require $k_1, \cdots, k_r$ to be big enough, but also not every Hermitian norm on $H^0(X, L)$ lies in the image of the Hilbert map, see Sun \cite{JingSunImage} for obstructions.
 		\par 
 		c) From Corollary \ref{cor_quot_ass}, $\mathcal{H}$ is multiplicatively generated.
 		Hence, Theorem \ref{thm_induc} is a special case of Theorem \ref{thm_tame_conv} by Lemma \ref{lem_segre}.
 		Our proof of Theorem \ref{thm_tame_conv}, however, passes by Theorem \ref{thm_induc}.
 	\end{rem}
 	\par 
 	Before proceeding with the proof of Theorem \ref{thm_induc}, let us show how it implies Theorem \ref{thm_tame_conv}. 
	\begin{proof}[Proof of Theorem \ref{thm_tame_conv} assuming Theorem \ref{thm_induc}]
		Let us fix $\epsilon > 0$.
		By the uniform convergence of the sequence of metrics $FS(N_k)^{\frac{1}{k}}$, $k \in \nat^*$, there is $p_1 \in \nat$, such that for any $k \geq p_1$, we have
		\begin{equation}\label{eq_pf_fina_thm_2}
			d_{\infty} \big( FS(N_k)^{\frac{1}{k}}, FS(N) \big) < \frac{\epsilon}{5}.
		\end{equation}		
		From the growth condition on $f$, we know that there is $p_2 \in \nat$, such that for any $k \geq p_2$, we have
		\begin{equation}\label{eq_pf_fina_thm_3}
			| f(k) | \leq \frac{\epsilon}{5} k.
		\end{equation}
		\par 
		Now, let us fix $p_0 = \max \{p_1, p_2 \}$, $k \geq p_0$. We consider the Hermitian metric $[N_k \otimes \cdots \otimes N_k]$ induced by the multiplication map (\ref{eq_mult_map}) on $H^0(X, L^{k r})$.
		The tensor product is repeated here $r$ times.
		According to Theorem \ref{thm_induc}, there is $p_3 \geq p_0$, such that for any $r \geq p_3$, we have
		\begin{equation}\label{eq_pf_fina_thm_4}
			d_{\infty} \Big( [N_k \otimes \cdots \otimes N_k], {\rm{Hilb}}_{rk}(FS(N_k)^{\frac{1}{k}}) \Big)
			<
			\frac{\epsilon}{5} r.
		\end{equation}
		From Corollary \ref{cor_ident_maps}, (\ref{eq_isemd_aux_0}), (\ref{eq_pf_fina_thm_2}) and (\ref{eq_pf_fina_thm_4}), there is $p_4 \geq p_3$, such that for any $r \geq p_4$, we have
		\begin{equation}\label{eq_pf_fina_thm_5}
			d_{\infty} \Big( [N_k \otimes \cdots \otimes N_k], {\rm{Hilb}}_{rk}(FS(N)) \Big)
			\leq
			\frac{2 \epsilon}{5} r k.
		\end{equation}
		However, from (\ref{eq_half_mult_gen}), we conclude that 
		\begin{equation}\label{eq_pf_fina_thm_6}
			N_{r k} \leq [N_k \otimes \cdots \otimes N_k] \cdot \exp\Big(\frac{\epsilon}{5} r k\Big).
		\end{equation}
		In total, from (\ref{eq_pf_fina_thm_5}) and (\ref{eq_pf_fina_thm_6}), we deduce that 
		\begin{equation}\label{eq_pf_fina_thm_7}
			N_{r k} \leq {\rm{Hilb}}_{rk}(FS(N))  \cdot \exp\Big(\frac{3 \epsilon}{5} r k \Big).
		\end{equation}
		Now, since $k \in \nat$ is fixed, and the spaces $H^0(X, L^p)$, $p = k, \ldots, 2k - 1$, are finitely dimensional, the norms $N_p$ and ${\rm{Hilb}}_{p}(FS(N))$ are comparable up to a uniform constant. 
		From this and Theorem \ref{thm_mult_gen}, there is $p_5 \geq p_4$, such that for any $0 \leq l \leq k - 1$, $r \geq p_5$, we have
		\begin{equation}\label{eq_pf_fina_thm_8}
			d_{\infty} \Big( [{\rm{Hilb}}_{rk}(FS(N)) \otimes N_{k + l}], {\rm{Hilb}}_{(r + 1)k + l}(FS(N)) \Big)
			\leq
			\frac{\epsilon}{5} ((r+1)k + l).
		\end{equation}		 
		On another hand, from Definition \ref{defn_mult_gen} and (\ref{eq_pf_fina_thm_3}), we conclude that 
		\begin{equation}\label{eq_pf_fina_thm_9}
			N_{(r + 1)k + l} \leq [N_{r k} \otimes N_{k + l}] \cdot \exp \Big( \frac{\epsilon}{5} ((r+1)k + l) \Big).
		\end{equation}
		Now, from (\ref{eq_pf_fina_thm_7}), (\ref{eq_pf_fina_thm_8}) and (\ref{eq_pf_fina_thm_9}), we deduce that for any $k \geq 2 p_5 p_0$, we have
		\begin{equation}\label{eq_fin_1_22}
			N_k \leq {\rm{Hilb}}_k(FS(N)) \cdot \exp ( \epsilon k ).
		\end{equation}
		\par 
		Directly from Lemma \ref{lem_fs_inf_d}, we see that for any $k \in \nat^*$, we have
		\begin{equation}\label{eq_norm_with_fs_compar}
			N_k \geq {\rm{Ban}}^{\infty}_{k}(FS(N_k)^{\frac{1}{k}}).
		\end{equation}	
		From Corollary \ref{cor_ident_maps}, (\ref{eq_pf_fina_thm_2}) and (\ref{eq_norm_with_fs_compar}), there is $p_6 \in \nat^*$, such that for any $k \geq p_6$, we have
		\begin{equation}\label{eq_fin_1_33}
			N_k \geq {\rm{Hilb}}_k(FS(N))  \cdot  \exp ( - \epsilon k ).
		\end{equation}
		As $\epsilon > 0$ was arbitrary, the statement follows from (\ref{eq_fin_1_22}) and (\ref{eq_fin_1_33}).
	\end{proof}
	
	\begin{proof}[Proof of Theorem \ref{thm_induc}]
		Our proof is essentially based on two ingredients.
		First, we need to verify Theorem \ref{thm_induc} for the projective space by an explicit calculation.
		Second, we apply the semiclassical version of Ohsawa-Takegoshi extension theorem, Theorem \ref{thm_semicalss_ot}, for the Kodaira embedding to extend  Theorem \ref{thm_induc} from projective spaces to general manifolds.
		\par 
		Let us first do the calculations for the projective space $X := \mathbb{P}(V^*)$, where $V$ is a fixed vector space of dimension $n$.
		In this case, we take $L := \mathscr{O}(1)$.
		Then through the identification of $H^0(X, L)$ with $V$, we may view the Hermitian metric $H_1$ on $H^0(X, L)$ as a Hermitian metric on $V$.
		We induce the Fubini-Study metric $FS(H_1)$ on $L$ from this identification.
		\par 
		A direct calculation shows that the Hermitian norm $H_k := [H_1 \otimes \cdots \otimes H_1]$ corresponds to the norm ${\rm{Sym}}^k(H_1)$ induced on ${\rm{Sym}}^k(V)$ by $H_1$ under the identification (\ref{eq_sym_isom}).
		Now, as it was already exemplified during the proof of Proposition \ref{prop_expl_exmpl}, see (\ref{eq_l2_norm_calc}), this norm is not isometric to ${\rm{Hilb}}_{k}(FS(H_1))$, but the following identity holds
		\begin{equation}\label{eq_induc_1}
			{\rm{Hilb}}_{k}(FS(H_1)) = \sqrt{\frac{k!}{(k + n - 1)!}} \cdot {\rm{Sym}}^k(H_1).
		\end{equation}
		Note that the factor in (\ref{eq_induc_1}) grows polynomially in $k$ (in particular, it is subexponential). 
		Hence, Theorem \ref{thm_induc} holds in this particular case.
		\par 
		Now, to establish Theorem \ref{thm_induc} in its full generality, let us consider the Kodaira embedding ${\rm{Kod}}_1$ from (\ref{eq_kod}).
		We denote by $\res_{{\rm{Kod}}, X}$ the associated restriction operator
		\begin{equation}
			\res_{{\rm{Kod}}, k} : H^0(\mathbb{P}(H^0(X, L)^*), \mathscr{O}(k)) \to H^0(X, L^k).
		\end{equation}
		The crux of the matter is to realize that the multiplication operator ${\rm{Mult}}_{1, \cdots, 1}$ from (\ref{eq_mult_map}) factorizes under the identification (\ref{eq_sym_isom}) of $H^0(\mathbb{P}(H^0(X, L)^*), \mathscr{O}(k))$ with ${\rm{Sym}}^k(H^0(X, L))$ through the symmetrization, ${\rm{Sym}}$, and the restriction operator as follows
		\begin{equation}\label{eq_kod_map_comm_d}
		\begin{tikzcd}
			H^0(X, L)^{\otimes k} \arrow[swap, rrdd, "{\rm{Mult}}_{1, \cdots, 1}"] \arrow[r, "{\rm{Sym}}"] & {\rm{Sym}}^k(H^0(X, L)) \arrow[rd, equal] &  \\
			&  & H^0(\mathbb{P}(H^0(X, L)^*), \mathscr{O}(k)) \arrow[d, "\res_{{\rm{Kod}}, k}"]  \\
 	 		&  & H^0(X, L^k).
    	\end{tikzcd}
		\end{equation}
		\par 
		Now, the proof goes as follows.
		From the first part of the proof of Theorem \ref{thm_induc}, we obtain that the quotient metric on ${\rm{Sym}}^k(H^0(X, L))$ associated to $H \otimes \cdots \otimes H$ is equivalent under the isomorphism (\ref{eq_sym_isom}) to the metric ${\rm{Hilb}}_{k}^{\mathbb{P}(H^0(X, L)^*)}(FS(H_1))$.
		Now, by the semiclassical version of Ohsawa-Takegoshi extension theorem, see Theorem \ref{thm_semicalss_ot} or Corollary \ref{cor_comp_res_1}, the quotient of the metric ${\rm{Hilb}}_{k}^{\mathbb{P}(H^0(X, L)^*)}(FS(H_1))$ under the map $\res_{{\rm{Kod}}, k}$ is equivalent to ${\rm{Hilb}}_{k}^{X}(FS(H_1))$. 
		But from the commutative diagram (\ref{eq_kod_map_comm_d}), the quotient metric $[H \otimes \cdots \otimes H]$ on $H^0(X, L^k)$ can be obtained by two subsequent quotients: one through the map ${\rm{Sym}}$, another through the map $\res_{{\rm{Kod}}, k}$, and the above reasoning shows that the subsequent quotient is equivalent to ${\rm{Hilb}}_{k}^{X}(FS(H_1))$.
	\end{proof}
	
	\subsection{Mabuchi geodesics and geodesics on the space of norms}\label{sect_ph_st}
	The main goal of this section is to prove a relation between Mabuchi geodesics and geodesics in the space of norms, announced in Theorem \ref{thm_pj_st_ref}, from which we borrow the notation.
	Let us first recall the following result.
	\begin{thm}\label{thm_ph_st_fs}
		For any $t \in [0, 1]$, $FS(H_{k, t})^{\frac{1}{k}}$ converges uniformly to $h^L_t$, as $k \to \infty$.
		Moreover, this convergence is uniform in $t$.
	\end{thm}
	\begin{proof}
		For smooth extremities $h^L_0, h^L_1$, this result was established by Phong-Sturm \cite[Theorem 1]{PhongSturm} for a slightly weaker than uniform convergence and by Berndtsson \cite[Theorem 1.2]{BerndtProb} for the uniform convergence.
		We argue that the proof of Berndtsson works more generally for $h^L_0, h^L_1 \in \mathcal{H}^L$.
		In fact, the regularity is used only in the first part of the proof of \cite[Theorem 4.2]{BerndtProb}, refining \cite[Theorem 1.2]{BerndtProb}, where author relies on the diagonal asymptotic expansion of the Bergman kernel to estimate from above the logarithm of the Bergman kernel of $H_{k, t}$.
		This argument can be replaced by Theorem \ref{thm_quant_hilb_conv} at the cost of replacing the $O(\frac{\log k}{k})$-estimate by an $o(1)$-estimate, which does not affect the final conclusion. 
	\end{proof}
	\par 
	Now, to establish Theorem \ref{thm_pj_st_ref}, we will need several basic fact from linear algebra.
	Let $U$ and $V$ be two finitely dimensional vector spaces and let $f : U \to V$ be any map between them.
	Recall that we say that $f$ is \textit{contracting} with respect to the norms $\| \cdot \|_U$ on $U$ and $\| \cdot \|_V$ on $V$ (or with respect to the pair $\| \cdot \|_U$, $\| \cdot \|_V$) if for any $x \in U$, we have $\| \pi(x) \|_V \leq \| x \|_U$, or, in other words $\|\pi \| \leq 1$, where $\| \cdot \|$ is the operator norm.
	Let $H^U_0$, $H^U_1$ (resp. $H^V_0$, $H^V_1$) be two Hermitian metrics on $U$ (resp. $V$).
	Let us denote by $H^U_t$, $H^V_t$, $t \in [0, 1]$, the geodesics between $H^U_0$, $H^U_1$ and $H^V_0$, $H^V_1$ in the space of Hermitian norms on $U$ and $V$ respectively.
	\begin{lem}\label{lem_interp}
		Assume that the map $\pi$ is contracting with respect to both pairs $H^U_0, H^V_0$ and $H^U_1, H^V_1$.
		Then it is contracting with respect to the pair $H^U_t, H^V_t$ for any $t \in [0, 1]$.
	\end{lem}
	\begin{rem}
		By choosing a simultaneous diagonalization of $H^U_0$, $H^U_1$, we may identify $(V, H^V_0)$ (resp. $(V, H^V_1)$) to the space $L^2(\{1, 2, \cdots, \dim V \}, \mu_V)$ (resp. $L^2(\{1, 2, \cdots, \dim V \}, \omega_V \mu_V)$), where $\mu_V = \delta_1 + \cdots + \delta_{\dim V}$ is the counting measure and $\omega_V : \{1, 2, \cdots, \dim V \} \to ]0, + \infty[$ is a certain weight function.
		Similarly, we may identify $(U, H^U_0)$ (resp. $(U, H^U_1)$) to the space $L^2(\{1, 2, \cdots, \dim U \}, \mu_U)$ (resp. $L^2(\{1, 2, \cdots, \dim U \}, \omega_U \mu_U)$), where similarly $\mu_U = \delta_1 + \cdots + \delta_{\dim U}$, $\omega_U : \{1, 2, \cdots, \dim V \} \to ]0, + \infty[$.
		The result now follows from the interpolation theorem of Stein-Weiss, see \cite[Theorem 5.4.1]{InterpSp}.
	\end{rem}
	\begin{cor}\label{cor_interp}
		Let us fix a surjective map $\pi : V \to Q$ between vector spaces $V, Q$. Let $N^V_0$, $N^V_1$ be two Hermitian norms on $V$ and $N^Q_0$, $N^Q_1$ be the induced quotient Hermitian norms on $Q$.
		For $t \in [0, 1]$, we denote by $N^V_t$ the geodesics between $N^V_0$ and $N^V_1$, and by $N^U_t$ the geodesics between $N^U_0$ and $N^U_1$.
		Then for any $t \in [0, 1]$, we have
		\begin{equation}
			[N^V_t]
			\geq
			N^U_t.
		\end{equation}
		Let us now denote by $N^{V, 0}_0$, $N^{V, 0}_1$ two other Hermitian norms on $V$, verifying $N^V_0 \geq N^{V, 0}_0$ and $N^V_1 \geq N^{V, 0}_1$.
		Similarly, we denote by $N^{V, 0}_t$, $t \in [0, 1]$ the geodesics between $N^{V, 0}_0$ and $N^{V, 0}_1$.
		Then for any $t \in [0, 1]$, we have
		\begin{equation}
			N^V_t \geq N^{V, 0}_t.
		\end{equation}
	\end{cor}
	\begin{proof}
		It follows directly from Lemma \ref{lem_interp}.
	\end{proof}
	\begin{proof}[Proof of Theorem \ref{thm_pj_st_ref}.]
		Let us denote by $f_0$ and $f_1$ the functions from Definition \ref{defn_mult_gen} associated to ${\rm{Hilb}}(h^L_0)$, ${\rm{Hilb}}(h^L_1)$ respectively by Theorem \ref{thm_mult_gen}. 
		We will now establish that for any $t \in [0, 1]$, the inequality (\ref{eq_half_mult_gen}) holds for $N_k := H_{k, t}$ and $f := f_t := (1 - t) f_0 + t f_1$.
		Remark first that  $H_{k_1, t} \otimes \cdots \otimes H_{k_r, t}$, $t \in [0, 1]$, is a geodesic between ${\rm{Hilb}}_{k_1}(h^L_0) \otimes \cdots \otimes {\rm{Hilb}}_{k_r}(h^L_0)$ and ${\rm{Hilb}}_{k_1}(h^L_1) \otimes \cdots \otimes {\rm{Hilb}}_{k_r}(h^L_1)$ in the space of Hermitian norms on $H^0(X, L^{k_1}) \otimes \cdots \otimes H^0(X, L^{k_r})$.
		\par 
		From this and the first part of Corollary \ref{cor_interp}, we deduce that if we denote by $H_{k_1, \cdots, k_r; t}$, $t \in [0, 1]$, the geodesic between $[{\rm{Hilb}}_{k_1}(h^L_0) \otimes \cdots \otimes {\rm{Hilb}}_{k_r}(h^L_0)]$ and $[{\rm{Hilb}}_{k_1}(h^L_1) \otimes \cdots \otimes {\rm{Hilb}}_{k_r}(h^L_1)]$ in the space of Hermitian norms on $H^0(X, L^k)$, then for any $t \in [0, 1]$, we have
		\begin{equation}\label{eq_fs_st_ph_1}
			H_{k_1, \cdots, k_r; t} \leq [H_{k_1, t} \otimes \cdots \otimes H_{k_r, t}].
		\end{equation}
		Now, from the definitions of $f_0$, $f_1$ and the second part of Corollary \ref{cor_interp}, we deduce that
		\begin{equation}\label{eq_fs_st_ph_2}
			H_{k, t} \leq H_{k_1, \cdots, k_r; t} \cdot \exp(f_t(k_1) + \cdots + f_t(k_r) + f_t(k)).
		\end{equation}
		From (\ref{eq_fs_st_ph_1}) and (\ref{eq_fs_st_ph_2}), we deduce that the inequality (\ref{eq_half_mult_gen}) holds for $N_k := H_{k, t}$ and $f := f_t := (1 - t) f_0 + t f_1$.
		From this, Theorem \ref{thm_tame_conv} and Theorem \ref{thm_ph_st_fs}, we conclude that for any $t \in [0, 1]$, the graded norm $H_t$ is equivalent to ${\rm{Hilb}}(h^L_t)$.
	\end{proof}
	\begin{rem}
		A quick inspection of the proof of Theorem \ref{thm_tame_conv} shows that the equivalence of $H_t$ to ${\rm{Hilb}}(h^L_t)$ is uniform in $t \in [0, 1]$ (i.e. the convergence $\frac{1}{k} d_{\infty}(H_{k, t}, {\rm{Hilb}}_k(h^L_t))  \to 0$ from (\ref{eq_equiv_rel_defn}) is uniform in $t$).
		This is due to the fact that the function $f_t$ above is uniformly bounded by a function $f_{\max} := \max \{ f_0, f_1 \}$, satisfying the growth assumption $f_{\max}(k) = o(k)$, and the convergence from Theorem \ref{thm_ph_st_fs} is uniform.
		Hence, Theorem \ref{thm_pj_st_ref} refines Theorem \ref{thm_ph_st_fs} by Theorem \ref{thm_quant_hilb_conv}, Lemma \ref{lem_bnd_FS} and Proposition \ref{prop_expl_exmpl}.
	\end{rem}

 \appendix 

	\section{On almost sub-additive sequences of functions}\label{sect_unif_conv}
	The main goal of this section is to study almost sub-additive sequences.
	More precisely, in Section \ref{sect_fek_1}, we state the analogue of Fekete's lemma.
	In Section \ref{sect_unif_alm_sub}, we study the uniform convergence of sub-additive sequences of functions and state the analogue of Dini's theorem.

	\subsection{A version of Fekete's lemma for almost sub-additive sequences}\label{sect_fek_1}
	In this section we discuss a generalization of Fekete's lemma, cf. \cite{Fekete}, studying sub-additive sequences of real numbers. 
	We will show that the conclusion of the lemma is still valid if the sequence is sub-additive up to a small error term. 
	\par 
	Recall that Bruijn-Erdős in \cite[Theorem 23]{BruijErd} considered sequences for which the sub-additivity condition $a_{m + n} \leq a_m + a_n$ is replaced by a weaker one 
	\begin{equation}\label{eq_br_erd}
		a_{m + n} \leq a_m + a_n + g_{m + n}.
	\end{equation}
	As they show, a conclusion of Fekete's lemma only holds under the assumption
	\begin{equation}
		\sum \frac{g_n}{n^2} < + \infty.
	\end{equation}
	As this condition doesn't appear quite naturally in geometric setting, such as Bernstein-Markov property, we find it unsuitable for the encapsulation of the features of the $L^2$-norm as we do it in Definition \ref{defn_mult_gen}.
	We, thus, consider an alternative condition, which demands more from the sequence $a_n$ and less from the error term.
	\begin{defn}\label{defn_alm_sub}
		We say that a sequence of numbers $a_n \in \real$, $n \in \nat$, is \textit{almost sub-additive} if for any $r \geq 2$, $k; k_1, \ldots, k_r \in \nat$, $k_1, \ldots, k_r \geq p_0$, $k_1 + \cdots + k_r = k$, the following inequality holds
		\begin{equation}\label{eq_alm_sub}
			a_k \leq a_{k_1} + \cdots + a_{k_r} + f_{k_1} + \cdots + f_{k_r} + f_{k},
		\end{equation}
		where $f_k$, verifies $f_k = o(k)$, as $k \to \infty$.
		We say that $a_k$ is \textit{almost super-additive} if the above inequality holds with a reverse sign.
		In the latter case we denote the sequence $f_k$ by $g_k$, $k \in \nat$.
	\end{defn}
	\begin{lem}[{Chen \cite[Proposition 3.1]{ChenHNolyg}}]\label{lemma_chen}
		Assume that a sequence $a_n$, $n \in \nat$, is almost sub-additive and non-negative.
		Then the sequence $\frac{a_n}{n}$ has a finite limit, as $n \to \infty$.
	\end{lem}
	
	\begin{cor}\label{cor_subadd_super_add}
		Assume that a sequence $a_n$, $n \in \nat$, is both almost sub-additive and almost super-additive.
		Then the sequence $\frac{a_n}{n}$ has a finite limit, as $n \to \infty$.
	\end{cor}
	\begin{proof}
		Since $a_n$ is super-additive, we have
		\begin{equation}\label{eq_bnd_below_an}
			a_n \geq n c_n,
		\end{equation}
		where $c_n$ is defined as $c_n = a_1 + \frac{g_n}{n}$ for a sequence $g: \nat \to \real$ from Definition \ref{defn_alm_sub}.
		Clearly, by our assumption on $g_n$, the sequence $c_n$ is bounded from below by $c \in \real$.
		We can now apply Lemma \ref{lemma_chen} for a sequence $a_n - n c$ to establish the result.
	\end{proof}
	
	\subsection{Uniform convergence of almost sub-additive sequences of functions}\label{sect_unif_alm_sub}
	
	The main goal of this section is to study the uniform convergence of almost sub-additive sequences of functions.
	The following definition will be particularly useful later on.
	\begin{defn}\label{defn_unif_almsubadd}
		A sequence of functions $a_k : Y \to \real$, $k \in \nat^*$, defined over a topological space $Y$ will be called \textit{uniformly almost sub-additive} if for any $x \in Y$, the sequence $a_k(x)$ is almost sub-additive, and in the definition of almost sub-additivity the sequences $f_n := f_n(x)$ can be chosen in such way that the convergence $\frac{f_k}{k} \to 0$ holds locally uniformly on $Y$, as $k \to \infty$. Similarly, we define the notion of \textit{uniformly almost super-additive} sequences of functions.
 	\end{defn}
	The following lemma is an analogue of Dini's theorem in our setting.
	\begin{lem}\label{lem_dini}
		Assume that the sequence of non-negative functions $a_k : Y \to \real$, $k \in \nat^*$, defined over a compact space $Y$, is uniformly almost sub-additive.
		Assume also that the pointwise limit $a : Y \to \real$ of $\frac{a_k}{k}$, as $k \to \infty$, which exists by Lemma \ref{lemma_chen}, is a continuous function.
		Then the convergence of $\frac{a_k}{k}$ to $a$ is uniform.
	\end{lem}
	\begin{proof}
		A straightforward modification of the proof of Dini's theorem.
	\end{proof}
	The next lemma is an almost-subadditive analogue of a well-known statement for decreasing sequences of functions.
	\begin{lem}\label{lem_usc}
		Let $Y$ be any topological space.
		Assume $a_k : Y \to \real$, $k \in \nat^*$, is a sequence of positive continuous functions which is uniformly almost sub-additive. Then the pointwise limit $a : Y \to \real$ of $\frac{a_k}{k}$, as $k \to \infty$, which exists by Lemma \ref{lemma_chen}, is upper semicontinuous.
	\end{lem}
	\begin{proof}
		A straightforward modification of the proof for decreasing sequences of functions.
	\end{proof}
	\begin{lem}\label{lem_sub_super_unif}
		Let $Y$ be a compact topological space.
		Assume that $a_k : Y \to \real$, $k \in \nat^*$, is a sequence of continuous functions which is uniformly almost sub-additive and uniformly almost super-additive at the same time.
		Then the sequence $\frac{a_k}{k}$ has a continuous limit $a$, and the convergence of $\frac{a_k}{k}$ to $a$ is uniform.
	\end{lem}
	\begin{proof}
		Uniform almost super-additivity implies that there is $C > 0$ such that $a_k \geq - C k$, see (\ref{eq_bnd_below_an}) and the paragraph after for the details.
		An easy verification shows that the sequence $h_k := a_k + Ck$ then satisfies all the assumptions of Lemma \ref{lem_usc} and hence the limit $a$ is upper semicontinuous.
		We repeat the same argument for the sequence $- a_k$ to obtain that $a$ is lower semicontinuous.
		Hence, $a$ is continuous.
		The result now follows from Lemma \ref{lem_dini}, applied to $h_k$.
	\end{proof}

\section{Normed vector spaces and induced tensor norms}\label{app_norms}

	The main goal of this section is to recall some classical results about normed vector spaces and the induced norms on the tensor products.
	\par 
	The first result is about the identity of two natural norms, defined on the quotient of a normed vector space.
	Let $(U, N_U = \| \cdot \|_U)$ be a normed vector space.
	For any quotient $\pi : U \to Q$, consider the following two metrics.
	First, we have a metric $\| \cdot \|_Q^0$ induced by $\| \cdot \|_U$ and the quotient map as in (\ref{eq_defn_quot_norm}).
	Second, we have a metric $\| \cdot \|_Q^1$, given by the dual of the norm on $Q^*$ induced by the dual metric $\| \cdot \|_U^*$ and the inclusion $\pi^* : Q^* \to U^*$.
	\begin{lem}\label{lem_dual_incl_eq}
		The following identity holds $\| \cdot \|_Q^0 = \| \cdot \|_Q^1$.
	\end{lem}
	\begin{proof}
		A trivial verification.
	\end{proof}
	\par 
	The next result is about decomposable tensors and their injective and projective norms.
	Let $(U, N_U = \| \cdot \|_U)$, $(V, N_V = \| \cdot \|_V)$ be two Banach spaces.
	Recall that the projective tensor product norm $N_U \otimes_{\pi} N_V = \| \cdot \|_{\otimes_{\pi}}$, and the injective tensor product norm $N_U \otimes_{\epsilon} N_V = \| \cdot \|_{\otimes_{\epsilon}}$ were defined in (\ref{eq_defn_proj_norm}) and (\ref{eq_defn_inf_norm}) respectively.
	We will also sometimes consider the completion $(U \hat{\otimes}_{\epsilon} V, N_U \otimes_{\epsilon} N_V)$ (resp. $(U \hat{\otimes}_{\pi} V, N_U \otimes_{\pi} N_V)$) of $(U \otimes V, N_U \otimes_{\epsilon} N_V)$ (resp. $(U \otimes V, N_U \otimes_{\pi} N_V)$).
	\begin{lem}\label{lem_tens_reas}
		For any $x \in U$, $y \in V$, we have $\| x \otimes y \|_{\otimes_{\pi}} = \| x \otimes y \|_{\otimes_{\epsilon}} = \| x \|_U \cdot \| y \|_V$.
		Also for any $\phi \in U^*$, $\psi \in V^*$, we have similar identities for the induced dual norms $\| \phi \otimes \psi \|_{\otimes_{\pi}}^* = \| \phi \otimes \psi \|_{\otimes_{\epsilon}}^* = \| \phi \|_U^* \cdot \| \psi \|_V^*$.
		Moreover, any norm on $U \otimes V$, which satisfies those properties lies in between $N_U \otimes_{\epsilon} N_V$ and $N_U \otimes_{\pi} N_V$.
	\end{lem}
	\begin{proof}
		For the projective (resp. injective) norm, the first part of the statement is proved in \cite[Propositions 2.1 and 2.3]{RyanTensProd} (resp. \cite[Proposition 3.1]{RyanTensProd}).
		The second part of the statement is established in  \cite[Proposition 6.1]{RyanTensProd}.
	\end{proof}
	Now let us recall the following standard duality property in finitely dimensional context, which helps to pass from injective tensor product to the projective one.
	\begin{lem}[{\cite[Theorem 4.21]{RyanTensProd}}]\label{lem_dual_proj_inj}
		Assume that both $U$ and $V$ are finitely dimensional. The following identity between the norms on $U^* \otimes V^*$ holds
		\begin{equation}\label{eq_dual_proj_inj}
			N_U^* \otimes_{\pi} N_V^*
			=
			(N_U \otimes_{\epsilon} N_V)^*.
		\end{equation}
		If, moreover, both $N_U$ and $N_V$ are Hermitian norms, then
		\begin{equation}
			N_U^* \otimes N_V^*
			=
			(N_U \otimes N_V)^*.
		\end{equation}
	\end{lem}
	\begin{rem}
		In \cite[Theorem 4.21]{RyanTensProd}, author even gives a necessary condition so that the above identity holds in the infinite dimensional context, if one takes the Banach closure of each side of (\ref{eq_dual_proj_inj}).
	\end{rem}
	\par 
	The following lemma gives a justification for the names injective and projective tensor norms.
	Let $Q$ be a quotient of $U$. Endow it with the quotient norm $N_Q = \| \cdot \|_Q = [N_U]$ as defined in (\ref{eq_defn_quot_norm}).
	Let $W$ be a closed subspace of $U$.
	Endow it with the induced norm $N_W = \| \cdot \|_W$.
	\begin{lem}\label{lem_quot_proj}
		The norm on $Q \otimes V$ induced by $N_U \otimes_{\pi} N_V$ and the quotient map $U \hat{\otimes}_{\pi} V \to Q \otimes V$ (or $U \otimes_{\pi} V \to Q \otimes V$) coincides with the norm $N_Q \otimes_{\pi} N_V$.
		Similarly, the norm on $W \otimes V$ induced by $N_U \otimes_{\epsilon} N_V$ and the induced map $W \otimes V \to U \hat{\otimes}_{\epsilon} V$ (or $W \otimes V \to U \otimes_{\epsilon} V$) coincides with the norm $N_W \otimes_{\epsilon} N_V$.
		When both $N_U$ and $N_V$ come from Hermitian tensor products, the analogous statements hold for the Hermitian tensor product.
	\end{lem}
	\begin{proof}
		The statement about the projective (resp. injective) tensor norm is established in  \cite[Proposition 2.5]{RyanTensProd} (resp. \cite[Proposition 3.2 and the discussion after]{RyanTensProd}).
		For Hermitian norms, the proof is direct.
	\end{proof}
	Let $(V_i, N_i = \| \cdot \|_i)$, $i = 1, 2, 3$, be finitely dimensional normed vector spaces.
	Our next result concerns the associativity property for injective and projective tensor norms. 
	\begin{lem}\label{lem_assoc}
		On the vector space $V_1 \otimes V_2 \otimes V_3$ we have the following identity between the norms
		\begin{equation}\label{eq_assoc}
			(N_1 \otimes_{\epsilon} N_2) \otimes_{\epsilon} N_3
			=
			N_1 \otimes_{\epsilon} (N_2 \otimes_{\epsilon} N_3),
			\quad
			(N_1 \otimes_{\pi} N_2) \otimes_{\pi} N_3
			=
			N_1 \otimes_{\pi} (N_2 \otimes_{\pi} N_3).
		\end{equation}
		In particular, the notations $N_1 \otimes_{\epsilon} N_2 \otimes_{\epsilon} N_3$, $N_1 \otimes_{\pi} N_2 \otimes_{\pi} N_3$ are well-defined.
	\end{lem}
	\begin{proof}
		Follows directly from the definitions.
	\end{proof}
	Let us establish the following transitivity property for the quotients of projective tensor norm.
	Let $Q_3$ be a quotient of $V_1 \otimes V_2$ and $Q$ be a quotient of $V_1 \otimes V_2 \otimes V_3$, which factorizes through the natural map $V_1 \otimes V_2 \otimes V_3 \to Q_3 \otimes V_3$.
	\begin{cor}\label{cor_quot_ass}
		On the vector space $Q$, we have the following identity between the norms
		\begin{equation}\label{eq_quot_ass}
			\big[ [N_1 \otimes_{\pi} N_2] \otimes_{\pi} N_3 \big]
			=
			\big[ N_1 \otimes_{\pi} N_2 \otimes_{\pi} N_3 \big],
		\end{equation}
		where we used the notation for the quotient norm introduced in (\ref{eq_defn_quot_norm}).
		Moreover, if $N_1$, $N_2$, $N_3$ come from Hermitian products, then the analogous property holds for the Hermitian tensor product
		\begin{equation}\label{eq_quot_ass2}
			\big[ [N_1 \otimes N_2] \otimes N_3 \big]
			=
			\big[ N_1 \otimes N_2 \otimes N_3 \big].
		\end{equation}
	\end{cor}
	\begin{proof}
		It follows trivially from Lemmas \ref{lem_quot_proj} and \ref{lem_assoc}.
	\end{proof}
	\par 
	On several occasions we need to compare injective and projective tensor norms.
	Let $(V_i, N_i = \| \cdot \|_i)$, $i = 1, \ldots, N$, $N \in \nat^*$ be finitely dimensional normed vector spaces.
	\begin{lem}\label{lem_inj_proj_bnd_nntr}
		The following inequality between the norms on $V_1 \otimes \cdots \otimes V_N$ holds
		\begin{equation}\label{eq_bnd_proj_inf}
			 N_1 \otimes_{\pi} \cdots \otimes_{\pi} N_N
			 \leq
			 N_1 \otimes_{\epsilon} \cdots \otimes_{\epsilon} N_N
			 \cdot
			 \frac{\dim N_1 \cdot \ldots \cdot \dim N_N}{\max \{ \dim N_1, \cdots, \dim N_N \}}.
		\end{equation}
	\end{lem}
	\begin{proof}
		It is classical, cf. \cite[Theorem 21]{LamiPalaWint}, that (\ref{eq_bnd_proj_inf}) holds for $N = 2$.
		The general result now follows from this special case by induction and Lemma \ref{lem_assoc}.
	\end{proof}
	\par 
	We will finally calculate explicitly some examples of injective and projective tensor norms.
	In what follows, we consider vector space $\real^n$, endowed with the $L^1$, $L^2$ and $L^{\infty}$-norms, which we denote below in a slightly non-standard manner as $l^1_n$, $l^2_n$, $l^{\infty}_n$.
	We fix also two compact topological spaces  $K_1, K_2$ and denote by $(\ccal^0(K_i), \| \cdot \|_{L^{\infty}_i})$, $i = 1, 2$, the associated spaces of continuous functions on $K_i$, endowed with the supremum norm.
	We consider the analogous functional space on the product space $(\ccal^0(K_1 \times K_2), \| \cdot \|_{L^{\infty}_{1 \times 2}})$.
	Let us also fix two measured spaces $(X_i, \mu_i)$, $i = 1, 2$, and consider the associated spaces of $L^1$ and $L^2$-functions $(L^q(X_i), \| \cdot \|_{L^q(\mu_i)})$, $q = 1, 2$.
	We consider the product measure $\mu_1 \times \mu_2$ on $X_1 \times X_2$ and the associated spaces of $L^1$ and $L^2$-functions, $(L^q(X_1 \times X_2), \| \cdot \|_{L^q(\mu_1 \times \mu_2)})$, $q = 1, 2$.
	\begin{lem}\label{lem_inj_proj_expl}
		The following identities between Banach spaces hold
		\begin{equation}\label{eq_inj_proj_expl11}
		\begin{aligned}
			&
			\big( L^1(X_1), \| \cdot \|_{L^1(\mu_1)} \big) \hat{\otimes}_{\pi} \big( L^1(X_2), \| \cdot \|_{L^1(\mu_2)} \big) &&= \big( L^1(X_1 \times X_2), \| \cdot \|_{L^1(\mu_1 \times \mu_2)} \big),
			\\
			&
			\big( L^2(X_1), \| \cdot \|_{L^2(\mu_1)} \big) \otimes \big( L^2(X_2), \| \cdot \|_{L^2(\mu_2)} \big) &&= \big( L^2(X_1 \times X_2), \| \cdot \|_{L^2(\mu_1 \times \mu_2)} \big),
			\\
			&
			\big( \ccal^0(K_1), \| \cdot \|_{L^{\infty}_1} \big) \hat{\otimes}_{\epsilon} \big( \ccal^0(K_2), \| \cdot \|_{L^{\infty}_2} \big) &&= \big( \ccal^0(K_1 \times K_2), \| \cdot \|_{L^{\infty}_{1 \times 2}} \big),
		\end{aligned}
		\end{equation}
		where the tensor product in the second line means the tensor product of Hilbert spaces.
		In particular, under the natural isomorphism $\real^n \otimes \real^m \to \real^{n \cdot m}$, for norms on $\real^n \otimes \real^m$, we have
		\begin{equation}\label{eq_inj_proj_expl22}
			l^1_n \otimes_{\pi} l^1_m = l^1_{n \cdot m},
			\qquad
			l^2_n \otimes l^2_m = l^2_{n \cdot m},
			\qquad
			l^{\infty}_n \otimes_{\epsilon} l^{\infty}_m = l^{\infty}_{n \cdot m}.
		\end{equation}
	\end{lem}
	\begin{proof}
		Remark first that the identities (\ref{eq_inj_proj_expl22}) follow directly from (\ref{eq_inj_proj_expl11}) by applying the result for discrete spaces.
		Now, the first (resp. last) identity from (\ref{eq_inj_proj_expl11}) is from Ryan \cite[Exercise 2.8]{RyanTensProd} (resp. Ryan \cite[\S 3.2, p.50]{RyanTensProd}).
		The middle identity from (\ref{eq_inj_proj_expl11}) is standard.
	\end{proof}

\bibliography{bibliography}

		\bibliographystyle{abbrv}

\Addresses

\end{document}